\pdfoutput=1
\documentclass[hidelinks, reqno]{amsart}
\usepackage{graphicx}
\usepackage{subcaption}
\usepackage{tabularx}
\usepackage{booktabs}
\usepackage{array}
\usepackage{amsmath}
\usepackage{amsfonts}
\usepackage{amssymb}
\usepackage{amsthm}
\usepackage{thmtools}
\usepackage[scr]{rsfso}
\usepackage{enumitem}
\usepackage{permute}
\usepackage{slashed}
\usepackage[usenames,dvipsnames]{xcolor}
\usepackage[pagebackref = true, colorlinks, linkcolor = Red, citecolor = Green, bookmarksdepth=2, linktocpage=true, hypertexnames=false]{hyperref}
\usepackage[capitalize]{cleveref}
\usepackage{comment}

\usepackage[T1]{fontenc}
\usepackage{makecell}

\usepackage{tikz}
\usetikzlibrary{calc}

\allowdisplaybreaks

\DeclareMathOperator{\interior}{int}
\DeclareMathOperator{\exterior}{ext}

\DeclareMathOperator{\Id}{Id}
\DeclareMathOperator{\tr}{tr}
\DeclareMathOperator{\Stab}{Stab}

\DeclareMathOperator{\supp}{supp}

\DeclareMathOperator{\SL}{SL}
\DeclareMathOperator{\GL}{GL}
\DeclareMathOperator{\SO}{SO}

\DeclareMathOperator{\PSL}{PSL}

\DeclareMathOperator{\Lip}{Lip}

\DeclareMathOperator{\proj}{proj}
\DeclareMathOperator{\Cay}{Cay}

\DeclareMathOperator{\T}{T}

\DeclareMathOperator{\diam}{diam}

\DeclareMathOperator{\len}{len}
\let\Pr\relax
\DeclareMathOperator{\Pr}{Pr}

\DeclareMathOperator{\Span}{span}

\DeclareMathOperator{\Hull}{Hull}

\DeclareFontFamily{U}{mathb}{\hyphenchar\font45}
\DeclareFontShape{U}{mathb}{m}{n}{
      <5> <6> <7> <8> <9> <10> gen * mathb
      <10.95> mathb10 <12> <14.4> <17.28> <20.74> <24.88> mathb12
      }{}
\DeclareSymbolFont{mathb}{U}{mathb}{m}{n}
\DeclareMathSymbol{\bigast}{2}{mathb}{"06}

\def\XXint#1#2#3{{\setbox0=\hbox{$#1{#2#3}{\int}$}
     \vcenter{\hbox{$#2#3$}}\kern-.5\wd0}}

\theoremstyle{plain}
\newtheorem{theorem}{Theorem}[section]
\newtheorem{proposition}[theorem]{Proposition}
\newtheorem{lemma}[theorem]{Lemma}
\newtheorem{corollary}[theorem]{Corollary}
\newtheorem{conjecture}[theorem]{Conjecture}

\theoremstyle{definition}
\newtheorem{definition}[theorem]{Definition}

\newtheoremstyle{remark}
{}   
{}   
{\normalfont}  
{}       
{\itshape} 
{.}         
{5pt plus 1pt minus 1pt} 
{}          

\theoremstyle{remark}
\newtheorem*{remark}{Remark}

\setlist[enumerate,1]{ref=(\arabic*)}
\setlist[enumerate,2]{ref=(\theenumi)(\alph*)}
\setlist[enumerate,3]{ref=(\theenumi)(\theenumii)(\roman*)}
\setlist[enumerate,4]{ref=(\theenumi)(\theenumii)(\theenumiii)(\Alph*)}

\newlist{alternative}{enumerate}{4}     
\setlist[alternative,1]{label=(\arabic*), ref=(\arabic*)}
\setlist[alternative,2]{label=(\alph*), ref=(\thealternativei)(\alph*)}
\setlist[alternative,3]{label=(\roman*), ref=(\thealternativei)(\thealternativeii)(\roman*)}
\setlist[alternative,4]{label=(\Alph*), ref=(\thealternativei)(\thealternativeii)(\thealternativeiii)(\Alph*)}

\Crefname{enumi}{Property}{Properties}
\Crefname{alternativei}{Alternative}{Alternatives}
\Crefname{subsection}{Subsection}{Subsections}

\begin{document}


\title[Congruence Counting in Schottky Semigroups of $\SO(n, 1)$]{Congruence Counting in Schottky and Continued Fractions Semigroups of $\SO(n, 1)$}
\author{Pratyush Sarkar}
\address{Department of Mathematics, UC San Diego, La Jolla, CA 92093, USA}
\email{psarkar@ucsd.edu}
\date{\today}

\begin{abstract}
In this paper, the two settings we are concerned with are $\Gamma < \SO(n, 1)$ a Zariski dense Schottky semigroup and $\Gamma < \SL_2(\mathbb C)$ a Zariski dense continued fractions semigroup. In both settings, we prove a uniform asymptotic counting formula for the associated congruence subsemigroups, generalizing the work of Magee--Oh--Winter \cite{MOW19} in $\SL_2(\mathbb R)$ to higher dimensions. Superficially, the proof requires two separate strategies: the expander machinery of Golsefidy--Varj\'{u}, based on the work of Bourgain--Gamburd--Sarnak, and Dolgopyat's method. However, there are several challenges in higher dimensions. Firstly, using the expander machinery requires a key input: the Zariski density and full trace field property of the return trajectory subgroups, newly introduced in \cite{Sar22}. Secondly, we need to adapt Stoyanov's version of Dolgopyat's method to circumvent some technical issues while the main difficulty is to prove the key inputs: the local non-integrability condition (LNIC) and the non-concentration property (NCP).
\end{abstract}

\maketitle

\setcounter{tocdepth}{1}
\tableofcontents

\section{Introduction}
This paper is concerned with counting results uniformly over congruence subsemigroups. The main theorem of this paper are for two different settings called the Schottky semigroup setting and the continued fractions semigroup setting. We briefly introduce the necessary background below.

We first introduce the Schottky semigroup setting. Let $\mathbf{G} = \SO_Q < \GL_{n + 1}$ be the algebraic group defined over $\mathbb Q$ consisting of elements which preserve the quadratic form $Q = x_1^2 + x_2^2 + \dotsb + x_n^2 - x_{n + 1}^2$ for $n \geq 2$. Let $G = \mathbf{G}(\mathbb R)^\circ$ which we recognize as the group of orientation preserving isometries of $\mathbb H^n$. Let $N_0 \in \mathbb N$, $0 \leq N_1 \leq N_0$ be an integer, and $D_1, D_2, \dotsc, D_{2N_0} \subset \mathbb R^{n - 1} \cup \{\infty\}$ be mutually disjoint closed Euclidean balls. Let $\Gamma < G$ be a Zariski dense semigroup generated by a finite subset of hyperbolic elements $\{g_1, g_2, \dotsc, g_{N_0 + N_1}\} \subset G$ such that
\begin{enumerate}
\item $g_j(\exterior(D_{j + N_0})) = \interior(D_j)$ for all $1 \leq j \leq N_0$;
\item $g_{N_0 + j} = {g_j}^{-1}$ for all $1 \leq j \leq N_1$.
\end{enumerate}
Such a semigroup is called a \emph{Schottky semigroup}. For the strong approximation theorem of Weisfeiler \cite{Wei84} to hold, we use the simply connected cover $\tilde{\pi}: \tilde{\mathbf G} \to \mathbf G$ defined over $\mathbb Q$. Since $\tilde{\mathbf{G}}$ is also defined over $\mathbb Q$, we can consider the set of integral points $\tilde{\mathbf G}(\mathbb Z)$. We assume $\Gamma$ was chosen such that there is a corresponding Zariski dense semigroup $\tilde{\Gamma} < \tilde{\mathbf G}(\mathbb Z)$ generated by $\{\tilde{g}_1, \tilde{g}_2, \dotsc, \tilde{g}_{N_0 + N_1}\} \subset \tilde{\mathbf G}(\mathbb Z)$ where $\tilde{g}_j \in \tilde{\mathbf G}(\mathbb Z)$ is a unique choice of a lift of $g_j$ for all $1 \leq j \leq N_0 + N_1$. For all $\gamma = g_{\alpha_1}g_{\alpha_2} \dotsb g_{\alpha_j} \in \Gamma$ for some sequence $(\alpha_1, \alpha_2, \dotsc, \alpha_j)$ and $j \in \mathbb N$, we denote $\tilde{\gamma} = \tilde{g}_{\alpha_1}\tilde{g}_{\alpha_2} \dotsb \tilde{g}_{\alpha_j} \in \tilde{\Gamma}$. For uniform notations, we denote $\mathcal{O} = \mathbb Z$ so that $\tilde{\mathbf G}(\mathcal{O}) = \tilde{\mathbf G}(\mathbb Z)$.

We now introduce the continued fractions semigroup setting. Let $\mathbf{G} = \SL_2$ as an algebraic group defined over $\mathbb Q$. Let $n = 2$ (resp. $n = 3$) and $G = \mathbf{G}(\mathbb R)$ (resp. $G = \mathbf{G}(\mathbb C)$) which we recognize as the double cover of the group of orientation preserving isometries of $\mathbb H^n$. Let $\Gamma < G$ be a Zariski dense semigroup generated by
\begin{align*}
\left\{
\begin{pmatrix}
0 & 1 \\
1 & a
\end{pmatrix}
\begin{pmatrix}
0 & 1 \\
1 & a'
\end{pmatrix}
: a, a' \in \mathscr{A}\right\} \subset \SL_2(\mathbb Z[i])
\end{align*}
for some finite subset $\mathscr{A} \subset \mathbb N$ (resp. $\mathscr{A} \subset \mathbb N + i\mathbb Z$). Such a semigroup is called a \emph{continued fractions semigroup}. For uniform notations, we also denote $\tilde{\mathbf{G}} = \mathbf{G}$, $\mathcal{O} = \mathbb Z$ (resp. $\mathcal{O} = \mathbb Z[i]$), $\tilde{\Gamma} = \Gamma$, and $\tilde{\gamma} = \gamma$ for all $\gamma \in \Gamma$.

Let $q \in \mathcal{O}$. Let $\pi_q: \tilde{\mathbf{G}}(\mathcal{O}) \to \tilde{\mathbf{G}}_q$ be the reduction map where $\tilde{\mathbf{G}}_q = \tilde{\mathbf{G}}(\mathcal{O}/q\mathcal{O})$. Let $\Gamma_q$ be the congruence subsemigroup of level $q$, i.e., $\Gamma_q = \tilde{\pi}(\tilde{\Gamma}_q)$ where $\tilde{\Gamma}_q = \ker(\pi_q) \cap \tilde{\Gamma} < \tilde{\Gamma}$. Denote by $L(\mathbb H^n \cup \mathbb R^{n - 1}, \mathbb R)$ the space of real-valued Lipschitz functions with respect to the Euclidean metric on $\mathbb H^n \cup \mathbb R^{n - 1} \subset \mathbb R^n$ where we view $\mathbb H^n$ in the upper half space model. The limit set $\Lambda(\Gamma) \subset \mathbb R^{n - 1} \cup \{\infty\}$ of $\Gamma$ is the set of limit points of any orbit $\Gamma o \subset \overline{\mathbb H^n}$. In fact, we can assume that $\Lambda(\Gamma) \subset \mathbb R^{n - 1}$ (see \cref{sec:Preliminaries}). Let $\delta_\Gamma > 0$ be the Hausdorff dimension of $\Lambda(\Gamma)$. Define $L^\star(\mathbb H^n \cup \mathbb R^{n - 1}, \mathbb R) \subset L(\mathbb H^n \cup \mathbb R^{n - 1}, \mathbb R)$ to be the subspace of bounded functions which are locally constant on some neighborhood of $\Lambda(\Gamma)$. Let $\|\cdot\|$ denote the Frobenius norm on the space of $(n + 1) \times (n + 1)$ matrices, i.e., the Euclidean norm after viewing the space as $\mathbb R^{(n + 1)^2}$. We also define the norm of $q \in \mathcal{O}$ by
\begin{align*}
	N(q) =
	\begin{cases}
		|q|, & \mathcal{O} = \mathbb Z, \\
		|q|^2, & \mathcal{O} = \mathbb Z[i].
	\end{cases}
\end{align*}
The following is the main theorem of this paper.

\begin{theorem}
\label{thm:MainTheorem}
There exist $\epsilon \in (0, \delta_\Gamma)$, $C > 0$, and a nonzero $q_0 \in \mathcal{O}$ such that for all $F \in L^\star(\mathbb H^n \cup \mathbb R^{n - 1},  \mathbb R)$ and $\gamma_0 \in \Gamma$, there exists $C_0 > 0$ such that for all $x \in \tilde{\Gamma}$, square-free $q \in \mathcal{O}$ coprime to $q_0$, we have as $R \to +\infty$
\begin{align*}
\sum_{\frac{\|\gamma \gamma_0\|}{\|\gamma_0\|} \leq R, \pi_q(\tilde{\gamma}) = \pi_q(x)} F(\gamma \gamma_0 o) = C_0 \frac{R^{2\delta_\Gamma}}{\#\tilde{\mathbf{G}}_q} + O\bigl(N(q)^C R^{2(\delta_\Gamma - \epsilon)} (\|F\|_\infty + \Lip(F|_{\mathbb R^{n - 1}}))\bigr).
\end{align*}
\end{theorem}

Taking $F = \chi_{\mathbb H^n \cup \mathbb R^{n - 1}} \in L^\star(\mathbb H^n \cup \mathbb R^{n - 1},  \mathbb R)$, $\gamma_0 = e \in \Gamma$, and $x = e \in \tilde{\Gamma}$ in the above theorem, we obtain the following main corollary. We denote by $B_R(e) \subset G$ the ball of radius $R > 0$ centered at $e \in G$ with respect to the Frobenius norm.

\begin{corollary}
\label{cor:MainCorollary}
There exist $\epsilon \in (0, \delta_\Gamma)$, $C_0 > 0$, $C > 0$, and $q_0 \in \mathcal{O}$ such that for all square-free $q \in \mathcal{O}$ coprime to $q_0$, we have
\begin{align*}
\#(\Gamma_q \cap B_R(e)) = C_0 \frac{R^{2\delta_\Gamma}}{\#\tilde{\mathbf{G}}_q} + O\bigl(N(q)^C R^{2(\delta_\Gamma - \epsilon)}\bigr) \qquad \text{as $R \to +\infty$}.
\end{align*}
\end{corollary}

\begin{remark}
We mention some special cases and related works.
\begin{enumerate}
\item \label{itm:MOW19} When $n = 2$, \cref{thm:MainTheorem} was proved by Magee--Oh--Winter \cite{MOW19} both for the Schottky semigroup and the continued fractions semigroup settings (when $F$ is of class $C^1$).
\item In the Schottky \emph{subgroup} setting for \emph{arbitrary} $n \geq 2$, \cref{cor:MainCorollary} follows from uniform exponential mixing of the frame flow established by Sarkar \cite{Sar22} via the work of Mohammadi--Oh \cite{MO15}. \Cref{cor:MainCorollary} also follows from the work of Edwards--Oh \cite{EO21} when $\delta_\Gamma > \frac{n - 1}{2}$ and of Magee \cite{Mag15} when $\delta_\Gamma > s_n^0 = (n - 1) - \frac{2(n - 2)}{n(n + 1)}$.
\item Similarly, in the Schottky \emph{subgroup} setting for $n = 2$, \cref{cor:MainCorollary} follows from uniform exponential mixing of the geodesic flow established by Oh--Winter \cite{OW16}. \Cref{cor:MainCorollary} also follows from the work of Gamburd \cite{Gam02} when $\delta_\Gamma > \frac{5}{6}$ and of Bourgain--Gamburd--Sarnak \cite{BGS11} when $\delta_\Gamma > \frac{1}{2}$.
\item In the Ph.D. thesis of McKeon \cite{McK18}, similar techniques of the combination of transfer operators with the expander machinery and Lalley's techniques of the renewal equation are used (but not Dolgopyat's method) to obtain a different counting result regarding low-lying geodesics in $\SL_2(\mathbb Z[i]) \backslash \mathbb H^3$.
\end{enumerate}
\end{remark}

\begin{remark}
In Case~\ref{itm:MOW19} above, Magee--Oh--Winter do not require the square-free hypothesis since they can use the expander machinery of Bourgain--Varj\'{u} \cite{BV12}. For similar reasons, the square-free hypothesis in \cref{thm:MainTheorem} comes from \cref{lem:GV_Expander} which uses the expander machinery of Golsefidy--Varj\'{u} \cite[Theorem 1]{GV12}. Recently however, He--de Saxc\'{e} \cite[Theorem 6.1]{HdS22} extended the ideas in \cite{BV12} to remove the square-free hypothesis for absolutely simple algebraic groups. Thus, the square-free hypothesis in \cref{thm:MainTheorem} can be \emph{removed} for $n \neq 3$.
\end{remark}

We will now give some context for the continued fractions semigroup setting which arose from the study of Zaremba's conjecture. We denote finite and infinite continued fractions by
\begin{align*}
[a_1, a_2, \dotsc, a_k] &= \cfrac{1}{a_1 + \cfrac{1}{a_2 + \cfrac{1}{\ddots + \cfrac{1}{a_k}}}}; &
[a_1, a_2, \dotsc] &= \cfrac{1}{a_1 + \cfrac{1}{a_2 + \cfrac{1}{\ddots}}}
\end{align*}
respectively for any sequence $\{a_j\}_{j \in \mathbb N} \subset \mathbb N + i\mathbb Z$. Let $\mathscr{A} \subset \mathbb N + i\mathbb Z$ be some finite subset and $\mathcal{G}_\mathscr{A} \subset \GL_2(\mathbb Z[i])$ be the semigroup generated by
\begin{align*}
\left\{
\begin{pmatrix}
0 & 1 \\
1 & a
\end{pmatrix}
: a \in \mathscr{A}\right\}
\subset \GL_2(\mathbb Z[i]).
\end{align*}
Define
\begin{align*}
\mathfrak{R}_\mathscr{A} &= \{[a_1, a_2, \dotsc, a_k]: \{a_j\}_{j = 1}^k \subset \mathscr{A}, k \in \mathbb N\}; \\
\mathfrak{D}_\mathscr{A} &= \left\{d: \frac{b}{d} \in \mathfrak{R}_\mathscr{A} \text{ where $b$ and $d$ are coprime}\right\}.
\end{align*}
Note that $\frac{b}{d} = [a_1, a_2, \dotsc, a_k]$ if and only if
\begin{align*}
\begin{pmatrix}
0 & 1 \\
1 & a_1
\end{pmatrix}
\begin{pmatrix}
0 & 1 \\
1 & a_2
\end{pmatrix}
\cdots
\begin{pmatrix}
0 & 1 \\
1 & a_k
\end{pmatrix}
=
\begin{pmatrix}
* & b \\
* & d
\end{pmatrix}.
\end{align*}
Hence, $\mathfrak{D}_\mathscr{A} = \{\langle \gamma \cdot e_2, e_2\rangle: \gamma \in \mathcal{G}_\mathscr{A}\}$ where $e_2 = (0, 1)$ and $\langle \cdot, \cdot \rangle$ is the standard inner product on $\mathbb{C}^2$. The following is Zaremba's conjecture \cite{Zar72}.

\begin{conjecture}
There exists $A \in \mathbb N$ such that $\mathfrak{D}_{\{1, 2, \dotsc, A\}} = \mathbb N$.
\end{conjecture}

There has been progress towards this conjecture in recent years. Notably, Bourgain--Kontorovich \cite{BK14} proved the following theorem.

\begin{theorem}
\label{thm:BourgainKontorovichZarembaConjecture}
There exists $A \in \mathbb N$ such that
\begin{align*}
\lim_{N \to \infty} \frac{1}{N} \#(\mathfrak{D}_{\{1, 2, \dotsc, A\}} \cap [1, N]) = 1.
\end{align*}
Moreover, $A = 50$ suffices.
\end{theorem}

\Cref{thm:BourgainKontorovichZarembaConjecture} can be restated as
\begin{align*}
\#(\mathfrak{D}_{\{1, 2, \dotsc, A\}} \cap [1, N]) = N + o(N) \qquad \text{for some $A \in \mathbb N$}
\end{align*}
where $A = 50$ suffices. In fact, Bourgain--Kontorovich showed that the error term $o(N)$ can be improved to
\begin{align}
\label{eqn:BK-ImprovedError}
O\bigl(Ne^{-c\sqrt{\log(N)}}\bigr) \qquad \text{for some $c > 0$}.
\end{align}
Huang \cite{Hua15} improved \cref{thm:BourgainKontorovichZarembaConjecture} and its refinements so that $A = 5$ suffices.

Counting results for continued fractions semigroups are related to Zaremba's conjecture because it is required in the techniques of Bourgain--Kontorovich. In particular, it was noticed by Magee--Oh--Winter that their counting result for continued fractions semigroup \cite[Theorem 11]{MOW19} can be used in place of \cite[Theorem 8.1]{BK14} to further improve the error term $o(N)$ to $O(N^{1 - \epsilon})$ for some $\epsilon \in (0, 1)$ (cf. \cite{Bou12}). In light of Huang's improvement, the state of the art towards Zaremba's conjecture is then
\begin{align*}
\#(\mathfrak{D}_{\{1, 2, 3, 4, 5\}} \cap [1, N]) = N + O(N^{1 - \epsilon}) \qquad \text{for some $\epsilon \in (0, 1)$}.
\end{align*}

We thus propose the following generalization of Zaremba's conjecture.

\begin{conjecture}
There exists a bounded subset $\mathscr{A} \subset \mathbb N + i\mathbb Z$ such that $\mathfrak{D}_{\mathscr{A}} = \mathbb N + i\mathbb Z$.
\end{conjecture}

Analogous to Magee--Oh--Winter's improvement for the theorem of Bourgain--Kontorovich, \cref{thm:MainTheorem} for the continued fractions semigroup case may find applications to prove theorems towards the above conjecture.

\subsection{Outline of the proof of \texorpdfstring{\cref{thm:MainTheorem}}{\autoref{thm:MainTheorem}}}
The proof is based on the techniques of transfer operators. We mainly follow \cite{Nau05,Sto11,MOW19,Sar22}.

As in \cite{MOW19}, we use \emph{congruence} transfer operators $\mathcal{M}_{\xi, q}: C(\Lambda, L^2(\tilde{\mathbf{G}}_q)) \to C(\Lambda, L^2(\tilde{\mathbf{G}}_q))$ for $\xi = a + ib \in \mathbb C$ defined by
\begin{align*}
\mathcal{M}_{\xi, q}(H)(u) &= \sum_{u' \in T^{-1}(u)} e^{(\delta_\Gamma + a + ib)\tau(u')} \mathtt{c}_q(u') H(u').
\end{align*}
The uniform counting result \cref{thm:MainTheorem} follows from spectral bounds of the congruence transfer operators in \cref{thm:CongruenceOperatorSpectralBounds} exactly as in \cite[Section 3]{MOW19} using arguments by Bourgain--Gamburd--Sarnak \cite{BGS11}. Proving \cref{thm:CongruenceOperatorSpectralBounds} requires two separate strategies: the expander machinery and Dolgopyat's method. The former technique is required because for small frequencies $|b| \ll 1$, the Ruelle--Perron--Frobenius theorem and perturbation theory are insufficient to obtain spectral bounds \emph{uniformly} in $q \in \mathcal{O}$ as there are countably infinitely many of them. This idea is now well developed in \cite{OW16,MOW19,Sar22}.

We follow \cite{Sar22} to use the expander machinery of Golsefidy--Varj\'{u} \cite{GV12} which is a generalization of the breakthrough work of Bourgain--Gamburd--Sarnak \cite{BGS10}. The argument in \cite{Sar22} generalizes that of \cite[Appendix]{MOW19} by Bourgain--Kontorovich--Magee and requires the newly introduced concept of \emph{return trajectory subgroups}. The crucial step is to prove that they are Zariski dense. In both the Schottky semigroup and the continued fractions semigroup settings, this amounts to proving that the limit set of the corresponding return trajectory subgroups are not contained in any strictly lower dimensional sphere. For the continued fractions semigroup setting, we also need to prove that the return trajectory subgroups have full trace field $\mathbb Q(i)$. We do this by computing traces of various types of elements and using the Zariski density of $\Gamma$.

For Dolgopyat's method, the greatest difficulty in this paper is to prove the required \emph{local non-integrability condition (LNIC)}. This was proved by Naud \cite[Subsection 4.2]{Nau05} for the 2-dimensional case. We follow the general strategy of Naud to prove LNIC in the higher dimensional case; however, the proof requires many new techniques to resolve difficulties in higher dimensions. Moreover, for LNIC to be useful, we also require the non-concentration property (NCP). Another problem lies in the framework of Dolgopyat's method itself. In \cite{Nau05,MOW19}, the techniques involving the triadic partition proposition (see \cite[Proposition 5.6]{Nau05} and \cite[Proposition 26]{MOW19}) does not lend itself to generalizations to the higher dimensional setting. Thus, we take a different approach. A key observation in this paper is that adapting Stoyanov's version \cite{Sto11} of Dolgopyat's method \cite{Dol98} provides a cleaner proof without the triadic partition proposition. This approach also has the advantage that there is no need to prove a Federer/doubling property for the equilibrium state obtained from thermodynamic formalism since we can directly use the property of Gibbs measures instead. After overcoming these difficulties, the proofs for Dolgopyat's method go through \emph{uniformly} in the congruence parameter $q \in \mathcal{O}$ because of unitarity and local constancy of the cocycle which was first observed in \cite{OW16}.

\subsection{Organization of the paper}
We formulate the dynamical system of interest in \cref{sec:Preliminaries} and recall the necessary background from symbolic dynamics and thermodynamic formalism in \cref{sec:CodingTheDynamicalSystem}. In \cref{sec:CocyclesAndCongruenceTransferOperators}, we define cocycles and the congruence transfer operators and state their spectral bounds which is the main technical theorem of the paper. In \cref{sec:ReductionToNewInvariantFunctionsAtLevel_q,sec:ApproximatingTheTransferOperator,sec:ZariskiDensityOfTheReturnTrajectorySubgroups,sec:L2FlatteningLemma} we go through the expander machinery part of the argument. Here we follow \cite{Sar22} with the main difference being \cref{sec:ZariskiDensityOfTheReturnTrajectorySubgroups} where we prove Zariski density of the return trajectory subgroups and also the full trace field property in the continued fractions semigroup setting. In \cref{sec:Dolgopyat'sMethod,sec:LocalNon-IntegrabilityCondition,sec:ConstructionOfDolgopyatOperators,sec:ProofOfDolgopyat'sMethod}, we go through the adaptation of Stoyanov's version of Dolgopyat's method part of the argument where \cref{sec:LocalNon-IntegrabilityCondition} contains the LNIC for the Schottky semigroup and the continued fractions semigroup settings in arbitrary dimensions. Finally, in \cref{sec:ConvertingUniformSpectralBoundsToUniformCounting} we summarize how to convert the uniform spectral bounds to a uniform counting result.

\subsection*{Acknowledgements}
I thank my advisor Hee Oh for suggesting this problem and encouraging me to work on it. I also thank Alex Kontorovich for pointing out the Ph.D. thesis work of his student Katie McKeon and for correcting a previous error when quoting the result of Bourgain--Kontorovich in \cref{eqn:BK-ImprovedError}.

\section{Preliminaries}
\label{sec:Preliminaries}
We will first fix some notations and then introduce the two settings of interest which are treated simultaneously in the rest of the paper.

Let $\mathbb H^n$ be the $n$-dimensional hyperbolic space for $n \geq 2$. We denote by $\langle \cdot, \cdot\rangle$ and $\|\cdot\|$ the inner product and norm respectively on any tangent space of $\mathbb H^n$ induced by the hyperbolic metric. We also denote by $d$ the distance function on $\mathbb H^n$ induced by the hyperbolic metric. Let $\partial_\infty\mathbb H^n$ denote the boundary at infinity and the compactification of $\mathbb H^n$ by $\overline{\mathbb H^n} = \mathbb H^n \cup \partial_\infty\mathbb H^n$. Recall that the action of isometries of $\mathbb{H}^n$ extends to an action on $\overline{\mathbb{H}^n}$. For any hyperbolic isometry $g$ of $\mathbb{H}^n$, we denote by $g^+ \in \partial_\infty \mathbb{H}^n$ (resp. $g^- \in \partial_\infty \mathbb{H}^n$) its attracting (resp. repelling) fixed point. Let $o \in \mathbb H^n$ be a reference point. Let $(e_1, e_2, \dotsc, e_n)$ be the standard basis of $\mathbb R^n$. Recall the upper half space model $\mathbb H^n \cong \{(x_1, x_2, \dotsc, x_n) \in \mathbb R^n: x_n > 0\}$ with boundary at infinity $\partial_\infty\mathbb H^n \cong \{(x_1, x_2, \dotsc, x_n) \in \mathbb R^n: x_n = 0\} \cup \{\infty\} \cong \mathbb R^{n - 1} \cup \{\infty\}$ such that the reference point is $o = e_n$.

\subsection{Schottky semigroup setting}
Let $\mathbf{G} = \SO_Q < \GL_{n + 1}$ be the algebraic group defined over $\mathbb Q$ consisting of elements which preserve the quadratic form $Q = x_1^2 + x_2^2 + \dotsb + x_n^2 - x_{n + 1}^2$ for $n \geq 2$. Let $G = \mathbf{G}(\mathbb R)^\circ$ which we recognize as the group of orientation preserving isometries of $\mathbb H^n$.

\begin{definition}[Schottky semigroup]
\label{def:SchottkySemigroup}
Let $N_0 \in \mathbb N$, $0 \leq N_1 \leq N_0$ be an integer, and $D_1, D_2, \dotsc, D_{2N_0} \subset \mathbb R^{n - 1} \cup \{\infty\}$ be mutually disjoint closed Euclidean balls. We call the finite subset $\{g_1, g_2, \dotsc, g_{N_0 + N_1}\} \subset G$ (resp. $\subset \tilde{G}$) a \emph{Schottky generating set} if it consists of hyperbolic elements such that
\begin{enumerate}
\item $g_j(\exterior(D_{j + N_0})) = \interior(D_j)$ for all $1 \leq j \leq N_0$;
\item $g_{N_0 + j} = {g_j}^{-1}$ for all $1 \leq j \leq N_1$.
\end{enumerate}
The semigroup in $G$ (resp. $\tilde{G}$) generated by the Schottky generating set is called a \emph{Schottky semigroup}.
\end{definition}

\begin{remark}
By the ping-pong lemma, a Schottky semigroup is a free semigroup when $N_1 = 0$ but not otherwise due to inverse relations. Similarly, when $N_1 = N_0$, we obtain the usual Schottky \emph{group} which is a free group.
\end{remark}

\begin{remark}
Precomposing the standard representation of $G$ by a conjugation by an element of $\Stab_G(e_n) = \SO(n)$ if necessary, we can assume that the closed Euclidean balls are contained in $\mathbb R^{n - 1}$. This does not affect the formulas in \cref{sec:ConvertingUniformSpectralBoundsToUniformCounting}.
\end{remark}

For the Schottky semigroup setting, let $\Gamma < G$ be a Zariski dense Schottky semigroup. We fix $N_0, N_1 \in \mathbb N$, $N = N_0 + N_1$, and $D_1, D_2, \dotsc, D_{2N_0} \subset \mathbb R^{n - 1}$, and a Schottky generating set $\{g_1, g_2, \dotsc, g_N\} \subset G$ corresponding to $\Gamma$ as in \cref{def:SchottkySemigroup} henceforth. Note that the Zariski dense condition implies $N_0 \geq 2$. Define $D = \bigcup_{j = 1}^N D_j$ and the related constants $\hat{D} = \max_{j \in \{1, 2, \dotsc, N\}} \diam(D_j)$ and $\check{D} = \min_{j \in \{1, 2, \dotsc, N\}} \diam(D_j)$.

To make use of the strong approximation theorem of Weisfeiler \cite{Wei84} later on, we need to work on the simply connected cover $\tilde{\mathbf G}$ endowed with the covering map $\tilde{\pi}: \tilde{\mathbf G} \to \mathbf G$ defined over $\mathbb Q$. Let $\tilde{G} = \tilde{\mathbf G}(\mathbb R)$ which is connected and projects down to $\tilde{\pi}(\tilde{G}) = G$. Choosing unique lifts $\tilde{g}_j \in \tilde{G}$ of $g_j$ for all $1 \leq j \leq N$, we have a corresponding Schottky generating set $\{\tilde{g}_1, \tilde{g}_2, \dotsc, \tilde{g}_N\} \subset \tilde{G}$. Let $\tilde{\Gamma} < \tilde{G}$ be the corresponding Zariski dense Schottky semigroup and note that $\Gamma = \tilde{\pi}(\tilde{\Gamma})$. For all $\gamma \in \Gamma$, writing $\gamma = g_{\alpha_1}g_{\alpha_2} \dotsb g_{\alpha_j}$ for some sequence $(\alpha_1, \alpha_2, \dotsc, \alpha_j)$ and $j \in \mathbb N$ as a word consisting of the Schottky generators, we denote $\tilde{\gamma} = \tilde{g}_{\alpha_1}\tilde{g}_{\alpha_2} \dotsb \tilde{g}_{\alpha_j} \in \tilde{\Gamma}$ to be the corresponding lift. To be able to discuss the notion of congruence subgroups, let us suppose that $\tilde{\Gamma} < \tilde{\mathbf G}(\mathbb Z)$. Then we take $\Gamma$ introduced previously to be $\Gamma = \tilde{\pi}(\tilde{\Gamma})$. For uniform notations throughout the paper, we denote $\mathcal{O} = \mathbb Z$ so that $\tilde{\mathbf G}(\mathcal{O}) = \tilde{\mathbf G}(\mathbb Z)$ for the Schottky semigroup setting. Fix $q_0 \in \mathcal{O}$ to be the nonzero integer such that both the strong approximation theorem of Weisfeiler \cite{Wei84} and the expander machinery of Golsefidy--Varj\'{u} \cite[Theorem 1]{GV12} hold for the return trajectory subgroups introduced in \cref{sec:ZariskiDensityOfTheReturnTrajectorySubgroups}.

\subsection{Continued fractions semigroup setting}
\label{subsec:ContinuedFractionsSemigroupSetting}
We generalize the setting formulated in \cite{MOW19}. Let $\mathbf{G} = \SL_2$ as an algebraic group defined over $\mathbb Q$. Let $n = 3$ and $G = \mathbf{G}(\mathbb C)$ which we recognize as the double cover of the group of orientation preserving isometries of $\mathbb H^3$. Define
\begin{align*}
g_a &=
\begin{pmatrix}
0 & 1 \\
1 & a
\end{pmatrix}
\qquad
\text{for all $a \in \mathbb C$};
&
g_{a, a'} &= g_ag_{a'}
\qquad
\text{for all $a, a' \in \mathbb C$}.
\end{align*}

\begin{definition}[Continued fractions semigroup]
Let $\mathscr{A} \subset \mathbb N + i\mathbb Z$ be a finite subset. The semigroup generated by $\{g_{a, a'}: a, a' \in \mathscr{A}\} \subset \SL_2(\mathbb Z[i])$ is called a \emph{continued fractions semigroup}.
\end{definition}

For the continued fractions semigroup setting, let $\Gamma < \SL_2(\mathbb Z[i]) < G$ be a continued fractions semigroup corresponding to some finite subset $\mathscr{A} \subset \mathbb N + i\mathbb Z$. Since the case $\mathscr{A} \subset \mathbb N$ was already treated in \cite{MOW19}, we focus on the case $\mathscr{A} \not\subset \mathbb N$ in this paper. The strong approximation theorem of Weisfeiler \cite{Wei84}, which we use later on, requires two hypotheses to be fulfilled. It is natural to impose the first hypothesis that $\Gamma < G$ be Zariski dense. Note that this implies $\#\mathscr{A} \geq 2$. The second trace field hypothesis $\mathbb Q(\tr(\Gamma)) = \mathbb Q(i)$ follows from $\mathscr{A} \not\subset \mathbb N$.

We introduce the right regions with which to do dynamics. For all $\epsilon \in [0, \sqrt{3} - 1)$, define
\begin{align*}
D^\epsilon &= \{\xi \in \mathbb C: \Re(\xi) \geq \epsilon, |\xi - (2 + \epsilon)^{-1}| \leq (2 + \epsilon)^{-1}\}; \\
D_a^\epsilon &= g_a D^\epsilon \qquad \text{for all $a \in \mathbb C$}.
\end{align*}
We call any of these \emph{trimmed disks} if $\epsilon \in (0, \sqrt{3} - 1)$ and \emph{untrimmed disks} if $\epsilon = 0$. We refer to \cref{fig:TrimmedDisks} for a visualization of the trimmed disks and the lemma below. We need to use trimmed disks rather than untrimmed disks because the $\epsilon > 0$ provided by the lemma below guarantees the hyperbolicity property in \cref{lem:Hyperbolicity} and at the same time ensures that the Markov property in \cref{eqn:MarkovProperty} holds.

\begin{lemma}
\label{lem:GammaPingPongActionOnTrimmedDisk}
There exists $\epsilon \in (0, \sqrt{3} - 1)$ such that $\{D_a^\epsilon: a \in \mathscr{A}\}$ is a set of mutually disjoint trimmed disks which are contained in $\interior(D^\epsilon)$.
\end{lemma}

\begin{proof}
Denote
\begin{align*}
\overline{e} &=
\begin{pmatrix}
0 & 1 \\
1 & 0
\end{pmatrix};
&
n_a^- &=
\begin{pmatrix}
1 & a \\
0 & 1
\end{pmatrix}
\quad
\text{for all $a \in \mathbb C$}.
\end{align*}
From the calculation $\overline{e}g_a = n_a^-$ for all $a \in \mathbb C$, which act by translations, it follows that $\{\overline{e}D_a^\epsilon: a \in \mathscr{A}\}$ and hence also $\{D_a^\epsilon: a \in \mathscr{A}\}$ is a set of mutually disjoint trimmed disks for all $\epsilon \in (0, \sqrt{3} - 1)$.

Since $\mathscr{A} \subset \mathbb N + i\mathbb Z$ is a finite subset, we can fix $\epsilon \in (0, \sqrt{3} - 1)$ such that
\begin{align*}
\overline{e}D_a^0 \subset \{\xi \in \mathbb C: |\xi - (2\epsilon)^{-1}| < (2\epsilon)^{-1}\} = \overline{e}\{\xi \in \mathbb C: \Re(\xi) > \epsilon\}
\end{align*}
for all $a \in \mathscr{A}$. In particular, $\overline{e}D_a^\epsilon \subset \overline{e}\{\xi \in \mathbb C: \Re(\xi) > \epsilon\}$ and we also have
\begin{align*}
\overline{e}D_a^\epsilon &\subset \{\xi \in \mathbb C: \Re(\xi) > 1 + \epsilon/2\} \subset \overline{e}\{\xi \in \mathbb C: |\xi - (2 + \epsilon)^{-1}| < (2 + \epsilon)^{-1}\}
\end{align*}
for all $a \in \mathscr{A}$. The two containments imply $D_a^\epsilon \subset \interior(D^\epsilon)$ for all $a \in \mathscr{A}$.
\end{proof}

\begin{figure}[h]
\begin{center}
\includegraphics[width=0.5\textwidth]{"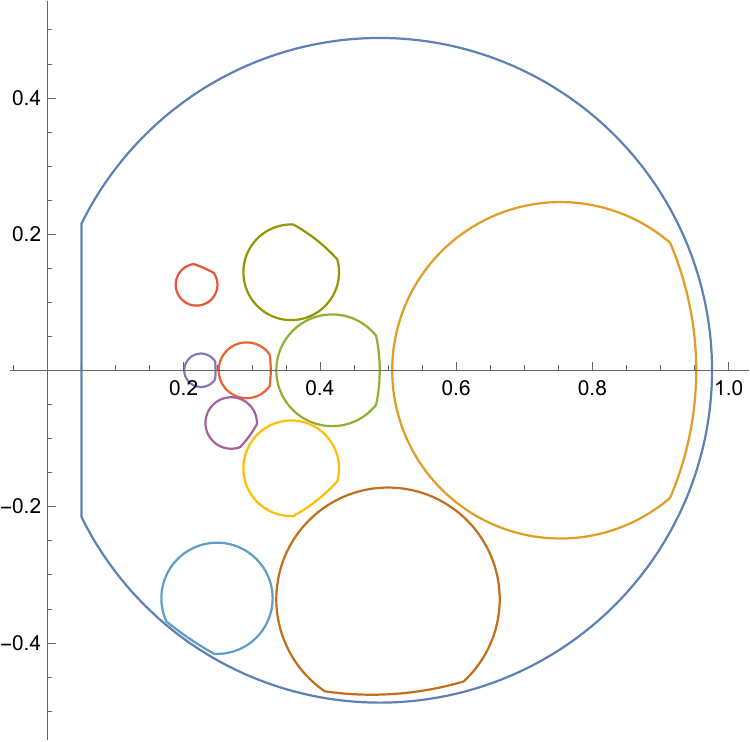"}
\end{center}
\caption{An illustration of an example of a set of mutually disjoint trimmed disks $\{D_a^\epsilon: a \in \mathscr{A}\}$ contained in the interior of the trimmed disk $D^\epsilon$ with $\epsilon = 0.05$. The disks are indicated by their boundaries.}
\label{fig:TrimmedDisks}
\end{figure}

Fix $\epsilon \in (0, \sqrt{3} - 1)$ provided by \cref{lem:GammaPingPongActionOnTrimmedDisk}. Define the trimmed disks $D_{a, a'} = g_aD_{a'}^\epsilon = g_{a, a'}D^\epsilon$ corresponding to the generator $g_{a, a'}$ for all $a, a' \in \mathscr{A}$. Henceforth, we set $N = \# \mathscr{A}^2$ and relabel the trimmed disks and the generators as $D_j$ and $g_j$ for all $1 \leq j \leq N$. Define $D = \bigcup_{j = 1}^N D_j$ and the related constants $\hat{D} = \max_{j \in \{1, 2, \dotsc, N\}} \diam(D_j)$ and $\check{D} = \min_{j \in \{1, 2, \dotsc, N\}} \diam(D_j)$.

\begin{remark}
By \cref{lem:GammaPingPongActionOnTrimmedDisk} and the ping-pong lemma, $\Gamma$ is a free semigroup.
\end{remark}

For uniform notations throughout the paper, we also denote $\tilde{\mathbf{G}} = \mathbf{G}$, $\mathcal{O} = \mathbb Z[i]$ so that $\tilde{\mathbf{G}}(\mathcal{O}) = \SL_2(\mathbb Z[i])$, $\tilde{\Gamma} = \Gamma$, and $\tilde{\gamma} = \gamma$ for all $\gamma \in \Gamma$, for the continued fractions semigroup setting. Fix $q_0 \in \mathcal{O}$ to be the nonzero Gaussian integer such that both the strong approximation theorem of Weisfeiler \cite{Wei84} and the expander machinery of Golsefidy--Varj\'{u} \cite[Corollary 6]{GV12} hold for the return trajectory subgroups introduced in \cref{sec:ZariskiDensityOfTheReturnTrajectorySubgroups}.

\subsection{Limit set}
\begin{definition}[Limit set]
The \emph{limit set} of $\Gamma$ is the set $\Lambda(\Gamma) \subset \partial_\infty\mathbb H^n$ consisting of limit points of the $\Gamma$-orbit $\Gamma o \subset \overline{\mathbb H^n}$, i.e., $\Lambda(\Gamma) = \lim(\Gamma o)$.
\end{definition}

We often denote $\Lambda = \Lambda(\Gamma)$. Note that in both the Schottky semigroup and the continued fractions semigroup settings, we have $\Lambda \subset \mathbb R^{n - 1} \subset \partial_\infty \mathbb H^n$.

\begin{definition}[Critical exponent]
The \emph{critical exponent} $\delta_\Gamma$ of $\Gamma$ is the abscissa of convergence of the Poincar\'{e} series $\mathscr{P}_\Gamma(s) = \sum_{\gamma \in \Gamma} e^{-s d(o, \gamma o)}$.
\end{definition}

\begin{remark}
It is well-known that the above definitions are independent of $o \in \mathbb H^n$. Also, in our case, $\delta_\Gamma \in (0, n - 1)$ and coincides with the Hausdorff dimension of $\Lambda$ (see \cite[Theorem 4]{Bow79} \footnote{Although the proof of \cite[Theorem 4]{Bow79} is for Schottky groups, it works just as well for our free semigroup settings with superficial modifications.}).
\end{remark}

\section{Coding the dynamical system}
\label{sec:CodingTheDynamicalSystem}
In this section, we will introduce the dynamical system and its coding. Then we will review symbolic dynamics and thermodynamic formalism.

\subsection{The dynamical system}
We define $T: D \to \mathbb R^{n - 1}$ by
\begin{align*}
T(u) = g_j^{-1} \cdot u = \tilde{g}_j^{-1} \cdot u \qquad \text{for all $u \in D_j$ and $1 \leq j \leq N$}.
\end{align*}

It easily follows from \cref{def:SchottkySemigroup} and \cref{lem:GammaPingPongActionOnTrimmedDisk} that $T$ satisfies the \emph{Markov property} for both the Schottky semigroup and the continued fractions semigroup settings, i.e.,
\begin{align}
\label{eqn:MarkovProperty}
\text{for all $1 \leq j, k \leq N$}, \qquad \interior(D_j) \cap T(\interior(D_k)) \neq \varnothing \implies D_j \subset T(D_k).
\end{align}

The following \cref{lem:Hyperbolicity} is a fundamental \emph{hyperbolicity property} of our dynamical system which will be used throughout the paper. It can be proved exactly as in \cite[Lemma 3.1]{GLZ04} for the Schottky semigroup setting.

\begin{remark}
In the proof of \cite[Lemma 3.1]{GLZ04}, keeping the same notations from that paper, we can see that $\gamma^{-1}$ is contracting on the compact set $\bigcup_{\gamma_j \neq \gamma(1)^{-1}} \overline{\mathcal{D}_j}$ and hence $\gamma$ is expanding on the compact set $\gamma^{-1}\big(\bigcup_{\gamma_j \neq \gamma(1)^{-1}} \overline{\mathcal{D}_j}\big)$ which is bounded away from $\gamma^{-1}\infty$. So, we also have the upper bound in \cref{lem:Hyperbolicity} by compactness.
\end{remark}

\begin{lemma}
\label{lem:Hyperbolicity}
There exist $c_0 \in (0, 1)$ and $\kappa_1 > \kappa_2 > 1$ such that
\begin{align*}
c_0 \kappa_2^k \leq \|(dT^k)_u\|_{\mathrm{op}} \leq c_0^{-1}\kappa_1^k \qquad \text{for all $u \in T^{-k}(D)$ and $k \in \mathbb N$}.
\end{align*}
\end{lemma}

\begin{proof}
By the above remark, it suffices to prove the lemma only for the continued fractions semigroup setting. We use the notation introduced in the beginning of \cref{subsec:ContinuedFractionsSemigroupSetting}. By chain rule, it suffices to show $\kappa_2 \leq |T'(\xi)| \leq \kappa_1$ for some constants $\kappa_1 > \kappa_2 > 1$, for all $\xi \in T^{-1}(D)$. Fix $\epsilon > 0$ to be the one provided by \cref{lem:GammaPingPongActionOnTrimmedDisk} and $C > \max_{a \in \mathscr{A}} |a|$. Since $|g_a'(\xi)| = |\xi + a|^{-2}$ for all $\xi \in \mathbb C$ and $a \in \mathscr{A}$, we then calculate that
\begin{align*}
(1 + C)^{-2} \leq (|\xi| + |a|)^{-2} \leq |g_a'(\xi)| \leq \Re(\xi + a)^{-2} \leq (1 + \epsilon)^{-2}
\end{align*}
for all $\xi \in D_a^\epsilon$ since $|a| < C$, $\Re(a) \geq 1$, and $D_a^\epsilon \subset D^\epsilon$ by \cref{lem:GammaPingPongActionOnTrimmedDisk} for all $a \in \mathscr{A}$. Recalling definitions, this bound gives
\begin{align*}
1 < (1 + \epsilon)^4 \leq |T'(\xi)| \leq (1 + C)^4 \qquad \text{for all $\xi \in T^{-1}(D)$}
\end{align*}
as desired.
\end{proof}

We fix $c_0$, $\kappa_1$, and $\kappa_2$ provided by \cref{lem:Hyperbolicity} henceforth and use the lemma extensively throughout the paper without further comments.

\subsection{Symbolic dynamics}
\label{subsec:SymbolicDynamics}
Denote by $\mathcal A = \{1, 2, \dotsc, N\}$ the alphabet for the symbolic coding. Define the space of infinite \emph{admissible sequences} by
\begin{align*}
\Sigma &= \{(\alpha_0, \alpha_1, \dotsc) \in \mathcal A^{\mathbb Z_{\geq 0}}: |\alpha_{j + 1} - \alpha_j| \neq N_0 \text{ for all } j \in \mathbb Z_{\geq 0}\}; & \Sigma &= \mathcal A^{\mathbb Z_{\geq 0}}
\end{align*}
for the Schottky semigroup and the continued fractions semigroup settings respectively. We also use the term \emph{admissible sequences} for finite sequences.

\begin{definition}[Cylinder]
For all $k \in \mathbb Z_{\geq 0}$, for all admissible sequences $\alpha = (\alpha_0, \alpha_1, \dotsc, \alpha_k)$, we define the corresponding \emph{cylinders}
\begin{align*}
\mathtt{C}_D[\alpha] &= \{x \in D: T^j(x) \in D_{\alpha_j} \text{ for all } 0 \leq j \leq k\}; & \mathtt{C}[\alpha] &= \Lambda \cap \mathtt{C}_D[\alpha]
\end{align*}
with \emph{length} $\len(\mathtt{C}_D[\alpha]) = \len(\mathtt{C}[\alpha]) = \len(\alpha) = k$. We also denote by $\len(\gamma)$ the word length of $\gamma \in \Gamma$ when it is expressed as a reduced word. We denote cylinders simply by $\mathtt{C}$ (or other typewriter style letters) when we do not specify the sequence.
\end{definition}

We have $\Lambda = \bigcap_{j = 1}^\infty T^{-j}(D) = \bigcup_{\alpha \in \Sigma} \bigcap_{j = 1}^\infty \mathtt{C}_D[\alpha_0, \alpha_1, \dotsc, \alpha_j]$. In fact, we have the homeomorphism $\Lambda \cong \Sigma$.

For all $k \in \mathbb{Z}_{\geq 0}$ and admissible sequences $\alpha = (\alpha_0, \alpha_1, \dotsc, \alpha_k)$, we often use the notation
\begin{align*}
g(\alpha) = \prod_{j = 0}^k g_{\alpha_j} = g_{\alpha_0} \cdot g_{\alpha_1} \dotsb g_{\alpha_k}
\end{align*}
where the product is taken in increasing order.

\subsection{Thermodynamics}
\label{subsec:Thermodynamics}
Define $L(\Lambda, \mathbb R)$ and $C(\Lambda, \mathbb R)$ to be the space of real-valued Lipschitz functions and continuous functions on $\Lambda$ respectively. For complex-valued functions, we simply write $L(\Lambda)$ and $C(\Lambda)$.
\begin{definition}[Pressure]
\label{def:Pressure}
For all \emph{potential} $f \in L(\Lambda, \mathbb R)$, its \emph{pressure} is
\begin{align*}
\Pr_T(f) = \sup_{\nu \in \mathcal{M}^1_T(\Lambda)}\left\{\int_\Lambda f \, d\nu + h_\nu(T)\right\}
\end{align*}
where $\mathcal{M}^1_T(\Lambda)$ is the set of $T$-invariant Borel probability measures on $\Lambda$ and $h_\nu(T)$ is the measure theoretic entropy of $T$ with respect to $\nu$.
\end{definition}

For all $f \in L(\Lambda, \mathbb R)$, there is in fact a unique $T$-invariant Borel probability measure $\nu_f$ on $\Lambda$ which attains the supremum in \cref{def:Pressure} called the \emph{$f$-equilibrium state} \cite[Theorems 2.17 and 2.20]{Bow08}.

\begin{definition}
We define the \emph{distortion function} $\tau: T^{-1}(D) \to \mathbb R$ by
\begin{align*}
\tau(u) &= \log\|d(J \circ T \circ J^{-1})_{J(u)}\|_{\mathrm{op}} \qquad \text{for all $u \in T^{-1}(D)$}
\end{align*}
where $J$ is a conformal map from the upper half space model to the ball model of $\overline{\mathbb H^n}$.
\end{definition}

Corresponding to the distortion function, we denote the $-\delta_\Gamma \tau$-equilibrium state by $\nu_\Lambda = \nu_{-\delta_\Gamma \tau}$. It is known that its pressure is $\Pr_T(-\delta_\Gamma \tau) = 0$ by the Bowen formula \cite{Bow79}. Moreover, $\nu_\Lambda$ is a Gibbs measure \cite{Bow08,PP90}, i.e., there exist $c_1^\Lambda, c_2^\Lambda > 0$ such that
\begin{align}
\label{eqn:PropertyOfGibbsMeasures}
c_1^\Lambda e^{-\delta_\Gamma \tau_k(x)} \leq \nu_\Lambda(\mathtt{C}) \leq c_2^\Lambda e^{-\delta_\Gamma \tau_k(x)}
\end{align}
for all $x \in \mathtt{C}$ and cylinders $\mathtt{C}$ with $\len(\mathtt{C}) = k \in \mathbb Z_{\geq 0}$, where we use the notation introduced in \cref{eqn:BirkhoffSums}.

\section{Congruence transfer operators and their uniform spectral bounds}
\label{sec:CocyclesAndCongruenceTransferOperators}
In this section, we first introduce the congruence setting and define cocycles. Then we define congruence transfer operators and present the main technical theorem about their uniform spectral bounds.

Let $q \in \mathcal{O}$. Denote $\tilde{\mathbf{G}}_q = \tilde{\mathbf{G}}(\mathcal{O}/q\mathcal{O})$. we have the reduction map $\pi_q: \tilde{\mathbf{G}}(\mathcal{O}) \to \tilde{\mathbf{G}}_q$ and we define the principal congruence subgroup of level $q$ to be $\ker(\pi_q)$. We define the congruence subsemigroup of $\tilde{\Gamma}$ of level $q$ to be the subsemigroup $\tilde{\Gamma}_q = \ker(\pi_q) \cap \tilde{\Gamma} < \tilde{\Gamma}$. We define $\Gamma_q = \tilde{\pi}(\tilde{\Gamma}_q)$.

\begin{definition}[Cocycle]
The \emph{cocycle} is a map $\mathtt{c}: D \to \tilde{\Gamma}$ defined by $\mathtt{c}|_{D_j} = \tilde{g}_j$ for all $1 \leq j \leq N$.
\end{definition}

\begin{definition}[Congruence cocycle]
For all $q \in \mathcal{O}$, define the \emph{congruence cocycle} $\mathtt{c}_q: D \to \tilde{\mathbf{G}}_q$ by $\mathtt{c}_q = \pi_q \circ \mathtt{c}$.
\end{definition}

For all $k \in \mathbb N$, we use the notation
\begin{align*}
\mathtt{c}^k(u) = \prod_{j = 1}^k\mathtt{c}(T^{k - j}(u)) = \mathtt{c}(T^{k - 1}(u)) \cdot \mathtt{c}(T^{k - 2}(u)) \dotsb \mathtt{c}(u)
\end{align*}
and $\mathtt{c}^0(u) = e \in \Gamma$, for all $u \in U$. Note that the order in the product is important in the definition above. We use similar notations for the congruence version.

\subsection{Transfer operators}
\label{subsec:TransferOperators}
We will use the notation $\xi = a + ib \in \mathbb C$ for the complex parameter for the transfer operators throughout the paper. Let $\Gamma$ and $\Gamma_q$ act on $L^2(\tilde{\mathbf{G}}_q)$ from the left by the \emph{right} regular representation. We also assume that sums over sequences are always over \emph{admissible sequences}.

\begin{definition}[Congruence transfer operator]
For all $f \in C(\Lambda)$ and nonzero $q \in \mathcal{O}$, the \emph{congruence transfer operator} $\mathcal{M}_{f, q}: C(\Lambda, L^2(\tilde{\mathbf{G}}_q)) \to C(\Lambda, L^2(\tilde{\mathbf{G}}_q))$ is defined by
\begin{align*}
\mathcal{M}_{f, q}(H)(u) = \sum_{u' \in T^{-1}(u)} e^{f(u')} \mathtt{c}_q(u') H(u')
\end{align*}
for all $u \in \Lambda$ and $H \in C(\Lambda, L^2(\tilde{\mathbf{G}}_q))$.
\end{definition}

From \cref{sec:Dolgopyat'sMethod} onward, we prefer to drop the subscript of the cocycle in the above definition since it is not significant there that it takes values in a finite group. For all $f \in C(\Lambda)$, denote $\mathcal{L}_f = \mathcal{M}_{f, 1}$ and simply call it the \emph{transfer operator}. Note that if $f \in L(\Lambda)$, then $\mathcal{L}_f$ preserves the subspace $L(\Lambda) \subset C(\Lambda)$.

The following is the Ruelle--Perron--Frobenius (RPF) theorem with the theory of Gibbs measures \cite{Bow08,PP90}.

\begin{theorem}
\label{thm:RPF}
For all $f \in L(\Lambda, \mathbb R)$, the operator $\mathcal{L}_f: C(\Lambda) \to C(\Lambda)$ and its dual $\mathcal{L}_f^*: C(\Lambda)^* \to C(\Lambda)^*$ has eigenvectors with the following properties. There exist a unique positive function $h \in L(\Lambda, \mathbb R)$ and a unique Borel probability measure $\nu$ on $\Lambda$ such that
\begin{enumerate}
\item	$\mathcal{L}_f(h) = e^{\Pr_T(f)}h$;
\item	$\mathcal{L}_f^*(\nu) = e^{\Pr_T(f)}\nu$;
\item	the eigenvalue $e^{\Pr_T(f)}$ is maximal simple and the rest of the spectrum of $\mathcal{L}_f|_{L(\Lambda)}$ is contained in a disk of radius strictly less than $e^{\Pr_T(f)}$;
\item	$\nu(h) = 1$ and the Borel probability measure $\mu$ defined by $d\mu = h \, d\nu$ is $T$-invariant and coincides with the $f$-equilibrium state $\nu_f$.
\end{enumerate}
\end{theorem}

We simply denote $\mathcal{M}_{\xi, q} = \mathcal{M}_{-(\delta_\Gamma + \xi)\tau, q}$ and $\mathcal{L}_\xi = \mathcal{L}_{-(\delta_\Gamma + \xi)\tau}$ for all $\xi \in \mathbb C$ and nonzero $q \in \mathcal{O}$ henceforth. Now we normalize the transfer operators for convenience. Let $a \in \mathbb R$. Define $\lambda_a = e^{\Pr_\sigma(-(\delta_\Gamma + a)\tau)}$ which is the maximal eigenvalue of $\mathcal{L}_a$ and recall that $\lambda_0 = 1$. Define the eigenvectors, the unique positive function $h_a \in L(\Lambda, \mathbb R)$ and the unique probability measure $\nu_a$ on $\Lambda$ with $\nu_a(h_a) = 1$ such that
\begin{align*}
\mathcal{L}_a(h_a) &= \lambda_a h_a; & \mathcal{L}_a^*(\nu_a) &= \lambda_a \nu_a
\end{align*}
provided by \cref{thm:RPF}. Note that $d\nu_\Lambda = h_0 \, d\nu_0$. Moreover, by perturbation theory for operators as in \cite[Chapter 7]{Kat95} and \cite[Proposition 4.6]{PP90}, we can fix $a_0' > 0$ such that the map $[-a_0', a_0'] \to \mathbb R$ defined by $a \mapsto \lambda_a$ and the map $[-a_0', a_0'] \to C(\Lambda, \mathbb R)$ defined by $a \mapsto h_a$ are Lipschitz. For all $a \in \mathbb R$, we define
\begin{align*}
f^{(a)} = -(\delta_\Gamma + a)\tau + \log(h_0) - \log(h_0 \circ T) - \log(\lambda_a).
\end{align*}
We can fix $A_f > 0$ such that $\bigl|f^{(a)} - f^{(0)}\bigr| \leq A_f|a|$ for all $|a| \leq a_0'$. Due to \cite[Lemma 4.1]{PS16}, we can also fix $T_0 > \max\left(\|\tau\|_{\Lip}, \sup_{|a| \leq a_0'} \bigl\|f^{(a)}\bigr\|_{\Lip}\right)$ (see \cref{sec:UniformSpectralBoundsWithHolonomy} for notations). For all $k \in \mathbb N$, we use the notation
\begin{align}
\label{eqn:BirkhoffSums}
f_k^{(a)}(u) &= \sum_{j = 0}^{k - 1} f^{(a)}(T^j(u)); & \tau_k(u) &= \sum_{j = 0}^{k - 1} \tau(T^j(u))
\end{align}
and $f_0^{(a)}(u) = \tau_0(u) = 0$, for all $u \in \Lambda$ and $a \in \mathbb R$. Let $\xi \in \mathbb C$ and $q \in \mathcal{O}$ be nonzero. We define $M_{\xi, q}: C(\Lambda, L^2(\tilde{\mathbf{G}}_q)) \to C(\Lambda, L^2(\tilde{\mathbf{G}}_q))$ by
\begin{align*}
M_{\xi, q}(H)(u) &= \sum_{u' \in T^{-1}(u)} e^{(f^{(a)} - ib\tau)(u')} \mathtt{c}_q(u') H(u')
\end{align*}
for all $u \in \Lambda$ and $H \in C(\Lambda, L^2(\tilde{\mathbf{G}}_q))$. For all $k \in \mathbb N$, its $k$\textsuperscript{th} iteration is
\begin{align*}
M_{\xi, q}^k(H)(u)  = \sum_{u' \in T^{-k}(u)} e^{(f_k^{(a)} - ib\tau_k)(u')} \mathtt{c}_q^k(u') H(u')
\end{align*}
for all $u \in \Lambda$ and $H \in C(\Lambda, L^2(\tilde{\mathbf{G}}_q))$. Denote $L_\xi = M_{\xi, 1}$. Then the transfer operators are normalized such that the maximal eigenvalue of $L_a$ is $1$ with normalized eigenvector $\frac{h_a}{h_0}$ for all $a \in \mathbb R$. Moreover, $L_0^*(\nu_\Lambda) = \nu_\Lambda$.

\subsection{Uniform spectral bounds with holonomy}
\label{sec:UniformSpectralBoundsWithHolonomy}
We first introduce some norms and seminorms. Let $q \in \mathcal{O}$ be nonzero, and $H \in C(\Lambda, L^2(\tilde{\mathbf{G}}_q)) \subset L^2(\Lambda, L^2(\tilde{\mathbf{G}}_q))$. Denote $\|H\| \in C(\Lambda, \mathbb R)$ to be the function defined by $\|H\|(u) = \|H(u)\|_2$ for all $u \in \Lambda$ and $|H| \in C(\Lambda, L^2(\tilde{\mathbf{G}}_q, \mathbb R))$ to be the function defined by $|H|(u) = |H(u)| \in L^2(\tilde{\mathbf{G}}_q, \mathbb R)$ for all $u \in \Lambda$. We use the notations
\begin{align*}
\Lip(H) &= \sup_{\substack{u, u' \in \Lambda\\ \text{such that } u \neq u'}} \frac{\|H(u) - H(u')\|_2}{\|u - u'\|}; & \|H\|_{\Lip} &= \|H\|_\infty + \Lip(H).
\end{align*}
We generalize this to another useful norm denoted by
\begin{align*}
\|H\|_{1, b} = \|H\|_\infty + \frac{1}{\max(1, |b|)} \Lip(H).
\end{align*}
We denote any operator norms simply by $\|\cdot\|_{\mathrm{op}}$.

For all nonzero $q \in \mathcal{O}$, we define the Banach space $\mathcal{V}_q(\Lambda) = L(\Lambda, L^2(\tilde{\mathbf{G}}_q))$ and when $q$ is coprime to $q_0$, we can similarly define the Banach space $\mathcal{W}_q(\Lambda) = L(\Lambda, L_0^2(\tilde{\mathbf{G}}_q)) \subset \mathcal{V}_q(\Lambda)$ where $L_0^2(\tilde{\mathbf{G}}_q) = \left\{\phi \in L^2(\tilde{\mathbf{G}}_q): \sum_{g \in \tilde{\mathbf{G}}_q} \phi(g) = 0\right\}$.

Let $q \in \mathcal{O}$ be nonzero. Define its norm by
\begin{align*}
N(q) =
\begin{cases}
|q|, & \mathcal{O} = \mathbb Z, \\
|q|^2, & \mathcal{O} = \mathbb Z[i].
\end{cases}
\end{align*}
It is well-known that $\#(\mathcal{O}/q\mathcal{O}) = N(q)$. We say that $q$ is \emph{square-free} if it has no square factors. Now we can state \cref{thm:CongruenceOperatorSpectralBounds} which is the main technical theorem regarding uniform spectral bounds of congruence transfer operators.

\begin{theorem}
\label{thm:CongruenceOperatorSpectralBounds}
There exist $\eta > 0$, $C \geq 1$, $a_0 > 0$, $b_0 > 0$, a nonzero $q_0' \in \mathcal{O}$, and for all $\sigma > 0$ there exist $\eta_\sigma > 0$ and $C_\sigma > 0$, such that for all $\xi \in \mathbb C$ with $|a| < a_0$ and $k \in \mathbb N$, we have
\begin{enumerate}
\item\label{itm:CongruenceOperatorSpectralBoundsSmall|b|} if $|b| \leq b_0$, then for all square-free $q \in \mathcal{O}$ coprime to $q_0q_0'$ and $H \in \mathcal{W}_q(\Lambda)$, we have
\begin{align*}
\big\|\mathcal{M}_{\xi, q}^k(H)\big\|_{\Lip} \leq CN(q)^C e^{-\eta k} \|H\|_{\Lip};
\end{align*}
\item\label{itm:CongruenceOperatorSpectralBoundsLarge|b|} if $|b| > b_0$, then for all nonzero $q \in \mathcal{O}$ and $H \in \mathcal{V}_q(\Lambda)$, we have
\begin{align*}
\big\|\mathcal{M}_{\xi, q}^k(H)\big\|_{\Lip} \leq C_\sigma|b|^{1 + \sigma} e^{-\eta_\sigma k} \|H\|_{\Lip}.
\end{align*}
\end{enumerate}
\end{theorem}

Note that in \cref{thm:CongruenceOperatorSpectralBounds}, \cref{itm:CongruenceOperatorSpectralBoundsSmall|b|} has conditions on $q \in \mathcal{O}$ whereas \cref{itm:CongruenceOperatorSpectralBoundsLarge|b|} does not. This stems from the fact that the expander machinery, where the same condition appears, is required for small frequencies $|b| \leq b_0$ and Dolgopyat's method, which is independent of nonzero $q \in \mathcal{O}$, is used for large frequencies $|b| > b_0$.

We will first prove \cref{itm:CongruenceOperatorSpectralBoundsSmall|b|} in \cref{thm:CongruenceOperatorSpectralBounds}. We fix $b_0 > 0$ to be the one from \cref{thm:TheoremLarge|b|} where if we examine the proof of \cref{thm:Dolgopyat}, it is clear from \cref{eqn:Constantb0} that we can assume $b_0 = 1$.

\section{Reduction of \texorpdfstring{\cref{itm:CongruenceOperatorSpectralBoundsSmall|b|}}{Property \ref{itm:CongruenceOperatorSpectralBoundsSmall|b|}} of \texorpdfstring{\cref{thm:CongruenceOperatorSpectralBounds}}{\autoref{thm:CongruenceOperatorSpectralBounds}}}
\label{sec:ReductionToNewInvariantFunctionsAtLevel_q}
In this section, we reduce \cref{itm:CongruenceOperatorSpectralBoundsSmall|b|} of \cref{thm:CongruenceOperatorSpectralBounds} to \cref{thm:ReducedTheoremSmall|b|} as in \cite{OW16}. This is done by introducing the concept of new invariant functions.

Let $q, q' \in \mathcal{O}$ such that $q' \mid q$. We have the reduction map $\pi_{q, q'}: \tilde{\mathbf{G}}_q \to \tilde{\mathbf{G}}_{q'}$ which induces the pull back $\pi_{q, q'}^*: L^2(\tilde{\mathbf{G}}_{q'}) \to L^2(\tilde{\mathbf{G}}_q)$. Define $\hat{E}_{q'}^q = \pi_{q, q'}^*(L^2(\tilde{\mathbf{G}}_{q'}))$ and $E_{q'}^q = \hat{E}_{q'}^q \cap \left(\bigoplus_{q' \neq q'' \mid q'} \hat{E}_{q''}^q\right)^\perp$ which can be viewed as the space of \emph{new} functions which are invariant under $\tilde{\mathbf{G}}_{q'}$ but not invariant under $\tilde{\mathbf{G}}_{q''}$ for any $q' \neq q'' \mid q'$. Then, we have the orthogonal decomposition
\begin{align*}
L_0^2(\tilde{\mathbf{G}}_q) = \bigoplus_{1 \neq q' \mid q} E_{q'}^q \qquad \text{for all $q \in \mathcal{O}$}.
\end{align*}
Again let $q, q' \in \mathcal{O}$ such that $q' \mid q$. We define $\mathcal{W}_{q'}^q(\Lambda) = \{H \in \mathcal{W}_q(\Lambda): H(u) \in E_{q'}^q \text{ for all } u \in \Lambda\}$, so that we have the orthogonal decomposition
\begin{align*}
\mathcal{W}_q(\Lambda) = \bigoplus_{1 \neq q' \mid q} \mathcal{W}_{q'}^q(\Lambda) \qquad \text{for all $q \in \mathcal{O}$}.
\end{align*}
We have the orthogonal projection operator $e_{q, q'}: \mathcal{W}_q(\Lambda) \to \mathcal{W}_{q'}^q(\Lambda)$ and the canonical projection map $\proj_{q, q'} = (\pi_{q, q'}^*)^{-1}\big|_{E_{q'}^q}: E_{q'}^q \to E_{q'}^{q'}$ since $\pi_{q, q'}^*$ is injective and $E_{q'}^q \subset \pi_{q, q'}^*(L^2(\tilde{\mathbf{G}}_{q'}))$. Then $\proj_{q, q'}(\phi)(g) = \phi(\tilde{g})$ where $\tilde{g}$ is any lift of $g$ under $\pi_{q, q'}: \tilde{\mathbf{G}}_q \to \tilde{\mathbf{G}}_{q'}$, for all $g \in \tilde{\mathbf{G}}_{q'}$ and $\phi \in E_{q'}^q$. We use the same notation $\proj_{q, q'}: \mathcal{W}_{q'}^q(\Lambda) \to \mathcal{W}_{q'}^{q'}(\Lambda)$ for the induced projection map defined pointwise. For all $\xi \in \mathbb C$, the space $\mathcal{W}_{q'}^q(\Lambda)$ is preserved by $M_{\xi, q}$ and we have
\begin{align*}
e_{q, q'} \circ M_{\xi, q} &= M_{\xi, q} \circ e_{q, q'} \\
\proj_{q, q'} \circ M_{\xi, q} &= M_{\xi, q'} \circ \proj_{q, q'}.
\end{align*}
By surjectivity of $\pi_{q, q'}$, we can denote $\spadesuit_{q, q'} = \#\ker(\pi_{q, q'}) = \frac{\#\tilde{\mathbf{G}}_q}{\#\tilde{\mathbf{G}}_{q'}}$ and by direct calculation it can be checked that $\|H\|_2 = \sqrt{\spadesuit_{q, q'}} \|\proj_{q, q'}(H)\|_2$ and $\|H\|_{\Lip(d)} = \sqrt{\spadesuit_{q, q'}} \|\proj_{q, q'}(H)\|_{\Lip(d)}$ for all $H \in \mathcal{W}_{q'}^q(\Lambda)$.

By a standard reduction process (see \cite[Section 6]{Sar22} and \cite[Subsection 4.1]{OW16}), it suffices to prove the following instead.

\begin{theorem}
\label{thm:ReducedTheoremSmall|b|}
There exist $C_s > 0$, $a_0 > 0$, $\theta \in (0, 1)$, and $q_1 \in \mathbb N$ such that for all $\xi \in \mathbb C$ with $|a| < a_0$ and $|b| \leq b_0$, square-free $q \in \mathcal{O}$ coprime to $q_0$ with $N(q) > q_1$, there exists an integer $s \in (0, C_s\log(N(q)))$ such that for all $j \in \mathbb Z_{\geq 0}$ and $H \in \mathcal{W}_q^q(\Lambda)$, we have
\begin{align*}
\big\|M_{\xi, q}^{js}(H)\big\|_{\Lip} \leq N(q)^{-j\theta} \|H\|_{\Lip}.
\end{align*}
\end{theorem}

\Cref{sec:ApproximatingTheTransferOperator,sec:ZariskiDensityOfTheReturnTrajectorySubgroups,sec:L2FlatteningLemma} are devoted to obtaining strong bounds which are crucial for the proof of \cref{thm:ReducedTheoremSmall|b|} in \cref{sec:SupremumAndLipschitzBounds}.

\section{Approximating the congruence transfer operators}
\label{sec:ApproximatingTheTransferOperator}
The aim of this section is to approximate the congruence transfer operators by convolution with measures to mimic a random walk.

We introduce some notations. Let $j \in \mathbb N$ and $(\alpha_j, \alpha_{j - 1}, \dotsc, \alpha_1)$ be an admissible sequence. We define $\alpha^j = (\alpha_j, \alpha_{j - 1}, \dotsc, \alpha_1)$. We adopt the convention that a sequence of sequences are to be concatenated. For all $y \in \mathcal A$, denote $\omega(y) \in \Lambda$ to be any point such that $(y, \omega(y))$ is admissible, i.e., $\omega(y) \notin D_z$ with $|y - z| = N_0$. We extend the notation to admissible sequences so that $\omega(\alpha^j) = \omega(\alpha_1) \in \Lambda$. Note that for all $q \in \mathcal{O}$ we have the formula
\begin{align*}
\mathtt{c}_q^j(x) = \mathtt{c}_q(T^{j - 1}(x)) \cdot \mathtt{c}_q(T^{j - 2}(x)) \dotsb \mathtt{c}_q(x) = \pi_q(\tilde{g}_{\alpha_1} \cdot \tilde{g}_{\alpha_2} \dotsb \tilde{g}_{\alpha_j})
\end{align*}
for all $x \in \mathtt{C}[\alpha^j] \subset \Lambda$.

Let $q \in \mathcal{O}$ be coprime to $q_0$. For any complex measure $\mu$ on $\tilde{\mathbf{G}}_q$ and $\phi \in L^2(\tilde{\mathbf{G}}_q)$ the convolution $\mu * \phi \in L^2(\tilde{\mathbf{G}}_q)$ is defined by
\begin{align*}
(\mu * \phi)(g) = \sum_{h \in \tilde{\mathbf{G}}_q} \mu(h) \phi(gh^{-1}) \qquad \text{for all $g \in \tilde{\mathbf{G}}_q$}.
\end{align*}

For all $\xi \in \mathbb C$, $q \in \mathcal{O}$ coprime to $q_0$, $x \in \Lambda$, integers $0 < r < s$, and admissible sequences $(\alpha_s, \alpha_{s - 1}, \dotsc, \alpha_{r + 1})$, we define the complex measures
\begin{align*}
\mu_{(\alpha_s, \alpha_{s - 1}, \dotsc, \alpha_{r + 1})}^{\xi, q, x} &= \sum_{\alpha^r} e^{(f_s^{(a)} + ib\tau_s)(g(\alpha^s)x)} \delta_{\mathtt{c}_q^r(g(\alpha^r)x)} \\
\nu_0^{a, q, x, r} &= \sum_{\alpha^r} e^{f_r^{(a)}(g(\alpha^r)x)} \delta_{\mathtt{c}_q^r(g(\alpha^r)x)} \\
\hat{\mu}_{(\alpha_s, \alpha_{s - 1}, \dotsc, \alpha_{r + 1})}^{a, q, x} &= \sum_{\alpha^r} e^{f_s^{(a)}(g(\alpha^s)x)} \delta_{\mathtt{c}_q^r(g(\alpha^r)x)} = e^{f_{s - r}^{(a)}(g(\alpha^s)x)} \nu_0^{a, q, x, r} \\
\nu_{(\alpha_s, \alpha_{s - 1}, \dotsc, \alpha_{r + 1})}^{a, q, x} &= e^{f_{s - r}^{(a)}(g_{\alpha_s} g_{\alpha_{s - 1}} \dotsb g_{\alpha_{r + 1}} \omega(\alpha_{r + 1}))} \nu_0^{a, q, x, r}
\end{align*}
on $\tilde{\mathbf{G}}_q$ and also for all $H \in C(\Lambda, L^2(\tilde{\mathbf{G}}_q))$, define the function
\begin{align*}
\phi_{(\alpha_s, \alpha_{s - 1}, \dotsc, \alpha_{r + 1})}^{q, H} = \delta_{\mathtt{c}_q^{s - r}(g_{\alpha_s} g_{\alpha_{s - 1}} \dotsb g_{\alpha_{r + 1}} \omega(\alpha_{r + 1}))} * H(g_{\alpha_s} g_{\alpha_{s - 1}} \dotsb g_{\alpha_{r + 1}} \omega(\alpha_{r + 1}))
\end{align*}
in $L^2(\tilde{\mathbf{G}}_q)$ where we note that $\big\|\phi_{(\alpha_s, \alpha_{s - 1}, \dotsc, \alpha_{r + 1})}^{q, H}\big\|_2 \leq \|H\|_\infty$.

Now we present a simple bound which will be used often (see for example \cite[Lemma 7.1]{Sar22} for a proof). Fix $C_f$ provided by \cref{lem:SumExpf^aBound} henceforth.

\begin{lemma}
\label{lem:SumExpf^aBound}
There exists $C_f > 1$ such that for all $|a| < a_0'$, $x \in \Lambda$, and $k \in \mathbb N$, we have $\sum_{\alpha^k} e^{f_k^{(a)}(g(\alpha^k)x)} \leq C_f$.
\end{lemma}

We now record a lemma which relate the complex measures defined above.

\begin{lemma}
\label{lem:muHatLessThanCnu}
There exists $C > 0$ such that for all $\xi \in \mathbb C$ with $|a| < a_0'$, $q \in \mathcal{O}$ coprime to $q_0$, $x \in \Lambda$, integers $0 < r < s$, and admissible sequences $(\alpha_s, \alpha_{s - 1}, \dotsc, \alpha_{r + 1})$, we have $\left|\mu_{(\alpha_s, \alpha_{s - 1}, \dotsc, \alpha_{r + 1})}^{\xi, q, x}\right| \leq \hat{\mu}_{(\alpha_s, \alpha_{s - 1}, \dotsc, \alpha_{r + 1})}^{a, q, x}$ and
\begin{align*}
C^{-1}\nu_{(\alpha_s, \alpha_{s - 1}, \dotsc, \alpha_{r + 1})}^{a, q, x} \leq \hat{\mu}_{(\alpha_s, \alpha_{s - 1}, \dotsc, \alpha_{r + 1})}^{a, q, x} \leq C\nu_{(\alpha_s, \alpha_{s - 1}, \dotsc, \alpha_{r + 1})}^{a, q, x}.
\end{align*}
\end{lemma}

\begin{proof}
Fix $C = e^{\frac{T_0 \theta}{1 - \theta}}$. Let $\xi \in \mathbb C$ with $|a| < a_0'$, $q \in \mathcal{O}$ be coprime to $q_0$, $x \in \Lambda$, $0 < r < s$ be integers, and $(\alpha_s, \alpha_{s - 1}, \dotsc, \alpha_{r + 1})$ be an admissible sequence. It is easy to check $f_s^{(a)}(g(\alpha^s)x) = f_{s - r}^{(a)}(g(\alpha^s)x) + f_r^{(a)}(g(\alpha^r)x)$ and the first inequality of the lemma. We also have
\begin{align}
\label{eqn:f^(a)LipschitzBounds}
\begin{aligned}
&\big|f_s^{(a)}(g(\alpha^s)x) - \big(f_{s - r}^{(a)}(g_{\alpha_s}g_{\alpha_{s - 1}} \dotsb g_{\alpha_{r + 1}} \omega(\alpha_{r + 1})) + f_r^{(a)}(g(\alpha^r)x)\big)\big| \\
={}&\big|f_{s - r}^{(a)}(g(\alpha^s)x) - f_{s - r}^{(a)}(g_{\alpha_s}g_{\alpha_{s - 1}} \dotsb g_{\alpha_{r + 1}}\omega(\alpha_{r + 1}))\big| \\
\leq{}&\sum_{k = 0}^{s - r - 1}\big|f^{(a)}\bigl(T^k(g(\alpha^s)x)\bigr) - f^{(a)}\bigl(T^k(g_{\alpha_s}g_{\alpha_{s - 1}} \dotsb g_{\alpha_{r + 1}}\omega(\alpha_{r + 1}))\bigr)\big| \\
\leq{}&\Lip(f^{(a)}) \cdot \max_{j \in \{1, 2, \dotsc, N\}} \diam(D_j) \cdot \sum_{k = 0}^{s - r - 1} \frac{1}{c_0\kappa_2^k} \leq \frac{T_0 \hat{D} \kappa_2}{c_0(\kappa_2 - 1)}.
\end{aligned}
\end{align}
Hence, the lemma follows by comparing $\hat{\mu}_{(\alpha_s, \alpha_{s - 1}, \dotsc, \alpha_{r + 1})}^{a, q, x}$ and $\nu_{(\alpha_s, \alpha_{s - 1}, \dotsc, \alpha_{r + 1})}^{a, q, x}$.
\end{proof}

Now we return to our goal of approximating the transfer operator. The following theorem is proved as in \cite[Theorem 4.21]{OW16} using bounds similar to \cref{eqn:f^(a)LipschitzBounds}.

\begin{lemma}
\label{lem:TransferOperatorConvolutionApproximation}
For all $\xi \in \mathbb C$ with $|a| < a_0'$, $q \in \mathcal{O}$ coprime to $q_0$, $x \in \Lambda$, integers $0 < r < s$, and $H \in \mathcal{V}_q(\Lambda)$, we have
\begin{multline*}
\left\|M_{\xi, q}^s(H)(x) - \sum_{\alpha_{r + 1}, \alpha_{r + 2}, \dotsc, \alpha_s} \mu_{(\alpha_s, \alpha_{s - 1}, \dotsc, \alpha_{r + 1})}^{\xi, q, x} * \phi_{(\alpha_s, \alpha_{s - 1}, \dotsc, \alpha_{r + 1})}^{q, H}\right\|_2 \\
\leq c_0^{-1}\kappa_2\hat{D}C_f \Lip(H)\kappa_2^{-(s - r)}.
\end{multline*}
\end{lemma}

We will use this approximation to study the convolution, rather than dealing with the transfer operator directly, and obtain strong bounds. This is the objective of \cref{sec:L2FlatteningLemma} but first we need to establish some important facts in \cref{sec:ZariskiDensityOfTheReturnTrajectorySubgroups}.

\section{Zariski density and trace field of the return trajectory subgroups}
\label{sec:ZariskiDensityOfTheReturnTrajectorySubgroups}
In this section, we prove Zariski density of the return trajectory subgroups in \cref{thm:Z-DenseInG} which will be required to use the expander machinery of Golsefidy--Varj\'{u} \cite{GV12} in \cref{sec:L2FlatteningLemma}.

We make a similar definition as in \cite[Section 8]{Sar22}.

\begin{definition}[Return trajectory subgroup]
For all $p \in \mathbb N$ and $(y, z) \in \mathcal{A}^2$, we define the \emph{return trajectory subgroup} $H^p(y, z)$ to be the subgroup of $\langle\Gamma\rangle$ generated by the subset $S^p(y, z)$ which consists of the elements
\begin{align*}
\prod_{j = 1}^p g_{\alpha_j} \prod_{j = 1}^p g_{\tilde{\alpha}_{p + 1 - j}}^{-1} \in \langle\Gamma\rangle
\end{align*}
for all admissible sequences $(y, \alpha_p, \alpha_{p - 1}, \dotsc, \alpha_1, z)$ and $(y, \tilde{\alpha}_p, \tilde{\alpha}_{p - 1}, \dotsc, \tilde{\alpha}_1, z)$.
\end{definition}

For all $p \in \mathbb N$ and $(y, z) \in \mathcal{A}^2$, we denote $\tilde{S}^p(y, z) = \{\tilde{\gamma} \in \langle\tilde{\Gamma}\rangle: \gamma \in S^p(y, z)\} \subset \langle\tilde{\Gamma}\rangle$ and $\tilde{H}^p(y, z) = \langle \tilde{S}^p(y, z) \rangle = \{\tilde{\gamma} \in \langle\tilde{\Gamma}\rangle: \gamma \in H^p(y, z)\} < \langle\tilde{\Gamma}\rangle$.

\begin{theorem}
\label{thm:Z-DenseInG}
For all $(y, z) \in \mathcal A^2$, there exists $p_0 \in \mathbb N$ such that for all integers $p > p_0$, the subgroup $H^p(y, z)$ is Zariski dense in $\mathbf{G}$.
\end{theorem}

We have the following corollary which is proved as in \cite[Corollary 8.5]{Sar22}.

\begin{corollary}
\label{cor:Z-DenseInU-CoverG}
There exists $p_0 \in \mathbb N$ such that for all integers $p > p_0$ and $(y, z) \in \mathcal A^2$, the subgroup $\tilde{H}^p(y, z)$ is Zariski dense in $\tilde{\mathbf{G}}$.
\end{corollary}

Recall the upper half space model. Let $B^{\mathrm{E}}_\epsilon(u) \subset \mathbb R^{n - 1} \subset \partial_\infty\mathbb H^n$ denote the open Euclidean ball of radius $\epsilon > 0$ centered at $u \in \mathbb R^{n - 1}$. By a $k$-sphere in $\partial_\infty\mathbb H^n$ we mean that it is a Euclidean $k$-dimensional sphere or affine space in $\mathbb R^{n - 1}$, for all $0 \leq k \leq n - 2$. For the proof of \cref{thm:Z-DenseInG}, we recall the following well-known lemma and include its proof for convenience.

\begin{lemma}
\label{lem:NotZ-DenseImpliesLimitSetIsContainedInASphere}
If $H < G$ is not Zariski dense, then its limit set $\Lambda(H)$ is contained in a $(n - 2)$-sphere in $\partial_\infty \mathbb H^n$.
\end{lemma}

\begin{proof}
Let $H < G$ be a subgroup which is not Zariski dense. If $H < G$ is elementary, then $\#\Lambda(H) \leq 2$ and hence the lemma follows. Otherwise, $H < G$ is nonelementary and hence $\#\Lambda(H) = \infty$. Let $H^\dagger < G$ be the maximal connected proper Lie subgroup containing the identity component of the Zariski closure of $H$ in $\mathbf{G}(\mathbb R)$. Note that $\#\Lambda(H) = \infty$ implies that $H^\dagger$ is not contained in any parabolic subgroup. Thus, $H^\dagger$ is reductive by \cite{Mos61}. The Karpelevi\v{c}--Mostow theorem \cite{Kar53,Mos55} then states that the (nontrivial) semisimple factor of $H^\dagger$ preserves a proper totally geodesic submanifold which in this case is a $(n - 1)$-sphere in $\mathbb H^n$ perpendicular to $\mathbb R^{n - 1}$. Consequently, $\Lambda(H)$ is contained in a unique $(n - 2)$-sphere in $\partial_\infty \mathbb H^n$.
\end{proof}

\begin{proof}[Proof of \cref{thm:Z-DenseInG}]
Let $(y, z) \in \mathcal{A}^2$. Recall that $(y, y)$ is an admissible pair. Since there are at least two Schottky generators, there exists $y_0 \in \mathcal{A} \setminus \{y\}$ such that $(y, y_0)$ is also an admissible pair. As $\Gamma$ is Zariski dense, we can use the contrapositive of \cite[Proposition 3.12]{Win15} to choose a set of distinct limit points $\{u_1, u_2, \dotsc, u_{n + 1}\} \subset \mathtt{C}[z] = \Lambda \cap D_z$ such that it is not contained in any $(n - 2)$-sphere in $\partial_\infty\mathbb H^n$. There exists $\epsilon > 0$ such that if $u_j' \in B_\epsilon^{\mathrm{E}}(u_j)$ for all $1 \leq j \leq n + 1$, then $\{u_1, u_2, \dotsc, u_{n + 1}\} \subset D_z$ and not contained in any $(n - 2)$-sphere in $\partial_\infty\mathbb H^n$. Note that $\{\gamma^+ \in \partial_\infty\mathbb H^n: \gamma \in \Gamma\} \subset \Lambda$ is dense. Thus, we can choose hyperbolic elements $\gamma_1', \gamma_2', \dotsc, \gamma_{n + 1}' \in \Gamma$ with $\len(\gamma_j') = L_j$ such that $(\gamma_j')^+ \in B_{\epsilon/2}^{\mathrm{E}}(u_j)$ for all $1 \leq j \leq n + 1$. Note that the words all start with $g_z$. Moreover, they can be chosen such that the words end also with $g_z$ because otherwise it simply requires appending the words with at most $2$ admissible generators of $\Gamma$. Fix $L = \prod_{j = 1}^{n + 1} L_j$. Either $(z, y)$ or $(z, y_0)$ is an admissible pair. Accordingly, fix the words $h_0 = g_y g_{y_0}$ and $h = g_y^2$ if the former holds and $h_0 = g_{y_0}^2$ and $h = g_{y_0} g_{y}$ otherwise. For some $r \in \mathbb N$, define $\gamma_0 = g_z^{rL}h_0$ and $\gamma_j = (\gamma_j')^{\frac{rL}{L_j}}h$ for all $1 \leq j \leq n + 1$ which all have word length $rL + 2$. Taking $r$ sufficiently large and $p = rL + 2$, we have $\gamma_j \gamma_0^{-1} \in H^p(y, z)$ with $(\gamma_j \gamma_0^{-1})^+ \in B_\epsilon^{\mathrm{E}}(u_j)$ for all $1 \leq j \leq n + 1$. Thus, $\Lambda(H^p(y, z))$ is not contained in any $(n - 2)$-sphere in $\partial_\infty\mathbb H^n$. Hence, $H^p(y, z)$ is Zariski dense by \cref{lem:NotZ-DenseImpliesLimitSetIsContainedInASphere}.
\end{proof}

In the continued fractions semigroup setting, we also require the following trace field property.

\begin{theorem}
\label{thm:TraceField}
We have $\mathbb Q(\tr(H^p(y, z))) = \mathbb Q(i)$ when $\Gamma$ is a continued fractions semigroup.
\end{theorem}

\begin{proof}
It suffices to show that there exists $\gamma \in H^p(y, z)$ such that $\tr(\gamma) \in \mathbb C \setminus \mathbb R$. We show this by computing traces of various elements in $H^p(y, z)$. This laborious task can be eased with a computer algebra system. We proceed via a series of cases where we use the notation introduced in \cref{subsec:ContinuedFractionsSemigroupSetting}. Recall that $\mathscr{A} \not\subset \mathbb N$ and $\#\mathscr{A} \geq 2$. Thus, we can choose distinct elements $a, b \in \mathscr{A}$ such that $\{a, b\} \not\subset \mathbb R$.

\medskip
\noindent
\textit{Case 1.} Suppose $(a - b)^2 \in \mathbb C \setminus \mathbb R$. Then, we compute that
\begin{align*}
\tr\bigl(g_{a, a}g_{b, b}^{-1}\bigr) = (a - b)^2 + 2 \in \mathbb C \setminus \mathbb R
\end{align*}
which proves the lemma.

\medskip
\noindent
\textit{Case 2.} Suppose $(a - b)^2 \in \mathbb R$ but $(ab)^2 \in \mathbb C \setminus \mathbb R$. Then, we compute that
\begin{align*}
\tr\bigl(g_{a, a}g_{a, a}g_{b, b}^{-1}g_{b, b}^{-1}\bigr) = (a - b)^4 + ((ab)^2 + 4)(a - b)^2 + 2 \in \mathbb C \setminus \mathbb R
\end{align*}
which proves the lemma.

\medskip
\noindent
\textit{Case 3.} Suppose $(a - b)^2, (ab)^2 \in \mathbb R$ but $ab \in \mathbb C \setminus \mathbb R$. Then, we compute that
\begin{align*}
\tr\bigl(g_{a, b}g_{a, b}g_{b, a}^{-1}g_{b, a}^{-1}\bigr) = -((ab)^2 + 4ab + 4)(a - b)^2 + 2 \in \mathbb C \setminus \mathbb R
\end{align*}
which proves the lemma.

\medskip
\noindent
\textit{Case 4.} Suppose $(a - b)^2, ab \in \mathbb R$ but $2a^3b - a^4 \in \mathbb C \setminus \mathbb R$. Then, we compute that
\begin{align*}
\tr\bigl(g_{a, a}g_{a, a}g_{b, a}^{-1}g_{b, a}^{-1}\bigr) = 2a^3b - a^4 - (ab)^2 + 2 \in \mathbb C \setminus \mathbb R
\end{align*}
which proves the lemma.

\medskip
\noindent
\textit{Case 5.} Suppose $(a - b)^2, ab, 2a^3b - a^4 \in \mathbb R$. We show that in fact this case is an impossibility by obtaining a contradiction. Since $(a - b)^2, ab \in \mathbb R$ and $\{a, b\} \subset \mathscr{A} \subset \mathbb N + i\mathbb Z$ while $\{a, b\} \not\subset \mathbb R$, we conclude that $b = \overline{a}$. Now write $a = re^{i\theta}$ and $b = \overline{a} = re^{-i\theta}$ for some $r > 0$ and $\theta \in \bigl(-\frac{\pi}{2}, \frac{\pi}{2}\bigr)$. Then, we compute that $2a^3b - a^4 = -r^4(\cos(4\theta) - 2\cos(2\theta)) + 8ir^4\sin^3(\theta)\cos(\theta) \in \mathbb R$ which holds if and only if $\sin^3(\theta)\cos(\theta) = 0$. Thus $\theta = 0$. But then $a = b$ and $\{a, b\} \subset \mathbb R$ which is a contradiction.
\end{proof}

\section{\texorpdfstring{$L^2$}{L-2}-flattening lemma}
\label{sec:L2FlatteningLemma}
In this section, we outline the proof of the following $L^2$-flattening lemma. The arguments are based on \cite[Section 9]{Sar22} which is itself generalized from \cite[Appendix]{MOW19} due to Bourgain--Kontorovich--Magee. The main tool required in the proof is the expander machinery of Golsefidy--Varj\'{u} \cite{GV12} which cannot be used directly but culminates in \cref{lem:ExpanderMachineryBound}.

\begin{lemma}
\label{lem:L2FlatteningLemma}
There exist $C > 0$, $C_0 > 0$, and $l \in \mathbb N$ such that for all $\xi \in \mathbb C$ with $|a| < a_0'$, square-free $q \in \mathcal{O}$ coprime to $q_0$, $x \in \Lambda$, integers $C_0 \log(N(q)) \leq r < s$ with $r \in l\mathbb Z$, admissible sequences $(\alpha_s, \alpha_{s - 1}, \dotsc, \alpha_{r + 1})$, and $\phi \in E_q^q$ with $\|\phi\|_2 = 1$, we have
\begin{align*}
\left\|\mu_{(\alpha_s, \alpha_{s - 1}, \dotsc, \alpha_{r + 1})}^{\xi, q, x} * \phi\right\|_2 &\leq C N(q)^{-\frac{1}{3}} \left\|\nu_{(\alpha_s, \alpha_{s - 1}, \dotsc, \alpha_{r + 1})}^{a, q, x}\right\|_1; \\
\left\|\hat{\mu}_{(\alpha_s, \alpha_{s - 1}, \dotsc, \alpha_{r + 1})}^{a, q, x} * \phi\right\|_2 &\leq C N(q)^{-\frac{1}{3}} \left\|\nu_{(\alpha_s, \alpha_{s - 1}, \dotsc, \alpha_{r + 1})}^{a, q, x}\right\|_1.
\end{align*}
\end{lemma}

In this section, we fix any $p > p_0$ from \cref{cor:Z-DenseInU-CoverG} so that the corollary applies when it is needed in \cref{lem:GV_Expander}. For the purposes of proving \cref{lem:L2FlatteningLemma}, we will fix $\xi \in \mathbb C$ with $|a| < a_0'$, $x \in \Lambda$, $r \in \mathbb N$ with factorization $r = r'l$ for some fixed $r' \in \mathbb N$ and some fixed integer $l > p$ henceforth in this section. For all $q \in \mathcal{O}$ coprime to $q_0$, for all admissible sequences $(\alpha_s, \alpha_{s - 1}, \dotsc, \alpha_{r + 1})$, denote the complex measures $\mu_{(\alpha_s, \alpha_{s - 1}, \dotsc, \alpha_{r + 1})}^{\xi, q, x}$, $\nu_0^{a, q, x, r}$, $\hat{\mu}_{(\alpha_s, \alpha_{s - 1}, \dotsc, \alpha_{r + 1})}^{a, q, x}$, and $\nu_{(\alpha_s, \alpha_{s - 1}, \dotsc, \alpha_{r + 1})}^{a, q, x}$ by $\mu_{(\alpha_s, \alpha_{s - 1}, \dotsc, \alpha_{r + 1})}^q$, $\nu_0^q$, $\hat{\mu}_{(\alpha_s, \alpha_{s - 1}, \dotsc, \alpha_{r + 1})}^q$, and $\nu_{(\alpha_s, \alpha_{s - 1}, \dotsc, \alpha_{r + 1})}^q$ respectively.

Let $\alpha^r$ be an admissible sequence. We introduce the following additional notations to manipulate sequences. Define
\begin{align*}
\alpha_j^l &= (\alpha_{jl}, \alpha_{jl - 1}, \dotsc, \alpha_{(j - 1)l + 1}); \\
\alpha_j^{(l - p)_1} &= (\alpha_{jl}, \alpha_{jl - 1}, \dotsc, \alpha_{(j - 1)l + p + 1}); \\
\alpha_j^{(p)_2} &= (\alpha_{(j - 1)l + p}, \alpha_{(j - 1)l + p - 1}, \dotsc, \alpha_{(j - 1)l + 1})
\end{align*}
for all $1 \leq j \leq r'$. With these notations, we have $\alpha^r = (\alpha_{r'}^l, \alpha_{r' - 1}^l, \dotsc,  \alpha_1^l) = (\alpha_r, \alpha_{r - 1}, \dotsc, \alpha_1)$ and $\alpha_j^l = \bigl(\alpha_j^{(l - p)_1}, \alpha_j^{(p)_2}\bigr)$ for all $1 \leq j \leq r'$. We also have $T^k(g(\alpha^j)x) = g(\alpha^{j - k})x$ for all $x \in \Lambda$, $1 \leq j \leq r$, and $0 \leq k \leq j - 1$.

For all admissible sequences $\alpha^r$, we compute that
\begin{align*}
f_r^{(a)}(g(\alpha^r)x) ={}&\sum_{k = 0}^{r - 1} f^{(a)}(T^k(g(\alpha^r)x)) = \sum_{k = 0}^{r - 1} f^{(a)}(g(\alpha^{r - k})x) \\
={}&\sum_{j = 0}^{r' - 1} \sum_{k = 0}^{l - 1} f^{(a)}(g(\alpha^{r - (jl + k)})x) = \sum_{j = 0}^{r' - 1} \sum_{k = 0}^{l - 1} f^{(a)}(T^k(g(\alpha^{r - jl})x)) \\
={}&\sum_{j = 0}^{r' - 1} f_l^{(a)}(g(\alpha^{r - jl})x) = \sum_{j = 1}^{r'} f_l^{(a)}(g(\alpha^{jl})x).
\end{align*}
We can estimate each term in the sum above so that in the $j$\textsuperscript{th} term, we remove dependence on $\alpha_k^{(p)_2}$ for all distinct integers $1 \leq j, k \leq r'$. This is not required for $j = 1$ since the first term does not have any dependence on $\alpha_k^{(p)_2}$ for all $2 \leq k \leq r'$.

\Cref{lem:Estimate_f_ToRemoveDependence} is proved as in \cite[Lemma 9.2]{Sar22} using bounds similar to \cref{eqn:f^(a)LipschitzBounds}.

\begin{lemma}
\label{lem:Estimate_f_ToRemoveDependence}
There exists $C > 0$ such that for all admissible sequences $\alpha^r$, we have
\begin{align*}
\big|f_l^{(a)}(g(\alpha^{jl})x) - f_l^{(a)}\big(g\big(\alpha_j^l, \alpha_{j - 1}^{(l - p)_1}\big) \omega\big(\alpha_{j - 1}^{(l - p)_1}\big)\big)\big| \leq C \kappa_2^{-l}
\end{align*}
for all $2 \leq j \leq r'$, where $C$ is independent of $|a| < a_0'$, $x \in \Lambda$, $r \in \mathbb N$ and its factorization $r = r'l$ with $l > p$.
\end{lemma}

To make sense of the notations in what follows, we make the convention that $\alpha_0^{(l - p)_1}$ is the empty sequence for all admissible sequences $\alpha^r$. Using the calculations and \cref{lem:Estimate_f_ToRemoveDependence} above, for all $q \in \mathcal{O}$ coprime to $q_0$, for all admissible sequences $\alpha^r$, define the coefficients
\begin{align*}
E\bigl(\alpha_j^l, \alpha_{j - 1}^{(l - p)_1}\bigr) =
\begin{cases}
e^{f_l^{(a)}(g(\alpha_1^l)x)}, & j = 1 \\
e^{f_l^{(a)}(g(\alpha_j^l, \alpha_{j - 1}^{(l - p)_1}) \omega(\alpha_{j - 1}^{(l - p)_1}))}, & 2 \leq j \leq r'
\end{cases}
\end{align*}
and the measures
\begin{align*}
\eta^q\bigl(\alpha_j^{(l - p)_1}, \alpha_{j - 1}^{(l - p)_1}\bigr) = \sum_{\alpha_j^{(p)_2}} E\bigl(\alpha_j^l, \alpha_{j - 1}^{(l - p)_1}\bigr) \delta_{\mathtt{c}_q^l(g(\alpha^{jl})x)} \qquad \text{for all $1 \leq j \leq r'$}
\end{align*}
where we show the dependence of the admissible choices of $\alpha_j^{(p)_2}$ on $\alpha_j^{(l - p)_1}$ and $\alpha_{j - 1}^{(l - p)_1}$ (or more precisely only on the last entry of $\alpha_j^{(l - p)_1}$ and the first entry of $\alpha_{j - 1}^{(l - p)_1}$). These measures satisfy a property as shown in \cref{lem:NearlyFlat} which is called \emph{nearly flat}. It is proved as in \cite[Lemma 9.3]{Sar22} using bounds similar to \cref{eqn:f^(a)LipschitzBounds}.

\begin{lemma}
\label{lem:NearlyFlat}
There exists $C > 1$ such that for all $1 \leq j \leq r'$, and for all pairs of admissible sequences $\bigl(\alpha_j^l, \alpha_{j - 1}^{(l - p)_1}\bigr)$ and $\bigl(\tilde{\alpha}_j^l, \tilde{\alpha}_{j - 1}^{(l - p)_1}\bigr)$ with $\bigl(\alpha_j^{(l - p)_1}, \alpha_{j - 1}^{(l - p)_1}\bigr) = \bigl(\tilde{\alpha}_j^{(l - p)_1}, \tilde{\alpha}_{j - 1}^{(l - p)_1}\bigr)$, we have
\begin{align*}
\frac{E\bigl(\tilde{\alpha}_j^l, \tilde{\alpha}_{j - 1}^{(l - p)_1}\bigr)}{E\bigl(\alpha_j^l, \alpha_{j - 1}^{(l - p)_1}\bigr)} \leq C
\end{align*}
where $C$ is independent of $|a| < a_0'$, $x \in \Lambda$, $r \in \mathbb N$ and its factorization $r = r'l$ with $l > p$.
\end{lemma}

For all $q \in \mathcal{O}$ coprime to $q_0$, we also define the measure
\begin{align*}
\nu_1^q = \sum_{\alpha_1^{(l - p)_1}, \alpha_2^{(l - p)_1}, \dotsc, \alpha_{r'}^{(l - p)_1}} \mathop{\bigast}\limits_{j = 1}^{r'} \eta^q\bigl(\alpha_j^{(l - p)_1}, \alpha_{j - 1}^{(l - p)_1}\bigr)
\end{align*}
which in particular consists of convolutions of nearly flat measures. \cref{lem:EstimateNu} shows that we can estimate $\nu_0^q$ with $\nu_1^q$ up to a multiplicative constant and vice versa where the constant has exponential dependence on $r'$ in the factorization $r = r'l$. It is proved using \cref{lem:Estimate_f_ToRemoveDependence} as in \cite[Lemma 9.4]{Sar22}.

\begin{lemma}
\label{lem:EstimateNu}
There exists $C > 0$ such that for all $q \in \mathcal{O}$ coprime to $q_0$, we have $\nu_0^q \leq e^{r' C\kappa_2^{-l}}\nu_1^q$ and $\nu_1^q \leq e^{r' C\kappa_2^{-l}} \nu_0^q$ where $C$ is independent of $|a| < a_0'$, $x \in \Lambda$, $r \in \mathbb N$ and its factorization $r = r'l$ with $l > p$.
\end{lemma}

We now quote the strong approximation theorem of Weisfeiler \cite{Wei84} and the expander machinery of Golsefidy--Varj\'{u} \cite{GV12} in the context of our setting which will be used to prove \cref{lem:GV_Expander} regarding a spectral gap. Here, $\mathbb K = \mathbb Q$ (resp. $\mathbb K = \mathbb Q(i)$) if $\mathcal{O} = \mathbb Z$  (resp. $\mathcal{O} = \mathbb Z[i]$). Note that the stronger \cite[Theorem 1]{GV12} is available for $\mathbb K = \mathbb Q$, though we do not need it. Also, recall from \cref{sec:Preliminaries} that we have already fixed $q_0 \in \mathcal{O}$ satisfying both the theorems below when invoked in \cref{lem:GV_Expander} for the return trajectory subgroup.

\begin{theorem}[{\cite[Theorem 10.1]{Wei84}}]
\label{thm:Wei}
Let $\mathbf{H}$ be a connected simply connected absolutely almost simple algebraic group defined over $\mathbb K$. Let $\Omega < \mathbf{H}(\mathcal{O})$ be a finitely generated Zariski dense subgroup with $\mathbb Q(\tr(\Omega)) = \mathbb K$. Then, there exists $q_0 \in \mathcal{O}$ such that the reduction maps $\pi_q|_{\Omega}: \Omega \to \mathbf{H}(\mathcal{O}/q\mathcal{O})$ are surjective for all $q \in \mathcal{O}$ coprime to $q_0$.
\end{theorem}

\begin{theorem}[{\cite[Corollary 6]{GV12}}]
\label{thm:GV}
Let $\Omega < \GL_N(\mathbb K)$ for some $N \in \mathbb N$ be a subgroup generated by a finite symmetric subset $S \subset \Omega$. If the Zariski closure of $\Omega$ is semisimple, then there exists $q_0 \in \mathcal{O}$ such that the Cayley graphs $\Cay(\pi_q(\Omega), \pi_q(S))$ form expanders with respect to square-free $q \in \mathcal{O}$ coprime to $q_0$.
\end{theorem}

Let $q \in \mathcal{O}$ be coprime to $q_0$, $1 \leq j \leq r'$ be an integer, and $\bigl(\alpha_j^l, \alpha_{j - 1}^{(l - p)_1}\bigr)$ and $\bigl(\tilde{\alpha}_j^l, \tilde{\alpha}_{j - 1}^{(l - p)_1}\bigr)$ be pairs of admissible sequences such that $\bigl(\alpha_j^{(l - p)_1}, \alpha_{j - 1}^{(l - p)_1}\bigr) = \bigl(\tilde{\alpha}_j^{(l - p)_1}, \tilde{\alpha}_{j - 1}^{(l - p)_1}\bigr)$. We have
\begin{align*}
\mathtt{c}_q^l(g(\alpha^{jl})x) &= \pi_q(g_{\alpha_{(j - 1)l + 1}} \cdot g_{\alpha_{(j - 1)l + 2}} \dotsb g_{\alpha_{jl}})
\end{align*}
for the first sequence and similarly for the second one. Using $g_{\alpha_{(j - 1)l + p + k}} = g_{\tilde{\alpha}_{(j - 1)l + p + k}}$ for all $1 \leq k \leq l - p$, we calculate that
\begin{align*}
\mathtt{c}_q^l(g(\alpha^{jl})x) \mathtt{c}_q^l(g(\tilde{\alpha}^{jl})x)^{-1} = \pi_q\left(\prod_{k = 1}^p g_{\alpha_{(j - 1)l + k}} \prod_{k = 1}^p g_{\tilde{\alpha}_{(j - 1)l + p + 1 - k}}^{-1}\right).
\end{align*}

Now \cref{cor:Z-DenseInU-CoverG} allows us to use \cref{thm:Wei,thm:GV} to prove the following lemma.

\begin{lemma}
\label{lem:GV_Expander}
There exists $\epsilon \in (0, 1)$ such that for all integers $0 \leq j \leq r'$, for all pairs of admissible sequences $\bigl(\alpha_{j + 1}^{(l - p)_2}, \alpha_j^{(l - p)_2}\bigr)$, for all square-free $q \in \mathcal{O}$ coprime to $q_0$, for all $\phi \in L_0^2(\tilde{\mathbf{G}}_q)$ with $\|\phi\|_2 = 1$, there exist admissible sequences $\bigl(\beta_{jl + 1}, \beta_j^{(p)_1}, \beta_{jl - p}\bigr)$ and $\bigl(\tilde{\beta}_{jl + 1}, \tilde{\beta}_j^{(p)_1}, \tilde{\beta}_{jl - p}\bigr)$ with $\beta_{jl + 1} = \tilde{\beta}_{jl + 1} = \alpha_{jl + 1}$ and $\beta_{jl - p} = \tilde{\beta}_{jl - p} = \alpha_{jl - p}$ such that
\begin{align*}
\|\delta_g * \phi - \phi\|_2 \geq \epsilon
\end{align*}
where $g = \pi_q\left(\prod_{k = 1}^p g_{\alpha_{(j - 1)l + k}} \prod_{k = 1}^p g_{\tilde{\alpha}_{(j - 1)l + p + 1 - k}}^{-1}\right)$ and $\epsilon$ is independent of $r \in \mathbb N$ and its factorization $r = r'l$ with $l > p$.
\end{lemma}

\begin{proof}
Uniformity of $\epsilon$ with respect to $r \in \mathbb N$ with factorization $r = r'l$ with $l > p$, integers $0 \leq j \leq r'$, and pairs of admissible sequences $\bigl(\alpha_{j + 1}^{(l - p)_2}, \alpha_j^{(l - p)_2}\bigr)$ is trivial since $\epsilon$ only depends on the first entry $\alpha_{jl + 1} \in \mathcal A$ of $\alpha_{j + 1}^{(l - p)_2}$ and the last entry $\alpha_{jl - p} \in \mathcal A$ of $\alpha_j^{(l - p)_2}$ and there are only a finite number of such elements. So let $0 \leq j \leq r'$ be an integer and $\bigl(\alpha_{j + 1}^{(l - p)_2}, \alpha_j^{(l - p)_2}\bigr)$ be a pair of admissible sequences. Denote $\tilde{S}^p(\alpha_{jl + 1}, \alpha_{jl - p})$ by $\tilde{S}^p$ and $\tilde{H}^p(\alpha_{jl + 1}, \alpha_{jl - p})$ by $\tilde{H}^p$. For all $q \in \mathcal{O}$, let $\tilde{S}^p_q = \pi_q(\tilde{S}^p)$ and $\tilde{H}^p_q = \pi_q(\tilde{H}^p) = \langle\tilde{S}^p_q\rangle$. We now invoke \cref{thm:Wei}. For the Schottky semigroup setting, only \cref{cor:Z-DenseInU-CoverG} suffices for the hypotheses to be fulfilled. For the continued fractions semigroup setting, both \cref{cor:Z-DenseInU-CoverG} and \cref{thm:TraceField} in tandem with \cite[Theorems 1 and 2]{Vin71} are required for the hypotheses to be fulfilled. In both the settings, we can conclude $\tilde{H}^p_q = \tilde{\mathbf{G}}_q$ for all $q \in \mathcal{O}$ coprime to $q_0$. Thus, again by \cref{cor:Z-DenseInU-CoverG}, we can use \cref{thm:GV} to further conclude that the Cayley graphs $\Cay(\tilde{\mathbf{G}}_q, \tilde{S}^p_q)$ form expanders with respect to \emph{square-free} $q \in \mathcal{O}$ coprime to $q_0$. This means that there exists $\epsilon \in (0, 1)$ such that for the graph Laplacian $\Delta_q: L^2(\tilde{\mathbf{G}}_q) \to L^2(\tilde{\mathbf{G}}_q)$ which is a self-adjoint operator defined by $\Delta_q(\phi)  = \phi - \frac{1}{\#\tilde{S}^p_q}\sum_{h \in \tilde{S}^p_q} \delta_h * \phi$ for all $\phi \in L^2(\tilde{\mathbf{G}}_q)$, the smallest eigenvalue is $\lambda_1(\Delta_q) = 0$ and the next smallest eigenvalue satisfies $\lambda_2(\Delta_q) \geq \epsilon$ for all square-free $q \in \mathcal{O}$ coprime to $q_0$. Note that the eigenspace corresponding to $\lambda_1(\Delta_q) = 0$ consists of constant functions. We conclude that for all $\phi \in L_0^2(\tilde{\mathbf{G}}_q)$ with $\|\phi\|_2 = 1$, we have $\|\Delta_q(\phi)\|_2 \geq \epsilon$ which implies $\sum_{h \in \tilde{S}^p_q} \|\delta_h * \phi - \phi\|_2 \geq \epsilon \cdot \#\tilde{S}^p_q$ and so there exists $g \in \tilde{S}^p_q$ such that $\|\delta_g * \phi - \phi\|_2 \geq \epsilon$, for all square-free $q \in \mathcal{O}$ coprime to $q_0$. But $\tilde{S}^p \subset \tilde{H}^p < \tilde{\Gamma}$ and recall the induced isomorphisms $\overline{\pi_{q}|_{\tilde{\Gamma}}}: \tilde{\Gamma}_q \backslash \tilde{\Gamma} \to \tilde{G}_q$ and $\overline{\tilde{\pi}}: \tilde{\Gamma}_q \backslash \tilde{\Gamma} \to \Gamma_q \backslash \Gamma$. Following these isomorphisms, we find that $g = \pi_q\left(\prod_{k = 1}^p g_{\alpha_{(j - 1)l + k}} \prod_{k = 1}^p g_{\tilde{\alpha}_{(j - 1)l + p + 1 - k}}^{-1}\right)$.
\end{proof}

Let $q \in \mathcal{O}$ be coprime to $q_0$. For all complex measures $\eta$ on $\tilde{\mathbf{G}}_q$, define the operator $\tilde{\eta}: L^2(\tilde{\mathbf{G}}_q) \to L^2(\tilde{\mathbf{G}}_q)$ by $\tilde{\eta}(\phi) = \eta * \phi$ for all $\phi \in L^2(\tilde{\mathbf{G}}_q)$ and denote by $\tilde{\eta}^*$ its adjoint. Define $\eta^*$ to be the complex measure on $\tilde{\mathbf{G}}_q$ by $\eta^*(g) = \overline{\eta(g^{-1})}$ for all $g \in \tilde{\mathbf{G}}_q$ so that $\tilde{\eta}^*(\phi) = \eta^* * \phi$ for all $\phi \in L^2(\tilde{\mathbf{G}}_q)$. The following \cref{lem:EtaOperatorBound} is proved as in \cite[Lemma 9.6]{Sar22} using \cref{lem:NearlyFlat,lem:GV_Expander}.

\begin{lemma}
\label{lem:EtaOperatorBound}
There exists $C \in (0, 1)$ such that for all square-free $q \in \mathcal{O}$ coprime to $q_0$, integers $1 \leq j \leq r'$, pairs of admissible sequences $\bigl(\alpha_j^{(l - p)_1}, \alpha_{j - 1}^{(l - p)_1}\bigr)$, for all $\phi \in L_0^2(\tilde{G}_q, \mathbb C)$ with $\|\phi\|_2 = 1$, we have
\begin{align*}
\left\|\eta^q\bigl(\alpha_j^{(l - p)_1}, \alpha_{j - 1}^{(l - p)_1}\bigr) * \phi\right\|_2 \leq C \left\|\eta^q\bigl(\alpha_j^{(l - p)_1}, \alpha_{j - 1}^{(l - p)_1}\bigr)\right\|_1
\end{align*}
where $C$ is independent of $|a| < a_0'$, $x \in \Lambda$, $r \in \mathbb N$ with factorization $r = r'l$ with $l > p$.
\end{lemma}

\Cref{lem:ExpanderMachineryBound} is proved as in \cite[Lemma 9.7]{Sar22} using \cref{lem:EstimateNu,lem:EtaOperatorBound}.

\begin{lemma}
\label{lem:ExpanderMachineryBound}
There exists $l_0 \in \mathbb N$ such that if $l > l_0$, then there exists $C \in (0, 1)$ such that for all square-free $q \in \mathcal{O}$ coprime to $q_0$, integers $s > r$, admissible sequences $(\alpha_s, \alpha_{s - 1}, \dotsc, \alpha_{r + 1})$, and $\phi \in L_0^2(\tilde{\mathbf{G}}_q)$ with $\|\phi\|_2 = 1$, we have
\begin{align*}
\left\|\nu_{(\alpha_s, \alpha_{s - 1}, \dotsc, \alpha_{r + 1})}^q * \phi\right\|_2 \leq C^r \left\|\nu_{(\alpha_s, \alpha_{s - 1}, \dotsc, \alpha_{r + 1})}^q\right\|_1
\end{align*}
where $C$ is independent of $|a| < a_0'$, $x \in \Lambda$, $r \in \mathbb N$ but dependent on the factorization $r = r'l$.
\end{lemma}

Note that $E_q^q$ is a $\tilde{\mu}$-invariant submodule of the left $\mathbb C[\tilde{\mathbf{G}}_q]$-module $L^2(\tilde{\mathbf{G}}_q)$ for all nonzero $q \in \mathcal{O}$ and complex measures $\mu$ on $\tilde{\mathbf{G}}_q$. The following lemma is proved as in \cite[Lemma 9.8]{Sar22}.

\begin{lemma}
\label{lem:ConvolutionBoundOnE_q^q}
There exists $C > 0$ such that for all square-free $q \in \mathcal{O}$ coprime to $q_0$, for all complex measures $\mu$ on $\tilde{\mathbf{G}}_q$, we have
\begin{align*}
\|\tilde{\mu}|_{E_q^q}\|_{\mathrm{op}} \leq C N(q)^{-\frac{1}{3}} (\#\tilde{\mathbf{G}}_q)^{\frac{1}{2}} \|\mu\|_2.
\end{align*}
\end{lemma}

\begin{remark}
The hypothesis that $q \in \mathcal{O}$ be square-free is not required in \cref{lem:ConvolutionBoundOnE_q^q}. In general, we do not get Chevalley groups but \cite[Proposition 4.2]{KS13} still holds.
\end{remark}

Now, \cref{lem:L2FlatteningLemma} is proved using \cref{lem:ExpanderMachineryBound} and \cref{lem:ConvolutionBoundOnE_q^q} exactly as in the proof of \cite[Lemma 9.1]{Sar22}.

\section{Lipschitz norm bounds and proof of \texorpdfstring{\cref{thm:ReducedTheoremSmall|b|}}{\autoref{thm:ReducedTheoremSmall|b|}}}
\label{sec:SupremumAndLipschitzBounds}
In this section we use \cref{lem:L2FlatteningLemma} to prove the Lipschitz norm bounds in \cref{lem:ReducedTheoremEstimate}. We then use them to prove \cref{thm:ReducedTheoremSmall|b|} by induction. We start with fixing some notations and easy bounds.

Let $q \in \mathcal{O}$ be nonzero. Fix integers
\begin{align*}
r_q &\in [C_0 \log(N(q)), C_0 \log(N(q)) + l) \cap l\mathbb Z; \\
s_q &\in \left(r_q - \frac{\log(N(q)) + \log(4C_1C_f)}{\log(\theta)}, C_s\log(N(q))\right)
\end{align*}
where we fix $C_0$ and $l$ to be constants from \cref{lem:L2FlatteningLemma}, $C_1$ to be the constant from \cref{lem:Small|b|Bound} and $C_s = C_0 - \frac{1}{\log(\theta)} + \frac{l}{\log(2)} - \frac{\log(4C_1C_f)}{\log(\theta)\log(2)} + \frac{1}{\log(2)}$ so that there is enough room for the integer $s_q$ to exist. These definitions of constants ensure that $C_0 \log(N(q)) \leq r_q < s_q$ and $4C_1C_f \theta^{s_q - r_q} \leq N(q)^{-1}$. For all $\xi \in \mathbb C$ with $|a| < a_0'$, for all square-free $q \in \mathcal{O}$ coprime to $q_0$, for all $x \in \Sigma^+$, for all integers $C_0 \log(N(q)) \leq r < s$ with $r \in l\mathbb Z$, for all admissible sequences $(\alpha_s, \alpha_{s - 1}, \dotsc, \alpha_{r + 1})$, we have
\begin{align*}
\left\|\nu_{(\alpha_s, \alpha_{s - 1}, \dotsc, \alpha_{r + 1})}^{a, q, x}\right\|_1 &= e^{f_{s - r}^{(a)}(g_{\alpha_s} g_{\alpha_{s - 1}} \dotsb g_{\alpha_{r + 1}} \omega(\alpha_{r + 1}))} \left(\sum_{\alpha^r} e^{f_r^{(a)}(g(\alpha^r)x)}\right) \\
&\leq C_f e^{f_{s - r}^{(a)}(g_{\alpha_s} g_{\alpha_{s - 1}} \dotsb g_{\alpha_{r + 1}} \omega(\alpha_{r + 1}))}
\end{align*}
by \cref{lem:SumExpf^aBound} and hence \cref{lem:L2FlatteningLemma} implies that for all $\phi \in E_q^q$ we have
\begin{align}
\label{eqn:BoundFromL2FlatteningLemma}
\|\mu * \phi\|_2 \leq CC_f N(q)^{-\frac{1}{3}} e^{f_{s - r}^{(a)}(g_{\alpha_s} g_{\alpha_{s - 1}} \dotsb g_{\alpha_{r + 1}} \omega(\alpha_{r + 1}))} \|\phi\|_2
\end{align}
where $\mu$ denotes either $\mu_{(\alpha_s, \alpha_{s - 1}, \dotsc, \alpha_{r + 1})}^{\xi, q, x}$ or $\hat{\mu}_{(\alpha_s, \alpha_{s - 1}, \dotsc, \alpha_{r + 1})}^{a, q, x}$ and $C$ is the constant from the same lemma. We also need the following basic lemma proved as in \cite[Lemma 10.2]{Sar22}.

\begin{lemma}
\label{lem:Small|b|Bound}
There exists $C > 1$ such that for all $\xi \in \mathbb C$ with $|a| < a_0'$ and $|b| \leq b_0$, elements $x, y \in \Lambda$, $s \in \mathbb N$, and admissible sequences $\alpha^s$, we have
\begin{align*}
\left|1 - e^{(f_s^{(a)} + ib\tau_s)(g(\alpha^s)y) - (f_s^{(a)} + ib\tau_s)(g(\alpha^s)x)}\right| \leq C \|x - y\|.
\end{align*}
\end{lemma}

Using the bounds from \cref{lem:TransferOperatorConvolutionApproximation,eqn:BoundFromL2FlatteningLemma,lem:Small|b|Bound}, we can derive the following \cref{lem:ReducedTheoremEstimate} as in \cite[Lemmas 10.1 and 10.3]{Sar22}.

\begin{lemma}
\label{lem:ReducedTheoremEstimate}
There exist $\theta \in (0, 1)$ and $q_1 \in \mathbb N$ such that for all $\xi \in \mathbb C$ with $|a| < a_0'$ and $|b| \leq b_0$, square-free $q \in \mathcal{O}$ coprime to $q_0$ with $N(q) > q_1$, and $H \in \mathcal{W}_q^q(\Lambda)$, we have
\begin{align*}
\big\|M_{\xi, q}^{s_q}(H)\big\|_{\Lip} \leq \frac{1}{2}N(q)^{-\theta} \|H\|_{\Lip}.
\end{align*}
\end{lemma}

\begin{proof}[Proof of \cref{thm:ReducedTheoremSmall|b|}]
Fix $\theta \in (0, 1)$ and $q_1 \in \mathbb N$ from \cref{lem:ReducedTheoremEstimate}. Fix $a_0 = a_0'$ and recall that we already fixed $b_0 = 1$. Recall the constant $C_s$. Let $\xi \in \mathbb C$ with $|a| < a_0$ and $|b| \leq b_0$. Let $q \in \mathcal{O}$ be square-free and coprime to $q_0$ with $N(q) > q_1$. Denote $s_q$ by $s$. Let $j \in \mathbb Z_{\geq 0}$ and $H \in \mathcal{W}_q^q(\Lambda)$. Then by induction, \cref{lem:ReducedTheoremEstimate} implies
\begin{align*}
\big\|M_{\xi, q}^{js}(H)\big\|_{\Lip} \leq N(q)^{-j\theta} \|H\|_{\Lip}.
\end{align*}
\end{proof}

\section{Reduction of \texorpdfstring{\cref{itm:CongruenceOperatorSpectralBoundsLarge|b|}}{Property \ref{itm:CongruenceOperatorSpectralBoundsLarge|b|}} of \texorpdfstring{\cref{thm:CongruenceOperatorSpectralBounds}}{\autoref{thm:CongruenceOperatorSpectralBounds}}}
\label{sec:Dolgopyat'sMethod}
In this section, we describe the reduction from \cref{itm:CongruenceOperatorSpectralBoundsLarge|b|} of \cref{thm:CongruenceOperatorSpectralBounds} to \cref{thm:Dolgopyat} which is the main technical theorem in our setting associated to Dolgopyat's method \cite{Dol98}. These techniques are now well developed and we mainly follow \cite{OW16,Sto11,SW21}.

Firstly, \cref{itm:CongruenceOperatorSpectralBoundsLarge|b|} in \cref{thm:CongruenceOperatorSpectralBounds} can be derived from \cref{thm:TheoremLarge|b|} by using the a priori estimates in \cref{lem:EstimatesToConvertL2BoundToLipschitzBound}. We omit the derivation since it is a minor modification of the one provided just after \cite[Proposition 5.3]{Nau05}.

\begin{theorem}
\label{thm:TheoremLarge|b|}
There exist $\eta > 0$, $C > 0$, $a_0 > 0$, and $b_0 > 0$ such that for all $\xi \in \mathbb C$ with $|a| < a_0$ and $|b| > b_0$, nonzero $q \in \mathcal{O}$, $k \in \mathbb N$, and $H \in \mathcal{V}_q(\Lambda)$, we have
\begin{align*}
\left\|M_{\xi, q}^k(H)\right\|_2 \leq Ce^{-\eta k} \|H\|_{1, |b|}.
\end{align*}
\end{theorem}

\begin{lemma}
\label{lem:EstimatesToConvertL2BoundToLipschitzBound}
There exist $\eta > 0$, $\kappa_1 > 0$, $\kappa_2 > 0$, $a_0 > 0$, and $b_0 > 0$ such that for all $\xi \in \mathbb C$ with $|a| < a_0$ and $|b| > b_0$, nonzero $q \in \mathcal{O}$, $k \in \mathbb N$, and $H \in \mathcal{V}_q(\Lambda)$, we have
\begin{align*}
\Lip\bigl(M_{\xi, q}^k(H)\bigr) &\leq \kappa_1 |b| \cdot \|H\|_\infty + e^{-\eta k}\Lip(H); \\
\bigl\|M_{0, q}^k(H)\bigr\|_\infty &\leq \int_\Lambda \|H(u)\|_2 \, d\nu_\Lambda(u) + \kappa_2 e^{-\eta k}\Lip(H).
\end{align*}
\end{lemma}

\begin{remark}
We omit proving \cref{lem:EstimatesToConvertL2BoundToLipschitzBound} since they have appeared before. The first inequality in \cref{lem:EstimatesToConvertL2BoundToLipschitzBound} is a corrected version of the corresponding inequalities in \cite[Lemma 5.2]{Nau05} and \cite[Lemma 23]{MOW19}. It is proved as in \cite[Proposition 2.1]{PP90} with minor modifications and using the hypothesis $|b| > b_0$. The second inequality is proved by expanding $M_{0, q}^k(H)$ from definitions, using $L_0^*(\nu_\Lambda) = \nu_\Lambda$, and then using the Lipschitz property with the average norm $\int_\Lambda \|H(u)\|_2 \, d\nu_\Lambda(u)$.
\end{remark}

Secondly, \cref{thm:TheoremLarge|b|} follows from \cref{thm:Dolgopyat} by a standard inductive argument which we omit since it can be found in the following papers. Theorems like \cref{thm:Dolgopyat}, which is the main objective in Dolgopyat's method, have appeared in many papers such as \cite{Dol98,Nau05,Sto11,OW16,MOW19,SW21,Sar22} where the \emph{congruence} versions have first appeared in \cite{OW16} and subsequently in \cite{MOW19,Sar22}. The main novelty here is that \cref{thm:Dolgopyat} is also a \emph{congruence} version for the Schottky \emph{semigroup} and the continued fractions \emph{semigroup} settings and in \emph{higher} dimensions. We mention that the proofs in \cref{sec:ProofOfDolgopyat'sMethod} for Dolgopyat's method go through uniformly in nonzero $q \in \mathcal{O}$ because of unitarity and local constancy of the cocycle $\mathtt{c}$, first observed in \cite{OW16}.

As in \cite[Section 5]{Sto11}, we define the crucial new distance function $d$ on $\Lambda$ by
\begin{align*}
d(u, u') =
\begin{cases}
\displaystyle\min_{\substack{u, u' \in \mathtt{C}\\ \mathtt{C} \text{ is a cylinder}}} \diam(\mathtt{C}), & u, u' \in D_j \text{ for some } j \in \mathcal{A} \\
1, & \text{otherwise}
\end{cases}
\qquad
\text{for all $u, u' \in \Lambda$}.
\end{align*}
We denote $L_d(\Lambda, \mathbb R)$ to be the space of $d$-Lipschitz functions and $\Lip_d(h)$ to be the $d$-Lipschitz constant for all $h \in L_d(\Lambda, \mathbb R)$. We define the cones
\begin{align*}
\mathcal{C}_B(\Lambda) &= \{h \in L_d(\Lambda, \mathbb R): h > 0, |h(u) - h(u')| \leq Bh(u)d(u, u') \text{ for all } u, u' \in \Lambda\} \\
\tilde{\mathcal{C}}_B(\Lambda) &= \Big\{h \in L_d(\Lambda, \mathbb R): h > 0, e^{-B d(u, u')} \leq \frac{h(u)}{h(u')} \leq e^{B d(u, u')} \text{ for all } u, u' \in \Lambda\Big\}.
\end{align*}

\begin{remark}
If $h \in \mathcal{C}_B(\Lambda)$, then using the convexity of $-\log$, we can derive that $h$ is $\log$-$d$-Lipschitz, i.e., $|(\log \circ h)(u) - (\log \circ h)(u')| \leq B d(u, u')$ for all $u, u' \in \Lambda$. It follows that $\mathcal{C}_B(\Lambda) \subset \tilde{\mathcal{C}}_B(\Lambda)$, however the reverse containment is not true.
\end{remark}

\begin{theorem}
\label{thm:Dolgopyat}
There exist $m \in \mathbb N$, $\eta \in (0, 1)$, $E > \max\big(1, \frac{1}{b_0}\big)$, $a_0 > 0$, $b_0 > 0$, and a set of operators $\{\mathcal{N}_{a, J}: L_d(\Lambda, \mathbb R) \to L_d(\Lambda, \mathbb R): |a| < a_0, J \in \mathcal{J}(b) \text{ for some } |b| > b_0\}$, where $\mathcal{J}(b)$ is some finite set for all $|b| > b_0$, such that
\begin{enumerate}
\item\label{itm:DolgopyatProperty1}	$\mathcal{N}_{a, J}(\mathcal{C}_{E|b|}(\Lambda)) \subset \mathcal{C}_{E|b|}(\Lambda)$ for all $|a| < a_0$, $J \in \mathcal{J}(b)$, and $|b| > b_0$;
\item\label{itm:DolgopyatProperty2}	$\|\mathcal{N}_{a, J}(h)\|_2 \leq \eta \|h\|_2$ for all $h \in \mathcal{C}_{E|b|}(\Lambda)$, $|a| < a_0$, $J \in \mathcal{J}(b)$, and $|b| > b_0$;
\item\label{itm:DolgopyatProperty3}	for all $\xi \in \mathbb C$ with $|a| < a_0$ and $|b| > b_0$ and nonzero $q \in \mathcal{O}$, if $H \in \mathcal{V}_q(\Lambda)$ and $h \in \mathcal{C}_{E|b|}(\Lambda)$ satisfy
\begin{enumerate}[label=(1\alph*), ref=\theenumi(1\alph*)]
\item\label{itm:DominatedByh}	$\|H(u)\|_2 \leq h(u)$ for all $u \in \Lambda$;
\item\label{itm:LogLipschitzh}	$\|H(u) - H(u')\|_2 \leq E|b|h(u)d(u, u')$ for all $u, u' \in \Lambda$;
\end{enumerate}
then there exists $J \in \mathcal{J}(b)$ such that
\begin{enumerate}[label=(2\alph*), ref=\theenumi(2\alph*)]
\item\label{itm:DominatedByDolgopyat}	$\big\|M_{\xi, q}^m(H)(u)\big\|_2 \leq \mathcal{N}_{a, J}(h)(u)$ for all $u \in \Lambda$;
\item\label{itm:LogLipschitzDolgopyat}	$\big\|M_{\xi, q}^m(H)(u) - M_{\xi, q}^m(H)(u')\big\|_2 \leq E|b|\mathcal{N}_{a, J}(h)(u)d(u, u')$ for all $u, u' \in \Lambda$.
\end{enumerate}
\end{enumerate}
\end{theorem}

\section{Local non-integrability condition and non-concentration property}
\label{sec:LocalNon-IntegrabilityCondition}
In this section, we will prove that the temporal distance function satisfies the local non-integrability condition (LNIC) and that the limit set satisfies the non-concentration property (NCP). We start with some definitions.

\begin{definition}[Temporal distance function]
For all $\alpha, \beta \in \Sigma$ and $k \in \mathcal{A}$ such that $(k, \alpha)$ and $(k, \beta)$ are admissible, the \emph{temporal distance function} $\varphi_{\alpha, \beta} \in C^1((D_k)^2, \mathbb R)$ is defined by
\begin{align*}
\varphi_{\alpha, \beta, u'}(u) = \varphi_{\alpha, \beta}(u, u') = \Delta_{\alpha}(u, u') - \Delta_{\beta}(u, u') \qquad \text{for all $u, u' \in D_k$}
\end{align*}
where for all $\alpha \in \Sigma$ and $k \in \mathcal{A}$ such that $(k, \alpha)$ is admissible, $\Delta_\alpha \in C^1((D_k)^2, \mathbb R)$ is defined by
\begin{align*}
\Delta_{\alpha}(u, u') = \sum_{j = 0}^\infty \big(\tau\big(g_{\alpha_j} g_{\alpha_{j - 1}} \dotsb g_{\alpha_0}u\big) - \tau\big(g_{\alpha_j} g_{\alpha_{j - 1}} \dotsb g_{\alpha_0}u'\big)\big)
\end{align*}
for all $u, u' \in D_k$.
\end{definition}

\begin{remark}
As $\tau$ is differentiable on cylinders of $D$, any partial derivative series corresponding to $\Delta_{\alpha}$ in the above definition converges uniformly using \cref{lem:Hyperbolicity}. Consequently, $\Delta_{\alpha}$ and $\varphi_{\alpha, \beta}$ are indeed of class $C^1$.
\end{remark}

\begin{definition}[LNIC]
We say $\varphi$ satisfies the \emph{local non-integrability condition (LNIC)} if there exist $\alpha, \beta \in \Sigma$, $k \in \mathcal{A}$, and $u_0, u_0' \in \mathtt{C}[k] = \Lambda \cap D_k$ such that $\nabla\varphi_{\alpha, \beta, u_0'}(u_0) \neq 0$.
\end{definition}

As in \cite{Nau05}, \cref{pro:ClosedGeodesicsBijectiveWithPeriodicOrbitsOf_T} will be very useful. It appears in \cite[Proposition 3.4]{PR97} for closed hyperbolic surfaces. Since the dimension is not critical, it holds in higher dimensions as recognized in the proof of \cite[Proposition 3.4]{GLZ04}. Also, it generalizes to the semigroup setting. We define $\len([\gamma]) = \inf_{\gamma' \in [\gamma]} \len(\gamma')$ for all conjugacy classes $[\gamma]$ of $\Gamma$ and denote by $\ell(\gamma)$ the translation length of $\gamma \in \Gamma$.

\begin{proposition}
\label{pro:ClosedGeodesicsBijectiveWithPeriodicOrbitsOf_T}
Let $\Gamma_0 \subset \Gamma$ be a Schottky subsemigroup. There is a one-to-one correspondence between nontrivial conjugacy classes $[\gamma]$ of $\Gamma_0$ and periodic orbits $\{T^j(u)\}_{j = 0}^{k - 1} \subset \Lambda(\Gamma_0)$ of length $k = \len([\gamma]) \in \mathbb N$ with $\tau_k(u) = \sum_{j = 0}^{k - 1} \tau(T^j(u)) = \ell(\gamma)$.
\end{proposition}

\begin{remark}
The explicit bijection in the above proposition is obtained by assigning $u = \gamma^+ \in \Lambda(\Gamma_0)$ for $\gamma \in [\gamma]$ such that $k = \len(\gamma) = \len([\gamma])$.
\end{remark}

The following is a useful lemma, where $A < \PSL_2(\mathbb C)$ is the subgroup of diagonal elements. We also denote by $\theta(\gamma)$ the rotational angle of a hyperbolic element $\gamma \in \PSL_2(\mathbb C)$ so that $\ell(\gamma) + i\theta(\gamma)$ is its \emph{complex} translation length.

\begin{lemma}
\label{lem:ComplexTranslationLengthOfProduct}
Let $g, h \in \PSL_2(\mathbb C)$ be hyperbolic elements such that $g \in A$ and $h \in QAQ^{-1}$ for some
$
Q =
\left(
\begin{smallmatrix}
a & b \\
c & d
\end{smallmatrix}
\right)
\in \PSL_2(\mathbb C)
$
and $gh$ is also hyperbolic.
Then, we have
\begin{align*}
\cosh\left(\frac{\ell(gh) + i\theta(gh)}{2}\right) ={}&\cosh\left(\frac{\ell(g)}{2}\right) \cosh\left(\frac{\ell(h)}{2}\right) \\
&{}+ (ad + bc) \cdot \sinh\left(\frac{\ell(g)}{2}\right) \sinh\left(\frac{\ell(h)}{2}\right)
\end{align*}
where the equality is to be understood up to sign.
\end{lemma}

\begin{proof}
Let $g, h \in \PSL_2(\mathbb C)$ as in the lemma. Write them explicitly as
\begin{align*}
g &=
\begin{pmatrix}
e^s & 0 \\
0 & e^{-s}
\end{pmatrix};
&
h &=
\begin{pmatrix}
a & b \\
c & d
\end{pmatrix}
\begin{pmatrix}
e^t & 0 \\
0 & e^{-t}
\end{pmatrix}
\begin{pmatrix}
a & b \\
c & d
\end{pmatrix}^{-1}
\end{align*}
for some $s, t \in \mathbb R$. Recalling $ad - bc = 1$, $s = \frac{\ell(g)}{2}$, and $t = \frac{\ell(h)}{2}$, we calculate that
\begin{align*}
\frac{\tr(gh)}{2}
&= ad \cdot \cosh(s + t) - bc \cdot \cosh(s - t) \\
&= \cosh\left(\frac{\ell(g)}{2}\right)\cosh\left(\frac{\ell(h)}{2}\right) + (ad + bc) \cdot \sinh\left(\frac{\ell(g)}{2}\right)\sinh\left(\frac{\ell(h)}{2}\right)
\end{align*}
\emph{up to sign}. The lemma now follows from the formula for the complex translation length $\cosh\left(\frac{\ell(gh) + i\theta(gh)}{2}\right) = \frac{\tr(gh)}{2}$
which also holds \emph{up to sign} \cite[Lemma 12.1.2]{MR03}.
\end{proof}

The following is the key proposition analogous to \cite[Lemma 4.4]{Nau05}. Though its proof is inspired from the latter, there are greater complications arising from the higher dimensional setting.

\begin{proposition}
\label{pro:tauNotCohomologousToLocallyConstantFunction}
The distortion function $\tau$ is not cohomologous to a locally constant function $\phi: \Lambda \to \mathbb R$.
\end{proposition}

\begin{proof}
Suppose that the distortion function $\tau$ is cohomologous to some locally constant function $\phi: \Lambda \to \mathbb R$. We use this to first derive an important identity. By compactness of $\Lambda$, there is a finite cover consisting of cylinders on which $\phi$ is constant. Hence, there exists $p_\phi \in \mathbb N$ such that $\phi$ is constant on any cylinder of length $p_\phi - 1$. In both the Schottky semigroup and the continued fractions semigroup settings, increasing $p_\phi$ and reordering the generators if necessary, there are distinct generators $g_1, g_2 \in \Gamma$ such that for any integers $r_1, r_2 > p_\phi$, we have that $\bigl\{g_1^{r_1}, g_2^{r_2}\bigr\}$ generates a Schottky subsemigroup $\Gamma_0 \subset \Gamma$. Let $\alpha_1 = (1, 1, \dotsc, 1)$ and $\alpha_2 = (2, 2, \dotsc, 2)$ with $\len(\alpha_1) = r_1 - 1$ and $\len(\alpha_2) = r_2 - 1$. Denote $x_1 = g_1^+ \in \Lambda$ and $x_2 = \bigl(g_1^{r_1}g_2^{r_2}\bigr)^+ \in \Lambda$. We have $\tau(x_1) = \phi(x_1)$ and $\tau_{r_1 + r_2}(x_2) = \phi_{r_1 + r_2}(x_2)$. By \cref{pro:ClosedGeodesicsBijectiveWithPeriodicOrbitsOf_T}, we also have $\tau(x_1) = \ell(g_1)$ and $\tau_{r_1 + r_2}(x_2) = \ell\big(g_1^{r_1}g_2^{r_2}\big)$. Thus, by local constancy, $\phi(u) = \ell(g_1)$ for all $u \in \mathtt{C}[\alpha_1]$ and $\phi_{r_1 + r_2}(u) = \ell\big(g_1^{r_1}g_2^{r_2}\big)$ for all $u \in \mathtt{C}[\alpha_1, \alpha_2, \alpha_1]$. Let $p_1, p_2 \in \mathbb N$. Let $g = \big(g_1^{r_1}\big)^{p_1}\big(g_1^{r_1}g_2^{r_2}\big)^{p_2}$ and $z \in \Lambda$ be its attracting fixed point. Let $\lambda_1 = (\alpha_1, \alpha_1, \dotsc, \alpha_1)$ and $\lambda_2 = (\alpha_1, \alpha_2, \alpha_1, \alpha_2, \dotsc, \alpha_1, \alpha_2)$ with $\len(\lambda_1) = p_1r_1 - 1$ and $\len(\lambda_2) = p_2(r_1 + r_2) - 1$. As before, we have both $\tau_{p_1r_1 + p_2(r_1 + r_2)}(z) = \phi_{p_1r_1 + p_2(r_1 + r_2)}(z) = \phi_{p_1r_1}(z) + \phi_{p_2(r_1 + r_2)}(T^{p_1r_1}(z))$ and $\tau_{p_1r_1 + p_2(r_1 + r_2)}(z) = \ell(g)$ which implies that $\ell(g) = \phi_{p_1r_1}(u) + \phi_{p_2(r_1 + r_2)}(T^{p_1r_1}(u))$ for all $u \in \mathtt{C}[\lambda_1, \lambda_2, \alpha_1]$ by local constancy. Thus, $\ell(g) = p_1r_1\ell(g_1) + p_2\ell\big(g_1^{r_1}g_2^{r_2}\big) = \ell\big(\big(g_1^{r_1}\big)^{p_1}\big) + \ell\big(\big(g_1^{r_1}g_2^{r_2}\big)^{p_2}\big)$. Denoting $h_1 = g_1^{r_1}$ and $h_2 = g_1^{r_1}g_2^{r_2}$, we have shown that
\begin{align}
\label{eqn:TranslationLengthOfProductIsSumOfTranslationLengths}
\ell\bigl(h_1^{p_1}h_2^{p_2}\bigr) = \ell\bigl(h_1^{p_1}\bigr) + \ell\bigl(h_2^{p_2}\bigr) \qquad \text{for all $p_1, p_2 \in \mathbb N$}.
\end{align}

We will later use \cref{eqn:TranslationLengthOfProductIsSumOfTranslationLengths} to derive \cref{eqn:ad_IsInACircle}. We will use the upper half space model of $\mathbb H^n$. For all $t \in \mathbb R$, let $a_t \in G$ denote the semisimple element which acts by translation by $t$ in the positive direction along the $e_n$-axis and $A = \{a_t: t \in \mathbb R\}$. Let $K = \Stab_G(e_n) < G$ and $M = \Stab_G(v_{e_n}) < K$ where $v_{e_n} = (e_n, e_n) \in \T_{e_n}(\mathbb H^n)$. Note that $K \cong \SO(n)$ and $M \cong \SO(n - 1)$ where the $M$-action on $\mathbb R^{n - 1} \subset \partial_\infty \mathbb H^n$ coincides with the standard representation of $\SO(n - 1)$. Let $d$ denote the distance function induced from any fixed left $G$-invariant and right $K$-invariant Riemannian metric on $G$ which descends to the hyperbolic metric on $G/K \cong \mathbb H^n$. Denoting $\ell_j = \ell(h_j) > 0$, there are $m_j \in M$ and $Q_j \in G$ such that
\begin{align*}
h_j = Q_ja_{\ell_j}m_j Q_j^{-1} \qquad \text{for all $j \in \{1, 2\}$}.
\end{align*}
For all $j \in \{1, 2\}$, let
\begin{align*}
\gamma_j = Q_ja_{\ell_j} Q_j^{-1}
\end{align*}
which has the same axis and translation length as $h_j$ so that $\ell(\gamma_j) = \ell_j$.

We now obtain some useful estimates. Using Poincar\'{e} recurrence theorem on $M$ and continuity of conjugation, we have that for all $j \in \{1, 2\}$ and $\epsilon > 0$, there exists arbitrarily large $p_j \in \mathbb N$ such that $d\big(Q_jm_j^{p_j}Q_j^{-1}, e\big) < \epsilon$ and hence $d\big(h_j^{p_j}, \gamma_j^{p_j}\big) < \epsilon$. We use this estimate to argue that in fact, for all $\epsilon > 0$, there exist arbitrarily large $p_2 \in \mathbb N$ and arbitrarily large $p_1(p_2) \in \mathbb N$ such that
\begin{align}
\label{eqn:ProductOfPowersEstimate}
d\bigl(h_1^{p_1(p_2)}h_2^{p_2}, \gamma_1^{p_1(p_2)}\gamma_2^{p_2}\bigr) < \epsilon.
\end{align}
Let $\epsilon > 0$. Fix an arbitrarily large $p_2 \in \mathbb N$ such that $d\bigl(h_2^{p_2}, \gamma_2^{p_2}\bigr) < \frac{\epsilon}{2}$. Then, there exists arbitrarily large $p_1(p_2) \in \mathbb N$ such that $d\bigl(h_1^{p_1(p_2)}, \gamma_1^{p_1(p_2)}\bigr)$ is sufficiently small such that $d\bigl(\gamma_2^{p_2}, h_1^{-p_1(p_2)}\gamma_1^{p_1(p_2)}\gamma_2^{p_2}\bigr) < \frac{\epsilon}{2}$ by continuity of conjugation. \Cref{eqn:ProductOfPowersEstimate} now follows by triangle inequality. We now use \cref{eqn:ProductOfPowersEstimate} to argue that for all $\epsilon > 0$, there exist arbitrarily large $p_2 \in \mathbb N$ and arbitrarily large $p_1(p_2) \in \mathbb N$ such that
\begin{align}
\label{eqn:TranslationLengthOfProductOfPowersEstimate}
\bigl|\ell\bigl(h_1^{p_1(p_2)}h_2^{p_2}\bigr) - \ell\bigl(\gamma_1^{p_1(p_2)}\gamma_2^{p_2}\bigr)\bigr| < \epsilon.
\end{align}
Let $\epsilon > 0$. Fix any point $x_0$ on the geodesic with forward and backward endpoints $h_1^+ = \gamma_1^+$ and $h_2^- = \gamma_2^-$. Let $\epsilon' \in \bigl(0, \frac{\epsilon}{6}\bigr)$ and $\epsilon'' > 0$ such that $d(x, gx) < \frac{\epsilon}{3}$ for all $x \in B_{\epsilon'}(x_0)$ and $g \in B_{\epsilon''}(e) \subset G$. Recall that $\{h_1, h_2\}$ generates a Schottky semigroup and hence so does $\{\gamma_1, \gamma_2\}$. Consequently, if $p_1, p_2 \in \mathbb N$ are sufficiently large, then the axes of $h_1^{p_1}h_2^{p_2}$ and $\gamma_1^{p_1}\gamma_2^{p_2}$ will intersect $B_{\epsilon'}(x_0)$. Hence, by \cref{eqn:ProductOfPowersEstimate}, we can fix arbitrarily large $p_2 \in \mathbb N$ and arbitrarily large $p_1(p_2) \in \mathbb N$ such that $d\bigl(h_1^{p_1(p_2)}h_2^{p_2}, \gamma_1^{p_1(p_2)}\gamma_2^{p_2}\bigr) < \epsilon''$ and the axes of $h_1^{p_1(p_2)}h_2^{p_2}$ and $\gamma_1^{p_1(p_2)}\gamma_2^{p_2}$ intersect $B_{\epsilon'}(x_0)$. Suppose that $x_h, x_\gamma \in B_{\epsilon'}(x_0)$ are points on the respective axes. Then, by repeated triangle inequality and symmetry, we have
\begin{align*}
&\bigl|d\bigl(x_h, h_1^{p_1(p_2)}h_2^{p_2} x_h\bigr) - d\bigl(x_\gamma, \gamma_1^{p_1(p_2)}\gamma_2^{p_2} x_\gamma\bigr)\bigr| \\
\leq{}&d(x_h, x_\gamma) + d\bigl(\gamma_1^{p_1(p_2)}\gamma_2^{p_2} x_\gamma, h_1^{p_1(p_2)}h_2^{p_2} x_\gamma\bigr) + d\bigl(h_1^{p_1(p_2)}h_2^{p_2} x_\gamma, h_1^{p_1(p_2)}h_2^{p_2} x_h\bigr) < \epsilon.
\end{align*}
\Cref{eqn:TranslationLengthOfProductOfPowersEstimate} now follows from definitions, $d\bigl(x_h, h_1^{p_1(p_2)}h_2^{p_2} x_h\bigr) = \ell\bigl(h_1^{p_1(p_2)}h_2^{p_2}\bigr)$ and $d\bigl(x_\gamma, \gamma_1^{p_1(p_2)}\gamma_2^{p_2} x_\gamma\bigr) = \ell\bigl(\gamma_1^{p_1(p_2)}\gamma_2^{p_2}\bigr)$. We need one last estimate for hyperbolic functions. Let $\epsilon > 0$ and $x, y, u, v > 0$ with $|x - y| < \epsilon$ and $y = u + v$. Then, using the Taylor series for $\sinh$ about $y$ and the sum of arguments formulas, we have
\begin{align}
\begin{aligned}
\label{eqn:HyperbolicFunctionEstimate}
&\left|\frac{\sinh(x)}{\cosh(u)\cosh(v)} - \frac{\sinh(y)}{\cosh(u)\cosh(v)}\right| = \left|\sum_{j = 1}^\infty \frac{\sinh^{(j)}(y)}{j!\cosh(u)\cosh(v)}(x - y)^j\right| \\
\leq{}&\frac{\cosh(y)}{\cosh(u)\cosh(v)}\sum_{j = 0}^\infty \frac{|x - y|^{1 + 2j}}{(1 + 2j)!} + \frac{\sinh(y)}{\cosh(u)\cosh(v)}\sum_{j = 1}^\infty \frac{|x - y|^{2j}}{(2j)!} \\
<{}&(1 + \tanh(u)\tanh(v))\sum_{j = 0}^\infty \frac{\epsilon^{1 + 2j}}{(1 + 2j)!} + (\tanh(u) + \tanh(v))\sum_{j = 1}^\infty \frac{\epsilon^{2j}}{(2j)!} \\
<{}&\frac{4\epsilon}{1 - \epsilon}.
\end{aligned}
\end{align}

We will now derive \cref{eqn:ad_IsInACircle} using the above tools. Denote the $3$-dimensional hyperbolic submanifold $\mathbb H^3 = \Span\{e_1, e_2, e_n\} \cap \mathbb H^n \subset \mathbb H^n$ with boundary $\partial_\infty \mathbb H^3 = \Span\{e_1, e_2\} \cup \{\infty\} \subset \partial_\infty \mathbb H^n$. Identify $\partial_\infty \mathbb H^3$ with the Riemann sphere $\hat{\mathbb C}$ such that $\Span\{e_1, e_2\} \cong \mathbb C$ is a linear isomorphism identifying $e_1$ with $1$ and $e_2$ with $i$. Applying an isometry in $G$, we can assume that $g_1^+ = h_1^+ = \gamma_1^+ = \infty \in \partial_\infty \mathbb H^3 \cong \hat{\mathbb C}$ and $g_1^- = h_1^- = \gamma_1^- = 0 \in \partial_\infty \mathbb H^3 \cong \hat{\mathbb C}$. Applying an isometry in $M$, we can also assume that $h_2^\pm = \gamma_2^\pm \in \Span\{e_1, e_2\}$ so that their axes are contained in $\mathbb H^3$. Further applying a \emph{unique} isometry in $AM$ preserving $\mathbb H^3$, we can assume that $h_2^- = \gamma_2^- = e_1 \in \partial_\infty \mathbb H^3$ which corresponds to $1 \in \hat{\mathbb C}$. Consequently, $h_2^+ = \gamma_2^+ \in \partial_\infty \mathbb H^3$ and its corresponding element in $\hat{\mathbb C}$ are uniquely determined. We call the above setup the \emph{$(h_1, h_2)$-$\mathbb H^3$-arrangement} and similarly for other pairs of elements in $G$ with no common fixed point. Moreover, it is clear from the forms of $\gamma_1$ and $\gamma_2$ above that they preserve $\mathbb H^3$. The subgroup of $G$ preserving $\mathbb H^3$ is Lie isomorphic to $\PSL_2(\mathbb C)$. Hence, we can view $\gamma_1, \gamma_2, Q_1, Q_2 \in \PSL_2(\mathbb C)$ where in particular $Q_1 = e$. Write
\begin{align*}
Q_2 =
\begin{pmatrix}
a & b \\
c & d
\end{pmatrix}.
\end{align*}
Note that from the above assumption, we have $Q_2(0) = 1$. For convenience, we use tildes on $\ell$ and $\theta$ to denote half of its counterparts. Let $\epsilon = \frac{1}{m}$ for some $m \in \mathbb N$. Fix $\epsilon' \in \left(0, \min\left(\frac{1}{2}, \frac{\epsilon}{2 \cdot 4 \cdot 6 \cdot 2}\right)\right)$. Then, for all $x, y, u, v > 0$ with $|x - y| < \epsilon'$ and $y = u + v$, the estimate in \cref{eqn:HyperbolicFunctionEstimate} is then bounded by $\frac{4\epsilon'}{1 - \epsilon'} < \frac{\epsilon}{6 \cdot 2}$ and hence using the mean value theorem for the square function on $[0, 3]$, we obtain the estimate
\begin{align}
\label{eqn:SquareOfHyperbolicFunctionEstimate}
\left|\frac{\sinh^2(x)}{\cosh^2(u)\cosh^2(v)} - \frac{\sinh^2(y)}{\cosh^2(u)\cosh^2(v)}\right| < \frac{\epsilon}{2}.
\end{align}
Similarly, since $|\tanh'| \leq 1$, we fix $\epsilon'' = \frac{\epsilon}{8(1 + |\Re(bc)|)^2}$ so that for all $x, y > 0$ with $|x - y| < \epsilon''$, we have
\begin{align}
\label{eqn:TanhEstimate}
4(1 + |\Re(bc)|)^2|\tanh(x) - \tanh(y)| < \frac{\epsilon}{2}.
\end{align}
By \cref{eqn:TranslationLengthOfProductOfPowersEstimate}, we can fix a arbitrarily large $p_2 \in \mathbb N$ and then fix a corresponding arbitrarily large $p_1 \in \mathbb N$ such that
\begin{align}
\label{eqn:TranslationLengthOfProductOfPowersEstimateFor1/m}
\bigl|\tilde{\ell}\bigl(h_1^{p_1}h_2^{p_2}\bigr) - \tilde{\ell}\bigl(\gamma_1^{p_1}\gamma_2^{p_2}\bigr)\bigr| < \min(\epsilon', \epsilon'').
\end{align}
Now, the sum of arguments formula gives
\begin{align*}
\cosh\bigl(\tilde{\ell}\bigl(\gamma_1^{p_1}\gamma_2^{p_2}\bigr) + i\tilde{\theta}\bigl(\gamma_1^{p_1}\gamma_2^{p_2}\bigr)\bigr) ={}&\cosh\bigl(\tilde{\ell}\bigl(\gamma_1^{p_1}\gamma_2^{p_2}\bigr)\bigr) \cos\bigl(\tilde{\theta}\bigl(\gamma_1^{p_1}\gamma_2^{p_2}\bigr)\bigr) \\
&{}+ i\sinh\bigl(\tilde{\ell}\bigl(\gamma_1^{p_1}\gamma_2^{p_2}\bigr)\bigr) \sin\bigl(\tilde{\theta}\bigl(\gamma_1^{p_1}\gamma_2^{p_2}\bigr)\bigr)
\end{align*}
while \cref{lem:ComplexTranslationLengthOfProduct} gives
\begin{align*}
\cosh\bigl(\tilde{\ell}\bigl(\gamma_1^{p_1}\gamma_2^{p_2}\bigr) + i\tilde{\theta}\bigl(\gamma_1^{p_1}\gamma_2^{p_2}\bigr)\bigr) ={}&\cosh\bigl(p_1\tilde{\ell}_1\bigr) \cosh\bigl(p_2\tilde{\ell}_2\bigr) \\
&{}+ (ad + bc) \cdot \sinh\bigl(p_1\tilde{\ell}_1\bigr) \sinh\bigl(p_2\tilde{\ell}_2\bigr)
\end{align*}
up to sign. Comparing the real and imaginary parts in the above equations give
\begin{align*}
\cosh\bigl(\tilde{\ell}\bigl(\gamma_1^{p_1}\gamma_2^{p_2}\bigr)\bigr) \cos\bigl(\tilde{\theta}\bigl(\gamma_1^{p_1}\gamma_2^{p_2}\bigr)\bigr) ={}&\cosh\bigl(p_1\tilde{\ell}_1\bigr) \cosh\bigl(p_2\tilde{\ell}_2\bigr) \\
&{}+ (1 + 2\Re(bc)) \cdot \sinh\bigl(p_1\tilde{\ell}_1\bigr) \sinh\bigl(p_2\tilde{\ell}_2\bigr); \\
\sinh\bigl(\tilde{\ell}\bigl(\gamma_1^{p_1}\gamma_2^{p_2}\bigr)\bigr) \sin\bigl(\tilde{\theta}\bigl(\gamma_1^{p_1}\gamma_2^{p_2}\bigr)\bigr) ={}& 2\Im(bc) \cdot \sinh\bigl(p_1\tilde{\ell}_1\bigr) \sinh\bigl(p_2\tilde{\ell}_2\bigr)
\end{align*}
up to sign. Using $\sin^2\bigl(\tilde{\theta}\bigl(\gamma_1^{p_1}\gamma_2^{p_2}\bigr)\bigr) + \cos^2\bigl(\tilde{\theta}\bigl(\gamma_1^{p_1}\gamma_2^{p_2}\bigr)\bigr) = 1$ with the above pair of equations, the sign ambiguity is removed and we get
\begin{align*}
\sinh^2\bigl(\tilde{\ell}\bigl(\gamma_1^{p_1}\gamma_2^{p_2}\bigr)\bigr) = {}&\big(\cosh\bigl(p_1\tilde{\ell}_1\bigr) \cosh\bigl(p_2\tilde{\ell}_2\bigr) + (1 + 2\Re(bc)) \cdot \sinh\bigl(p_1\tilde{\ell}_1\bigr) \sinh\bigl(p_2\tilde{\ell}_2\bigr)\big)^2 \\
&{}\cdot \tanh^2\bigl(\tilde{\ell}\bigl(\gamma_1^{p_1}\gamma_2^{p_2}\bigr)\bigr) + 4\Im(bc)^2 \cdot \sinh^2\bigl(p_1\tilde{\ell}_1\bigr) \sinh^2\bigl(p_2\tilde{\ell}_2\bigr).
\end{align*}
Dividing through by $\cosh^2\bigl(p_1\tilde{\ell}_1\bigr) \cosh^2\bigl(p_2\tilde{\ell}_2\bigr)$ and recalling the identity in \cref{eqn:TranslationLengthOfProductIsSumOfTranslationLengths}, the estimates from \cref{eqn:SquareOfHyperbolicFunctionEstimate,eqn:TanhEstimate,eqn:TranslationLengthOfProductOfPowersEstimateFor1/m}, and $\epsilon = \frac{1}{m}$, we get
\begin{align}
\begin{aligned}
\label{eqn:MainInequalityFromTranslationLengthFormulas}
&\bigl|\big(1 + (1 + 2\Re(bc)) \cdot \tanh\bigl(p_1\tilde{\ell}_1\bigr) \tanh\bigl(p_2\tilde{\ell}_2\bigr)\big)^2 \tanh^2\bigl(p_1\tilde{\ell}_1 + p_2\tilde{\ell}_2\bigr) \\
{}&+ 4\Im(bc)^2 \cdot \tanh^2\bigl(p_1\tilde{\ell}_1\bigr) \tanh^2\bigl(p_2\tilde{\ell}_2\bigr) - \bigl(\tanh\bigl(p_1\tilde{\ell}_1\bigr) + \tanh\bigl(p_2\tilde{\ell}_2\bigr)\bigr)^2\bigr| < \frac{1}{m}.
\end{aligned}
\end{align}
Thus, for all $m \in \mathbb N$, we have shown \cref{eqn:MainInequalityFromTranslationLengthFormulas} for any arbitrarily large $p_2 \in \mathbb N$ and any corresponding arbitrarily large $p_1 \in \mathbb N$ which satisfies \cref{eqn:TranslationLengthOfProductOfPowersEstimateFor1/m}. So there exist sequences $\{p_{1, m}\}_{m \in \mathbb N}, \{p_{2, m}\}_{m \in \mathbb N} \subset \mathbb N$ with $\lim_{m \to \infty}p_{1, m} = \lim_{m \to \infty}p_{2, m} = \infty$ such that for all $m \in \mathbb N$, the integers $p_{1, m}$ and $p_{2, m}$ satisfy \cref{eqn:MainInequalityFromTranslationLengthFormulas}. Taking the limit $m \to \infty$ in \cref{eqn:MainInequalityFromTranslationLengthFormulas} and simplifying the result gives the equation of a circle
\begin{align}
\label{eqn:ad_IsInACircle}
|ad| = |bc + 1| = 1.
\end{align}

Denote $z = ad$ so that $|z| = 1$. Then we have
\begin{align*}
Q_2 =
\begin{pmatrix}
a & b \\
\frac{z - 1}{b} & \frac{z}{a}
\end{pmatrix}.
\end{align*}
Now, $Q_2(0) = 1$ implies $ab = z$ and so $h_2^+ = Q_2(\infty) = \frac{z}{z - 1}$. Also, $|z| = 1$ if and only if $\frac{z}{z - 1} \in \bigl\{\xi \in \mathbb C: \Re(\xi) = \frac{1}{2}\bigr\} \cup \{\infty\}$. But $h_2^+ \neq h_1^+ = \infty$. Thus, we have shown using the hypothesis of this proposition that for all integers $r_1, r_2 > p_\phi$, we have
\begin{align}
\label{eqn:h_2^+IsOnTheLineRealPartEquals1/2}
\bigl(g_1^{r_1}g_2^{r_2}\bigr)^+ \in \left\{\xi \in \mathbb C: \Re(\xi) = \frac{1}{2}\right\}
\end{align}
in the $\bigl(g_1, g_1^{r_1}g_2^{r_2}\bigr)$-$\mathbb H^3$-arrangement.

We show that the above property is impossible. Denote by $B_r^{\mathrm{E}}(x) \subset \mathbb R^{ n - 1} \subset \partial_\infty \mathbb H^n$ (resp. $B_r^{\mathbb C}(x) \subset \mathbb C \subset \hat{\mathbb C}$) the open Euclidean ball of radius $r > 0$ centered at $x \in \mathbb R^{ n - 1}$ (resp. $x \in \mathbb C$). We start in the $(g_1, g_2)$-$\mathbb H^3$-arrangement. Let $g_1 = a_{\ell(g_1)}m_{g_1}$ for some $m_{g_1} \in M$ and $\tilde{D}_2 \ni g_2^+$ be the ball for the Schottky generating set $\bigl\{g_1^{r_1}, g_2^{r_2}\bigr\}$. Given $\epsilon_1 > 0$, there exists an integer $R_2 > p_\phi$ such that for all $r_1 \in \mathbb N$ and integers $r_2 > R_2$, we have $\bigl(g_1^{r_1}g_2^{r_2}\bigr)^- \in B_{\epsilon_1}^{\mathrm{E}}(g_2^-) = B_{\epsilon_1}^{\mathrm{E}}(e_1) \subset \mathbb R^{n - 1} \subset \partial_\infty \mathbb H^n$ and $g_2^{r_2 - 1} \cdot \tilde{D}_2 \subset B_{\epsilon_1}^{\mathrm{E}}\bigl(g_2^+\bigr)$. Given $\epsilon_2 > 0$, we take $\epsilon_1$ sufficiently small such that there exist $\epsilon' > 0$ and $am \in AM \cap B_{\epsilon'}(e)$ such that $am \cdot \bigl(g_1^{r_1}g_2^{r_2}\bigr)^- = e_1$  for all $r_1 \in \mathbb N$ and integers $r_2 > R_2$, and $a'm' \cdot B_{\epsilon_1}^{\mathrm{E}}\bigl(g_2^+\bigr) \subset B_{\epsilon_2}^{\mathrm{E}}\bigl(g_2^+\bigr)$ for all $a'm' \in AM \cap B_{\epsilon'}(e)$. We now consider the following two cases, each of which leads to a contradiction.

\medskip
\noindent
\textit{Case 1.} Suppose $\bigl\{m_{g_1}^j \cdot g_2^+\bigr\}_{j \in \mathbb Z_{\geq 0}} \not\subset \Span\{e_2, e_3, \dotsc, e_{n - 1}\}$. Fix $\tilde{g}_2^+ = m_{g_1}^{j_1} \cdot g_2^+$ for some $j_1 \in \mathbb{Z}_{\geq 0}$ and $\epsilon > 0$ such that $B_{2\epsilon}^{\mathrm{E}}\bigl(\tilde{g}_2^+\bigr) \cap \Span\{e_2, e_3, \dotsc, e_{n - 1}\} = \varnothing$. Take $\epsilon_2 = \frac{\epsilon}{2}$ and obtain $\epsilon_1$, $\epsilon'$, and $R_2$ as above. Also let $r_2 > R_2$ be an integer and $am \in AM \cap B_{\epsilon'}(e)$ as above. By Poincar\'{e} recurrence theorem on $M$, take an arbitrarily large integer $r_1 > p_\phi$ such that $m_{g_1}^{r_1} \cdot B_{\epsilon_2}^{\mathrm{E}}\bigl(g_2^+\bigr) \subset B_\epsilon^{\mathrm{E}}\bigl(\tilde{g}_2^+\bigr)$. By the Brouwer fixed point theorem, $\bigl(g_1^{r_1}g_2^{r_2}\bigr)^+ \in g_1^{r_1}g_2^{r_2 - 1} \cdot \tilde{D}_2 \subset a_{\ell(g_1)}^{r_1} m_{g_1}^{r_1} \cdot B_{\epsilon_1}^{\mathrm{E}}\bigl(g_2^+\bigr)$ and hence $am \cdot \bigl(g_1^{r_1}g_2^{r_2}\bigr)^+ \in a_{\ell(g_1)}^{r_1} m_{g_1}^{r_1} am' \cdot B_{\epsilon_1}^{\mathrm{E}}\bigl(g_2^+\bigr) \subset a_{\ell(g_1)}^{r_1} m_{g_1}^{r_1} \cdot B_{\epsilon_2}^{\mathrm{E}}\bigl(g_2^+\bigr) \subset a_{\ell(g_1)}^{r_1} \cdot B_\epsilon^{\mathrm{E}}\bigl(\tilde{g}_2^+\bigr)$ where $m' = m_{g_1}^{-r_1}mm_{g_1}^{r_1} \in M \cap B_{\epsilon'}(e)$. There exists $\tilde{m} \in M$ such that $\tilde{m} \cdot e_1 = e_1$ and $a\tilde{m}m \cdot \bigl(g_1^{r_1}g_2^{r_2}\bigr)^+ \in \partial_\infty \mathbb H^3$. Now, $a\tilde{m}m \cdot \bigl(g_1^{r_1}g_2^{r_2}\bigr)^+ \in a_{\ell(g_1)}^{r_1} \cdot B_\epsilon^{\mathrm{E}}\bigl(\tilde{m} \cdot \tilde{g}_2^+\bigr)$. Denoting $\tilde{z}_2 \in \mathbb C \subset \hat{\mathbb C}$ to be the point corresponding to $\tilde{m} \cdot \tilde{g}_2^+ \in \partial_\infty \mathbb H^3$, we have $|\Re(\tilde{z}_2)| = \bigl|\bigl\langle \tilde{m} \cdot \tilde{g}_2^+, e_1\bigr\rangle\bigr| = \bigl|\bigl\langle \tilde{g}_2^+, e_1\bigr\rangle\bigr| \geq 2\epsilon$. Note that $a\tilde{m}m$ is an isometry from the $(g_1, g_2)$-$\mathbb H^3$-arrangement to the $\bigl(g_1, g_1^{r_1}g_2^{r_2}\bigr)$-$\mathbb H^3$-arrangement. Thus, $\bigl(g_1^{r_1}g_2^{r_2}\bigr)^+ \in a_{\ell(g_1)}^{r_1} \cdot B_\epsilon^{\mathbb C}(\tilde{z}_2)$ in the $\bigl(g_1, g_1^{r_1}g_2^{r_2}\bigr)$-$\mathbb H^3$-arrangement. If $r_1$ is sufficiently large, we obtain $\bigl|\Re\bigl(\bigl(g_1^{r_1}g_2^{r_2}\bigr)^+\bigr)\bigr| > e^{r_1\ell(g_1)} (|\Re(\tilde{z}_2)| - \epsilon) \geq e^{r_1\ell(g_1)} \epsilon > \frac{1}{2}$ contradicting \cref{eqn:h_2^+IsOnTheLineRealPartEquals1/2}.

\medskip
\noindent
\textit{Case 2.} Suppose $\bigl\{m_{g_1}^j \cdot g_2^+\bigr\}_{j \in \mathbb Z_{\geq 0}} \subset \Span\{e_2, e_3, \dotsc, e_{n - 1}\}$. Fix any integer $r_1 > p_\phi$ and $\tilde{g}_2^+ = m_{g_1}^{r_1} \cdot g_2^+$. Take $\epsilon = \epsilon_2 = \frac{1}{4e^{r_1\ell(g_1)}}$ and obtain $\epsilon_1$, $\epsilon'$, and $R_2$ as above. Also let $r_2 > R_2$ be an integer and $am \in AM \cap B_{\epsilon'}(e)$ as above. Proceeding as in Case 1, there exists $\tilde{m} \in M$ such that $\tilde{m} \cdot e_1 = e_1$ and $a\tilde{m}m \cdot \bigl(g_1^{r_1}g_2^{r_2}\bigr)^+ \in \partial_\infty \mathbb H^3$. Then, $a\tilde{m}m \cdot \bigl(g_1^{r_1}g_2^{r_2}\bigr)^+ \in a_{\ell(g_1)}^{r_1} \cdot B_{\epsilon_2}^{\mathrm{E}}\bigl(\tilde{m} \cdot \tilde{g}_2^+\bigr)$. Denoting $\tilde{z}_2 \in \mathbb C \subset \hat{\mathbb C}$ to be the point corresponding to $\tilde{m} \cdot \tilde{g}_2^+ \in \partial_\infty \mathbb H^3$, we have $|\Re(\tilde{z}_2)| = |\langle \tilde{m} \cdot \tilde{g}_2^+, e_1\rangle| = |\langle \tilde{g}_2^+, e_1\rangle| = 0$. Again, $a\tilde{m}m$ is an isometry from the $(g_1, g_2)$-$\mathbb H^3$-arrangement to the $\bigl(g_1, g_1^{r_1}g_2^{r_2}\bigr)$-$\mathbb H^3$-arrangement. Thus, $\bigl(g_1^{r_1}g_2^{r_2}\bigr)^+ \in a_{\ell(g_1)}^{r_1} \cdot B_\epsilon^{\mathbb C}(\tilde{z}_2)$ in the $\bigl(g_1, g_1^{r_1}g_2^{r_2}\bigr)$-$\mathbb H^3$-arrangement. We obtain $\bigl|\Re\bigl(\bigl(g_1^{r_1}g_2^{r_2}\bigr)^+\bigr)\bigr| < e^{r_1\ell(g_1)} \epsilon = \frac{1}{4} < \frac{1}{2}$ contradicting \cref{eqn:h_2^+IsOnTheLineRealPartEquals1/2}.
\end{proof}

\begin{remark}
The above proof is greatly simplified if, after passing to some power, $h_1$ preserves $\mathbb H^3$ in the $(h_1, h_2)$-$\mathbb H^3$-arrangement, for example if $n = 3$ or $m_1$ is a torsion element. In this case, we can set $\gamma_1 = h_1$ and $p_1$ does not depend on $p_2$ in \cref{eqn:ProductOfPowersEstimate} and so from \cref{eqn:MainInequalityFromTranslationLengthFormulas} we get the stronger condition that $(|bc|^2 + \Re(bc)) X + \Re(bc) \in \mathbb R[X]$ has infinitely many zeros $X_{p_1} = \tanh\bigl(p_1\tilde{\ell}_1\bigr)$ for all $p_1 \in \mathbb N$. Thus, $b = 0$ or $c = 0$ contradicting the fact that $h_1$ and $h_2$ cannot have any common fixed point.
\end{remark}

\Cref{pro:varphiSatisfiesLNIC} now follows by adapting Naud's argument after \cite[Lemma 4.4]{Nau05} and replacing the latter with \cref{pro:tauNotCohomologousToLocallyConstantFunction}. Note that \cite[Lemma 4.3]{Nau05} continues to hold in our setting. To adapt his analyticity argument for the proof of LNIC to higher dimensions, we simply need to use the Grauert tube (see the explanation after \cite[Lemma 3.1]{GLZ04}) and \cite[Proposition 3.12]{Win15}.

\begin{proposition}
\label{pro:varphiSatisfiesLNIC}
The distortion function $\varphi$ satisfies LNIC.
\end{proposition}

We now prove the non-concentration property (NCP). Recall that the semigroup $\Gamma$ preserves $\Hull(\Lambda)$, the convex hull of $\Lambda$ in $\mathbb H^n$. Let $D_j^0$ simply denote $D_j$ in the Schottky semigroup setting and the untrimmed disks corresponding to $D_j$ in the continued fractions semigroup setting. Denote by $\tilde{D}_j^0 \subset \overline{\mathbb H^n}$ the closed Euclidean half ball over $D_j^0$. Since $\Lambda \subset \bigsqcup_{j \in \mathcal{A}} \interior(D_j^0)$, so $F := \Hull(\Lambda) \setminus \bigcup_{j \in \mathcal{A}} \interior(\tilde{D}_j^0) \subset \mathbb H^n$ is compact. Thus, we obtain the following lemma.

\begin{lemma}
\label{lem:GammaCocompactActionOnHull}
The semigroup $\Gamma$ acts cocompactly on $\Hull(\Lambda)$.
\end{lemma}

\begin{proposition}[NCP]
\label{pro:NonConcentrationProperty}
There exists $\delta \in (0, 1)$ such that for all $x \in \Lambda$, cylinders $\mathtt{C} \subset \Lambda$ containing $x$, and $w \in \mathbb R^{n - 1}$ with $\|w\| = 1$, there exists $y \in \mathtt{C} \setminus B_{\diam(\mathtt{C})/4}^{\mathrm{E}}(x)$ such that $|\langle y - x, w \rangle| \geq \delta \diam(\mathtt{C})$.
\end{proposition}

\begin{proof}
Since $\mathcal{A}$ is finite, it suffices to prove the proposition with $\Lambda$ replaced by $\mathtt{C}[k]$ for any $k \in \mathcal{A}$. By way of contradiction, suppose the proposition is false. Then for all $j \in \mathbb N$, taking $\delta_j = \frac{1}{j}$, there exist $x_j \in \mathtt{C}[k]$, cylinder $\mathtt{C}_j \subset \mathtt{C}[k]$ containing $x_j$, and $w_j \in \mathbb R^{n - 1}$ with $\|w_j\| = 1$ such that $|\langle y - x_j, w_j \rangle| \leq \delta_j \diam(\mathtt{C}_j) = \frac{\diam(\mathtt{C}_j)}{j}$ for all $y \in \mathtt{C}_j \setminus B_{\diam(\mathtt{C}_j)/4}^{\mathrm{E}}(x_j)$. Hence, we can rewrite this as
\begin{align}
\label{eqn:IfLemmaIsFalse}
\mathtt{C}_j \setminus B_{\diam(\mathtt{C}_j)/4}^{\mathrm{E}}(x_j) \subset \left\{y \in \mathbb R^{n - 1}: |\langle y - x_j, w_j \rangle| \leq \frac{\diam(\mathtt{C}_j)}{j}\right\} \qquad \text{for all $j \in \mathbb N$}.
\end{align}

We want to use the self-similarity property of the fractal set $\Lambda$. Recall the subgroups $A = \{a_t: t \in \mathbb R\} < G$ and $N^- = \{n_x^-: x \in \mathbb R^{n - 1}\} < G$ whose elements act on $\overline{\mathbb H^n}$ by dilation by $e^t$ and by translation by $x$, respectively. Recall that we may view $G$ as the oriented orthonormal frame bundle. Note that $\infty$ and $x$ are the forward and backward endpoints of the frame $n_x^-$, respectively. Let $N^+ < G$ be the opposite horospherical subgroup. Fix some $x_0 \in \Lambda \setminus \mathtt{C}[k]$. For all $j \in \mathbb N$, there exists $n_j^+ \in N^+$ such that $x_0$ is the forward endpoint of the frame $n_{x_j}^-n_j^+$ and hence there exists $t_0 \in \mathbb R$ such that $n_{x_j}^-n_j^+a_{t_0}o \in F$. For all $j \in \mathbb N$, take $\gamma_j \in \Gamma$ to be the word which corresponds to $\mathtt{C}_j$, i.e., $\gamma_j^{-1}\mathtt{C}_j = \mathtt{C}[k]$ and $t_j > t_0$ such that $n_{x_j}^-n_j^+a_{-t_j}o \in \gamma_j F$. Compactness of $\mathtt{C}[k]$ implies $\{a_{t_j}n_j^+a_{-t_j}\}_{j \in \mathbb N} \subset G$ is bounded and so together with \cref{lem:GammaCocompactActionOnHull}, there exists some fixed compact set $\Omega \subset G$ such that $n_{x_j}^-a_{-t_j} = \gamma_j h_j$ with $h_j \in \Omega$ for all $j \in \mathbb N$. Thus, applying $\gamma_j^{-1} = h_ja_{t_j}n_{-x_j}^-$ in \cref{eqn:IfLemmaIsFalse} gives
\begin{align}
\mathtt{C}[k] \setminus h_jB_{t_j\diam(\mathtt{C}_j)/4}^{\mathrm{E}}(0) \subset h_j\left\{y \in \mathbb R^{n - 1}: |\langle y, w_j \rangle| \leq \frac{t_j \diam(\mathtt{C}_j)}{j}\right\}
\end{align}
for all $j \in \mathbb N$. Of course $t_j \diam(\mathtt{C}_j)/4 < \diam(\mathtt{C}[k])$ for all $j \in \mathbb N$. Now, using \cref{lem:CylinderDiameterBound}, fix $p_1 \in \mathbb N$ such that $\rho^{p_1} \leq \frac{1}{4}$ and a subcylinder $\mathtt{C}_j' \subset \mathtt{C}_j \cap B_{\diam(\mathtt{C}_j)/4}^{\mathrm{E}}(x_j)$ containing $x_j$ of length $\len(\mathtt{C}_j') = \len(\mathtt{C}_j) + p_0p_1$ for all $j \in \mathbb N$. Then the subcylinder $\gamma_j^{-1}\mathtt{C}_j' \subset \mathtt{C}[k] \cap h_jB_{t_j\diam(\mathtt{C}_j)/4}^{\mathrm{E}}(0)$ is of length $\len(\gamma_j^{-1}\mathtt{C}_j') = p_0p_1$ and so \cref{lem:CylinderDiameterBound} again gives the lower bound $t_j\diam(\mathtt{C}_j)/2 \geq \diam(\gamma_j^{-1}\mathtt{C}_j') \geq \rho^{p_0p_1}\diam(\mathtt{C}[k]) =: 2R$ for all $j \in \mathbb N$. By compactness, we can pass to subsequences such that $\lim_{j \to \infty} w_j = w \in \mathbb R^{n - 1}$ with $\|w\| = 1$ and $\lim_{j \to \infty} h_j = h \in \Omega$. So in the limit $j \to \infty$, we have $\mathtt{C}[k] \setminus hB_R^{\mathrm{E}}(0) \subset h\left\{y \in \mathbb R^{n - 1}: \langle y, w \rangle = 0\right\}$ which contradicts \cite[Proposition 3.12]{Win15} since $\Gamma < G$ is Zariski dense.
\end{proof}

Fix $\delta_1$ to be the $\delta$ provided by \cref{pro:NonConcentrationProperty} henceforth.

The following crucial corollary of \cref{pro:varphiSatisfiesLNIC} is derived similar to \cite[Proposition 5.5]{Nau05}.

\begin{corollary}
\label{cor:LNIC_Corollary}
There exist $\delta > 0$, $C > 0$, $\epsilon > 0$, and $m_0 \in \mathbb N$ such that for all $m > m_0$, there exist distinct sections $v_1, v_2: D \to D$ of $T^m$, i.e., $T^m \circ v_j = \Id_D$ for all $j \in \{1, 2\}$, such that
\begin{enumerate}
\item $\delta \leq \|\nabla(\tau_m \circ v_1 - \tau_m \circ v_2)(u)\| \leq C$ for all $u \in D$;
\item for all $u \in D$ and $u' \in D \cap B_\epsilon^{\mathrm{E}}(u)$, if $\left|\left\langle \frac{\nabla(\tau_m \circ v_1 - \tau_m \circ v_2)(u)}{\|\nabla(\tau_m \circ v_1 - \tau_m \circ v_2)(u)\|}, \omega \right\rangle\right| \geq \frac{\delta_1}{2}$ for some $\omega \in \mathbb R^{n - 1}$ with $\|\omega\| = 1$, then $\left|\left\langle \frac{\nabla(\tau_m \circ v_1 - \tau_m \circ v_2)(u')}{\|\nabla(\tau_m \circ v_1 - \tau_m \circ v_2)(u')\|}, \omega \right\rangle\right| \geq \frac{\delta_1}{2}$.
\end{enumerate}
\end{corollary}

Fix $\delta_1'$, $\epsilon_0$, and $m_0$ to be the $\delta$, $\epsilon$, and $m_0$ provided by \cref{cor:LNIC_Corollary}  and $\delta_0 = \delta_1\delta_1'$ henceforth.

\section{Construction of Dolgopyat operators}
\label{sec:ConstructionOfDolgopyatOperators}
The goal of this section is to construct the required Dolgopyat operators. We will start by recording some lemmas.

The following lemma (see \cite[Proposition 3.6]{OW16}) can be derived using the bounds from \cref{lem:Hyperbolicity} as in \cite[Proposition 3.3]{Sto11}.

\begin{lemma}
\label{lem:CylinderDiameterBound}
There exist $p_0 \in \mathbb N$, $\rho \in (0, 1)$ such that for all $l \in \mathbb Z_{\geq 0}$, for all cylinders $\mathtt{C} \subset \Lambda$ with $\len(\mathtt{C}) = l$, for all subcylinders $\mathtt{C}', \mathtt{C}'' \subset \mathtt{C}$ with $\len(\mathtt{C}') = l + 1$ and $\len(\mathtt{C}'') = l + p_0$ respectively, we have $\diam(\mathtt{C}'') \leq \rho\diam(\mathtt{C}) \leq \diam(\mathtt{C}')$.
\end{lemma}

Fix $p_0$ and $\rho$ provided by \cref{lem:CylinderDiameterBound} henceforth. Fix $p_1 \in \mathbb N$ such that
\begin{align}
\label{eqn:Constantp1}
\rho^{p_1 - 1} \leq \frac{\delta_0c_0(\kappa_2 - 1)}{64T_0}.
\end{align}

Now we state a Lasota--Yorke \cite{LY73} type of lemma which will be useful. It is proved similar to \cite[Lemma 3.9]{OW16} and \cite[Lemma 7.3]{SW21}. Note that from those proofs, it is clear that we can take any $A_0 > 2c_0^{-1}e^{\frac{T_0}{c_0(\kappa_2 - 1)}}\max\bigl(1, \frac{T_0}{\kappa_2 - 1}\bigr)$.

\begin{lemma}
\label{lem:PreliminaryLogLipschitz}
There exists $A_0 > 0$ such that for all $\xi \in \mathbb C$ with $|a| < a_0'$ and $|b| > 1$, for all nonzero $q \in \mathcal{O}$, for all $k \in \mathbb N$, we have
\begin{enumerate}
\item\label{itm:PreliminaryLogLipschitzProperty1}	if $h \in \mathcal{C}_B(\Lambda)$ (resp. $h \in \tilde{\mathcal{C}}_B(\Lambda)$) for some $B > 0$, then $L_a^k(h) \in \mathcal{C}_{B'}(\Lambda)$ (resp. $L_a^k(h) \in \tilde{\mathcal{C}}_{B'}(\Lambda)$) for $B' = A_0\left(\frac{B}{\kappa_2^k} + 1\right)$;
\item\label{itm:PreliminaryLogLipschitzProperty2}	if $H \in C(\Lambda, L^2(\tilde{\mathbf{G}}_q))$ and $h \in B(\Lambda, \mathbb R)$ satisfy
\begin{align*}
\|H(u) - H(u')\|_2 \leq Bh(u)d(u, u')
\end{align*}
for all $u, u' \in \Lambda$, for some $B > 0$, then
\begin{align*}
\left\|M_{\xi, q}^k(H)(u) -M_{\xi, q}^k(H)(u')\right\|_2 \leq A_0\left(\frac{B}{\kappa_2^k}L_a^k(h)(u) + |b|L_a^k\|H\|(u)\right)d(u, u')
\end{align*}
for all $u, u' \in \Lambda$.
\end{enumerate}
\end{lemma}

Fix $A_0$ provided by \cref{lem:PreliminaryLogLipschitz}. Fix positive constants
\begin{align}
\label{eqn:Constantb0}
b_0 &= 1; \\
\label{eqn:ConstantE}
E &> \max(1, 2A_0); \\
\label{eqn:Constantepsilon1}
\epsilon_1 &< \min\left(\check{D}, \frac{\epsilon_0}{2}, \frac{\pi c_0(\kappa_2 - 1)}{2T_0}\right); \\
\label{eqn:Constantm}
m &> m_0 \text{ such that } \kappa_2^m > \max\left(8A_0, \frac{4E\rho^{p_1}\epsilon_1}{c_0}, \frac{4 \cdot 128E}{c_0 \delta_0 \rho}\right); \\
\label{eqn:Constantmu}
\mu &< \min\left(\frac{2E\epsilon_1 c_0 \rho^{p_0p_1 + 1}}{\kappa_1^m}, \frac{1}{4}, \frac{1}{16 \cdot 16e^{2mT_0}}\left(\frac{\delta_0\rho\epsilon_1}{64}\right)^2\right).
\end{align}
Having fixed $m$, we can now fix $v_1$ and $v_2$ to be the corresponding distinct sections of $T^m$ provided by \cref{cor:LNIC_Corollary}.

For all $|b| > b_0$, we define $\{\mathtt{C}_1(b), \mathtt{C}_2(b), \dotsc, \mathtt{C}_{c_b}(b)\}$ for some $c_b \in \mathbb N$ to be the set of maximal cylinders $\mathtt{C} \subset \Lambda$ with $\diam(\mathtt{C}) \leq \frac{\epsilon_1}{|b|}$ so that $\Lambda = \bigcup_{l = 1}^{c_b} \mathtt{C}_l(b)$. By \cref{lem:Hyperbolicity} and \cref{eqn:Constantepsilon1}, these cylinders are of length at least $1$. Also, from definitions, the maps $\mathtt{c}^m \circ v_j|_{\mathtt{C}_l(b)}: \mathtt{C}_l(b) \to \tilde{\Gamma}$ are constant for all $|b| > b_0$, $1 \leq l \leq c_b$, and $j \in \{1, 2\}$.

Let $|b| > b_0$. Define $\mathtt{c}_{l, j}(b)$ to be the constant image of $\mathtt{c}^m \circ v_j|_{\mathtt{C}_l(b)}$, for all $1 \leq l \leq c_b$ and $j \in \{1, 2\}$. We define $\{\mathtt{D}_1(b), \mathtt{D}_2(b), \dotsc, \mathtt{D}_{d_b}(b)\}$ for some $d_b \in \mathbb N$ to be the set of subcylinders $\mathtt{D} \subset \mathtt{C}_l(b)$ with $\len(\mathtt{D}) = \len(\mathtt{C}_l(b)) + p_0p_1$ for some integer $1 \leq l \leq c_b$. We say $\mathtt{D}_k(b)$ and $\mathtt{D}_{k'}(b)$ are \emph{adjacent} if $\mathtt{D}_k(b), \mathtt{D}_{k'}(b) \subset \mathtt{C}_l(b)$ for some integer $1 \leq l \leq c_b$. We define $\Xi(b) = \{1, 2\} \times \{1, 2, \dotsc, d_b\}$ and $\mathtt{X}_{j, k}(b) = v_j(\mathtt{D}_k(b))$ for all $(j, k) \in \Xi(b)$. Note that $\mathtt{X}_{j, k}(b) \cap \mathtt{X}_{j', k'}(b) = \varnothing$ since $v_1$ and $v_2$ are distinct sections of $T^m$, for all $(j, k), (j', k') \in \Xi(b)$ with $(j, k) \neq (j', k')$. For all $J \subset \Xi(b)$, we define the function $\beta_J = \chi_{U} - \mu \sum_{(j, k) \in J} \chi_{\mathtt{X}_{j, k}(b)}$, and it can be checked that in fact $\beta_J \in L_d(\Lambda, \mathbb R)$.

Let $|b| > b_0$. Here we record a number of basic facts derived from \cref{lem:Hyperbolicity,lem:CylinderDiameterBound}. For all integers $1 \leq l \leq c_b$ and $(j, k) \in \Xi(b)$, we have the diameter bounds
\begin{gather}
\label{eqn:DiameterBoundC_l}
\rho\frac{\epsilon_1}{|b|} \leq \diam(\mathtt{C}_l(b)) \leq \frac{\epsilon_1}{|b|}; \\
\label{eqn:DiameterBoundD_k}
\rho^{p_0p_1 + 1}\frac{\epsilon_1}{|b|} \leq \diam(\mathtt{D}_k(b)) \leq \rho^{p_1}\frac{\epsilon_1}{|b|}; \\
\label{eqn:DiameterBoundX_jk}
\frac{\epsilon_1c_0\rho^{p_0p_1 + 1}}{|b|\kappa_1^m} \leq \diam(\mathtt{X}_{j, k}^\ell(b)) \leq \frac{\epsilon_1\rho^{p_1}}{|b|c_0\kappa_2^m}.
\end{gather}
For all $J \subset \Xi(b)$, we have $\beta_J \in L_d(\Lambda, \mathbb R)$ with $d$-Lipschitz constant
\begin{align}
\label{eqn:LipschitzConstantbeta_J}
\Lip_d(\beta_J) \leq \frac{\mu}{\min_{(j, k) \in J} \diam(\mathtt{X}_{j, k}(b))} \leq \frac{\mu |b|\kappa_1^m}{\epsilon_1c_0\rho^{p_0p_1 + 1}}.
\end{align}

\begin{definition}
For all $\xi \in \mathbb C$ with $|a| < a_0'$ and $|b| > b_0$, and $J \subset \Xi(b)$, we define the \emph{Dolgopyat operators} $\mathcal{N}_{a, J}: L_d(\Lambda, \mathbb R) \to L_d(\Lambda, \mathbb R)$ by
\begin{align*}
\mathcal{N}_{a, J}(h) = L_{a}^m(\beta_J h) \qquad \text{for all $h \in L_d(\Lambda, \mathbb R)$}.
\end{align*}
\end{definition}

\begin{definition}
For all $|b| > b_0$, a subset $J \subset \Xi(b)$ is said to be \emph{dense} if for all integers $1 \leq l \leq c_b$, there exists $(j, k) \in J$ such that $\mathtt{D}_k(b) \subset \mathtt{C}_l(b)$.
\end{definition}

For all $|b| > b_0$,  define $\mathcal{J}(b) = \{J \subset \Xi(b): J \text{ is dense}\}$.

\section{Proof of \texorpdfstring{\cref{thm:Dolgopyat}}{\autoref{thm:Dolgopyat}}}
\label{sec:ProofOfDolgopyat'sMethod}
In this section, we will prove \cref{thm:Dolgopyat} by showing each of its properties.

Firstly, we omit the proofs of \cref{itm:DolgopyatProperty1,itm:LogLipschitzDolgopyat} in \cref{thm:Dolgopyat} as they are very similar to \cite[Lemma 3.15]{OW16} and \cite[Lemmas 9.1 and 9.2]{SW21} (which originally appear in \cite[Proposition 6 and Lemma 11]{Dol98}). The proofs use \cref{lem:PreliminaryLogLipschitz} and the choice of constants \cref{eqn:ConstantE,eqn:Constantm,eqn:Constantmu}.

\subsection{Proof of \texorpdfstring{\cref{itm:DolgopyatProperty2}}{Property \ref{itm:DolgopyatProperty2}} in \texorpdfstring{\cref{thm:Dolgopyat}}{\autoref{thm:Dolgopyat}}}
We first recall a definition from \cite{Sto11} which will be useful.

\begin{definition}
\label{def:tSDenseSubset}
We say that a subset  $W \subset \Lambda$ is \emph{$(t, C)$-dense} for some $t > 0$ and $C \geq 1$, if there exists a set of mutually disjoint cylinders $\{\mathtt{B}_1, \mathtt{B}_2, \dotsc, \mathtt{B}_k\}$ for some $k \in \mathbb N$ with $\bigsqcup_{j = 1}^k \mathtt{B}_j = \Lambda$ such that for all integers $1 \leq j \leq k$, we have
\begin{enumerate}
\item	$\diam(\mathtt{B}_j) \leq tC$
\item	there exists a subcylinder $\mathtt{B}_j' \subset W \cap \mathtt{B}_j$ with $\diam(\mathtt{B}_j') \geq t$.
\end{enumerate}
\end{definition}

\begin{lemma}
\label{lem:WDenseInequality}
Let $B > 0$ and $C \geq 1$. There exists $\eta \in (0, 1)$ such that for all $t > 0$, $(t, C)$-dense subsets $W \subset \Lambda$, and $h \in \tilde{\mathcal{C}}_{B/t}(\Lambda)$, we have $\int_W h \, d\nu_\Lambda \geq \eta \int_\Lambda h \, d\nu_\Lambda$.
\end{lemma}

\begin{proof}
Let $B > 0$ and $C \geq 1$. Fix $\eta = e^{-BC} \cdot \frac{c_1^\Lambda}{c_2^\Lambda}e^{-p_0 \delta_\Gamma \overline{\tau}\left(1 - \frac{\log(C)}{\log(\rho)}\right)} \in (0, 1)$. Let $t > 0$, $W \subset \Lambda$ be a $(t, C)$-dense subset, and $\{\mathtt{B}_1, \mathtt{B}_2, \dotsc, \mathtt{B}_k\}$ be the set of mutually disjoint cylinders as in \cref{def:tSDenseSubset}. Then $\bigsqcup_{j = 1}^k \mathtt{B}_j = \Lambda$, $\sum_{j = 1}^k \nu_\Lambda(\mathtt{B}_j) = 1$, $\diam(\mathtt{B}_j) \leq tC$, and there exists a subcylinder $\mathtt{B}_j' \subset W \cap \mathtt{B}_j$ with $\diam(\mathtt{B}_j') \geq t$, for all integers $1 \leq j \leq k$. Let $h \in \tilde{\mathcal{C}}_{B/t}(\Lambda)$. Consider some integers $1 \leq j \leq k$. Setting $l_j = \inf_{u \in \mathtt{B}_j} h(u)$ and $L_j = \sup_{u \in \mathtt{B}_j} h(u)$, we have $l_j \geq L_j e^{-BC}$. Write $\len(\mathtt{B}_j') = \len(\mathtt{B}_j) + r_jp_0 + s_j$ for some $r_j, s_j \in \mathbb Z_{\geq 0}$ with $s_j < p_0$. \Cref{lem:CylinderDiameterBound} gives $t \leq \diam(\mathtt{B}_j') \leq \rho^{r_j}\diam(\mathtt{B}_j) \leq \rho^{r_j}Ct$ which implies $r_j \leq -\frac{\log(C)}{\log(\rho)}$. The property of Gibbs measures in \cref{eqn:PropertyOfGibbsMeasures} gives $\frac{\nu_\Lambda(\mathtt{B}_j')}{\nu_\Lambda(\mathtt{B}_j)} \geq \frac{c_1^\Lambda}{c_2^\Lambda}e^{-(r_jp_0 + s_j)\delta_\Gamma\overline{\tau}} \geq \frac{c_1^\Lambda}{c_2^\Lambda}e^{-p_0 \delta_\Gamma \overline{\tau}\left(1 - \frac{\log(C)}{\log(\rho)}\right)}$ which implies $\nu_\Lambda(W \cap \mathtt{B}_j) \geq \nu_\Lambda(\mathtt{B}_j') \geq \frac{c_1^\Lambda}{c_2^\Lambda}e^{-p_0 \delta_\Gamma \overline{\tau}\left(1 - \frac{\log(C)}{\log(\rho)}\right)} \nu_\Lambda(\mathtt{B}_j)$. Thus, decomposing $W = \bigsqcup_{j = 1}^k W \cap \mathtt{B}_j$, we have
\begin{align*}
\int_W h(u) \, d\nu_\Lambda(u) 
&\geq \sum_{j = 1}^k l_j \nu_\Lambda(W \cap \mathtt{B}_j) \geq \sum_{j = 1}^k L_j e^{-BC} \nu_\Lambda(W \cap \mathtt{B}_j) \\
&\geq e^{-BC} \cdot \frac{c_1^\Lambda}{c_2^\Lambda}e^{-p_0 \delta_\Gamma \overline{\tau}\left(1 - \frac{\log(C)}{\log(\rho)}\right)} \sum_{j = 1}^k L_j \nu_\Lambda(\mathtt{B}_j) \\
&\geq \eta\sum_{j = 1}^k \int_{\mathtt{B}_j} h(u) \, d\nu_\Lambda(u) = \eta \int_\Lambda h(u) \, d\nu_\Lambda(u).
\end{align*}
\end{proof}

\begin{lemma}
\label{lem:DolgopyatProperty2}
There exist $a_0 > 0$ and $\eta \in (0, 1)$ such that for all $\xi \in \mathbb C$ with $|a| < a_0$ and $|b| > b_0$, $J \in \mathcal J(b)$, and $h \in \mathcal{C}_{E|b|}(\Lambda)$, we have $\|\mathcal{N}_{a, J}(h)\|_2 \leq \eta \|h\|_2$.
\end{lemma}

\begin{proof}
Fix $B = E\epsilon_1\rho^{p_0p_1 + 1} > 0$ and $C = \rho^{-(p_0p_1 + 1)} \geq 1$. Fix $\eta' \in (0, 1)$ to be the $\eta$ provided by \cref{lem:WDenseInequality}. Fix a positive constant
\begin{align*}
a_0 < \min\left(a_0', \frac{1}{mA_f}\log\left(\frac{1}{1 - \eta'\mu e^{-mT_0}}\right)\right)
\end{align*}
so that we can also fix $\eta = \sqrt{e^{mA_fa_0}(1 - \eta'\mu e^{-mT_0})} \in (0, 1)$. Recall that we already fixed $b_0 = 1$. Let $\xi \in \mathbb C$ with $|a| < a_0$ and $|b| > b_0$, $J \in \mathcal J(b)$, and $h \in \mathcal{C}_{E|b|}(\Lambda)$. We have the estimate $\mathcal{N}_{a, J}(h)^2 \leq e^{mA_fa_0}\mathcal{N}_{0, J}(h)^2$ since $|f^{(a)} - f^{(0)}| \leq A_f|a| < A_fa_0$ and by the Cauchy--Schwarz inequality, we have
\begin{align*}
\mathcal{N}_{0, J}(h)^2 = L_0^m(\beta_J h)^2 \leq L_0^m({\beta_J}^2) L_0^m(h^2).
\end{align*}
We would like to apply \cref{lem:WDenseInequality} on $h$ but first we need to ensure all the hypotheses hold. Let $t = \rho^{p_0p_1 + 1}\frac{\epsilon_1}{|b|}$ and note that $\frac{B}{t} = E|b|$. Let $W = \bigsqcup_{(j, k) \in J} \mathtt{D}_k(b) \subset \Lambda$. We will show that $W$ is $(t, C)$-dense. Let $\mathtt{B}_l = \mathtt{C}_l(b)$ for all $1 \leq l \leq c_b$. Then $\mathtt{B}_l \cap \mathtt{B}_{l'} = \varnothing$ and $\diam(\mathtt{B}_l) \leq tC$ by \cref{eqn:DiameterBoundC_l} for all integers $1 \leq l, l' \leq c_b$ with $l \neq l'$, and $\bigsqcup_{l = 1}^{c_b} \mathtt{B}_l = \Lambda$. Since $J$ is dense, for all integers $1 \leq l \leq c_b$, there exists $(j, k) \in J$ such that $\mathtt{D}_k(b) \subset \mathtt{C}_l(b)$ and so choosing $\mathtt{B}_l' = \mathtt{D}_k(b)$ gives $\mathtt{B}_l' \subset W \cap \mathtt{B}_l$ with $\diam(\mathtt{B}_l') \geq t$ by \cref{eqn:DiameterBoundD_k}. Hence, $W$ is $(t, C)$-dense. Since $h \in \mathcal{C}_{E|b|}(\Lambda) \subset \tilde{\mathcal{C}}_{E|b|}(\Lambda)$, we have $h^2 \in \tilde{\mathcal{C}}_{2E|b|}(\Lambda)$. Applying \cref{lem:PreliminaryLogLipschitz} gives $L_0^m(h^2) \in \tilde{\mathcal{C}}_{B'}(\Lambda)$ where $B' = A_0\left(\frac{2E|b|}{\kappa_2^m} + 1\right) \leq A_0\left(\frac{2E|b|}{8A_0} + \frac{E|b|}{2A_0}\right) \leq E|b|$ using \cref{eqn:ConstantE,eqn:Constantm}. Thus $L_0^m(h^2) \in \tilde{\mathcal{C}}_{E|b|}(\Lambda) = \tilde{\mathcal{C}}_{B/t}(\Lambda)$. Now we can apply \cref{lem:WDenseInequality} to get $\int_W L_0^m(h^2) \, d\nu_\Lambda \geq \eta' \int_\Lambda L_0^m(h^2) \, d\nu_\Lambda$. Note that $L_0^m({\beta_J}^2)(u) \leq L_0^m(\chi_U - \mu \chi_{\mathtt{X}_{j, k}(b)})(u) \leq 1 - \mu e^{-mT_0}$ for all $u \in W$ by choosing any $(j, k) \in J$. So putting everything together and using $L_0^*(\nu_\Lambda) = \nu_\Lambda$, we have
\begin{align*}
&\int_\Lambda \mathcal{N}_{a, J}(h)^2 \, d\nu_\Lambda \leq \int_\Lambda e^{mA_fa_0}\mathcal{N}_{0, J}(h)^2 \, d\nu_\Lambda \\
\leq{}&e^{mA_fa_0}\left(\int_{W} L_0^m({\beta_J}^2) L_0^m(h^2) \, d\nu_\Lambda + \int_{\Lambda \setminus W} L_0^m({\beta_J}^2) L_0^m(h^2) \, d\nu_\Lambda\right) \\
\leq{}&e^{mA_fa_0}\left((1 - \mu e^{-mT_0})\int_{W} L_0^m(h^2) \, d\nu_\Lambda + \int_{\Lambda \setminus W} L_0^m(h^2) \, d\nu_\Lambda\right) \\
={}&e^{mA_fa_0}\left(\int_\Lambda L_0^m(h^2) \, d\nu_\Lambda - \mu e^{-mT_0}\int_{W} L_0^m(h^2) \, d\nu_\Lambda\right) \\
\leq{}&e^{mA_fa_0}(1 - \eta'\mu e^{-mT_0})\int_\Lambda L_0^m(h^2) \, d\nu_\Lambda = \eta^2 \int_\Lambda h^2 \, d\nu_\Lambda.
\end{align*}
\end{proof}

\subsection{Proof of \texorpdfstring{\cref{itm:DominatedByDolgopyat}}{Property \ref{itm:DominatedByDolgopyat}} in \texorpdfstring{\cref{thm:Dolgopyat}}{\autoref{thm:Dolgopyat}}}
\Cref{lem:LNIC_Output} follows from \cref{pro:NonConcentrationProperty,cor:LNIC_Corollary} and \cref{eqn:DiameterBoundC_l,eqn:Constantepsilon1,eqn:DiameterBoundD_k,eqn:Constantp1} similar to \cite[Lemma 5.9]{Sto11}.

\begin{lemma}
\label{lem:LNIC_Output}
Let $|b| > b_0$. Suppose $\mathtt{D}_k(b) \subset \mathtt{C}_l(b)$ for some integers $1 \leq k \leq d_b$ and $1 \leq l \leq c_b$. Then there exists an adjacent $\mathtt{D}_{k'}(b) \subset \mathtt{C}_l(b)$ for some integer $1 \leq k' \leq d_b$ such that
\begin{align*}
\frac{\delta_0 \rho \epsilon_1}{16} \leq |b| \cdot |(\tau_m \circ v_1 - \tau_m \circ v_2)(u) - (\tau_m \circ v_1 - \tau_m \circ v_2)(u')| \leq \pi
\end{align*}
for all $u \in \mathtt{D}_k(b)$ and $u' \in \mathtt{D}_{k'}(b)$.
\end{lemma}

Now, for all $\xi \in \mathbb C$ with $|a| < a_0'$ and $|b| > b_0$, $H \in C(\Lambda, L^2(\tilde{\mathbf{G}}_q))$ for some nonzero $q \in \mathcal{O}$, and $h \in \mathcal{C}_{E|b|}(\Lambda)$, we define the functions $\chi_1^{[\xi, H, h]}, \chi_2^{[\xi, H, h]}: \Lambda \to \mathbb R$ by
\begin{align*}
&\chi_1^{[\xi, H, h]}(u) \\
={}&\frac{\left\|e^{(f_m^{(a)} + ib\tau_m)(v_1(u))} \mathtt{c}_{l, 1}(b)H(v_1(u)) + e^{(f_m^{(a)} + ib\tau_m)(v_2(u))} \mathtt{c}_{l, 2}(b)H(v_2(u))\right\|_2}{(1 - \mu)e^{f_m^{(a)}(v_1(u))}h(v_1(u)) + e^{f_m^{(a)}(v_2(u))}h(v_2(u))}; \\
&\chi_2^{[\xi, H, h]}(u) \\
={}&\frac{\left\|e^{(f_m^{(a)} + ib\tau_m)(v_1(u))} \mathtt{c}_{l, 1}(b)H(v_1(u)) + e^{(f_m^{(a)} + ib\tau_m)(v_2(u))} \mathtt{c}_{l, 2}(b)H(v_2(u))\right\|_2}{e^{f_m^{(a)}(v_1(u))}h(v_1(u)) + (1 - \mu)e^{f_m^{(a)}(v_2(u))}h(v_2(u))}
\end{align*}
for all $u \in \mathtt{C}_l(b)$ and $1 \leq l \leq c_b$. We also need the following lemma which can be proved as in \cite[Lemma 3.17]{OW16} and \cite[Lemma 9.8]{SW21} (which originally appears in \cite[Lemma 14]{Dol98}).

\begin{lemma}
\label{lem:HTrappedByh}
Let $|b| > b_0$ and $q \in \mathcal{O}$ be nonzero. Suppose $H \in C(\Lambda, L^2(\tilde{\mathbf{G}}_q))$ and $h \in \mathcal{C}_{E|b|}(\Lambda)$ satisfy \cref{itm:DominatedByh,itm:LogLipschitzh} in \cref{thm:Dolgopyat}. Then, for all $(j, k) \in \Xi(b)$, we have
\begin{align*}
\frac{1}{2} \leq \frac{h(v_j(u))}{h(v_j(u'))} \leq 2
\end{align*}
for all $u, u' \in \mathtt{D}_k(b)$ and also either of the alternatives
\begin{alternative}
\item\label{alt:HLessThan3/4h}	$\big\|H(v_j(u))\big\|_2 \leq \frac{3}{4}h(v_j(u))$ for all $u \in \mathtt{D}_k(b)$;
\item\label{alt:HGreaterThan1/4h}	$\big\|H(v_j(u))\big\|_2 \geq \frac{1}{4}h(v_j(u))$ for all $u \in \mathtt{D}_k(b)$.
\end{alternative}
\end{lemma}

For any integer $k \geq 2$, let $\Theta: (\mathbb R^k \setminus \{0\}) \times (\mathbb R^k \setminus \{0\}) \to [0, \pi]$ be the map which gives the angle defined by $\Theta(w_1, w_2) = \arccos\left(\frac{\langle w_1, w_2\rangle}{\|w_1\| \cdot \|w_2\|}\right)$ for all $w_1, w_2 \in \mathbb R^k \setminus \{0\}$, where we use the standard inner product and norm. \Cref{lem:StrongTriangleInequality} is a stronger version of the triangle inequality proven by elementary trigonometry.

\begin{lemma}
\label{lem:StrongTriangleInequality}
Let $k \geq 2$ be an integer. Suppose $w_1, w_2 \in \mathbb R^k \setminus \{0\}$ such that $\Theta(w_1, w_2) \geq \alpha$ and $\frac{\|w_1\|}{\|w_2\|} \leq L$ for some $\alpha \in [0, \pi]$ and $L \geq 1$. Then, we have
\begin{align*}
\|w_1 + w_2\| \leq \left(1 - \frac{\alpha^2}{16L}\right)\|w_1\| + \|w_2\|.
\end{align*}
\end{lemma}

\begin{lemma}
\label{lem:chiLessThan1}
Let $\xi \in \mathbb C$ with $|a| < a_0'$ and $|b| > b_0$. Let $q \in \mathcal{O}$ be nonzero and suppose $H \in C(\Lambda, L^2(\tilde{\mathbf{G}}_q))$ and $h \in \mathcal{C}_{E|b|}(\Lambda)$ satisfy \cref{itm:DominatedByh,itm:LogLipschitzh} in \cref{thm:Dolgopyat}. For all integers $1 \leq l \leq c_b$, there exists $(j, k) \in \Xi(b)$ such that $\mathtt{D}_k(b) \subset \mathtt{C}_l(b)$ and such that $\chi_j^{[\xi, H, h]}(u) \leq 1$ for all $u \in \mathtt{D}_k(b)$.
\end{lemma}

\begin{proof}
Recall that we already fixed $b_0 = 1$. Let $\xi \in \mathbb C$ with $|a| < a_0'$ and $|b| > b_0$. Let $q \in \mathcal{O}$ be nonzero and suppose $H \in C(\Lambda, L^2(\tilde{\mathbf{G}}_q))$ and $h \in \mathcal{C}_{E|b|}(\Lambda)$ satisfy \cref{itm:DominatedByh,itm:LogLipschitzh} in \cref{thm:Dolgopyat}. Let $1 \leq l \leq c_b$ be an integer. There exist adjacent $\mathtt{D}_k(b), \mathtt{D}_{k'}(b) \subset \mathtt{C}_l(b)$ for some integers $1 \leq k, k' \leq d_b$ so that the conclusion of \cref{lem:LNIC_Output} holds. Now, suppose \cref{alt:HLessThan3/4h} in \cref{lem:HTrappedByh} holds for one of $(j, k), (j, k') \in \Xi(b)$ for some $j \in \{1, 2\}$. Without loss of generality, we can assume it holds for $(j, k) \in \Xi(b)$ and then it is a straightforward calculation to check that $\chi_j^{[\xi, H, h]}(u) \leq 1$ for all $u \in \mathtt{D}_k(b)$, using \cref{eqn:Constantmu}. Otherwise, \cref{alt:HGreaterThan1/4h} in \cref{lem:HTrappedByh} holds for all of $(1, k), (2, k), (1, k'), (2, k') \in \Xi(b)$. Let $u \in \mathtt{D}_k(b)$ and $u' \in \mathtt{D}_{k'}(b)$. Note that $\big\|H(v_j(u))\big\|_2, \big\|H(v_j(u'))\big\|_2 > 0$ for all $j \in \{1, 2\}$. We would like to apply \cref{lem:StrongTriangleInequality} but first we need to establish bounds on relative angle and relative size using $L^2(\tilde{\mathbf{G}}_q) \cong \mathbb R^{2\#\tilde{\mathbf{G}}_q}$ as real vector spaces. We start with the former. For all $j \in \{1, 2\}$, let $u_j \in \{u, u'\}$ such that $\big\|H(v_j(u_j))\big\|_2 = \min\left(\big\|H(v_j(u))\big\|_2, \big\|H(v_j(u'))\big\|_2\right)$. Then recalling \cref{lem:Hyperbolicity,eqn:Constantm}, for all $j \in \{1, 2\}$, we have
\begin{align*}
\frac{\big\|H(v_j(u)) - H(v_j(u'))\big\|_2}{\min\left(\big\|H(v_j(u))\big\|_2, \big\|H(v_j(u'))\big\|_2\right)} &\leq \frac{E|b|h(v_j(u_j)) \|v_j(u) - v_j(u')\|}{\big\|H(v_j(u_j))\big\|_2} \\
&\leq 4E|b| \cdot \frac{\epsilon_1}{|b|c_0\kappa_2^m} \leq \frac{\delta_0 \rho \epsilon_1}{128}.
\end{align*}
Thus, $\sin(\Theta(H(v_j(u)), H(v_j(u')))) \leq \frac{\delta_0 \rho \epsilon_1}{128}$ with $\Theta(H(v_j(u)), H(v_j(u'))) \in [0, \frac{\pi}{2})$ for all $j \in \{1, 2\}$ by the sine law. A simple inequality $\frac{\theta}{2} \leq \sin(\theta)$ for all $\theta \in \left[0, \frac{\pi}{2}\right]$ gives $\Theta(H(v_j(u)), H(v_j(u'))) \leq 2\sin(\Theta(H(v_j(u)), H(v_j(u')))) \leq \frac{\delta_0 \rho \epsilon_1}{64}$ for all $j \in \{1, 2\}$. Define $\varphi: U \to \mathbb R$ by $\varphi = b(\tau_m \circ v_1 - \tau_m \circ v_2)$. By \cref{lem:LNIC_Output}, we have $\frac{\delta_0 \rho \epsilon_1}{16} \leq |\varphi(u) - \varphi(u')| \leq \pi$. We will use these bounds to obtain a lower bound for $\Theta(V_1(u), V_2(u))$ or $\Theta(V_1(u'), V_2(u'))$ where we define
\begin{align*}
V_j(w) &= e^{(f_m^{(a)} + ib\tau_m)(v_j(w))} \mathtt{c}_{l, j}(b)H(v_j(w)) \qquad \text{for all $w \in U$ and $j \in \{1, 2\}$}.
\end{align*}
Using the triangle inequality and the unitarity of the cocycle, we have
\begin{align*}
&\Theta(V_1(u), V_2(u)) = \Theta\big(e^{i\varphi(u)} \mathtt{c}_{l, 1}(b)H(v_1(u)), \mathtt{c}_{l, 2}(b)H(v_2(u))\big) \\
\geq{}&\Theta\big(e^{i\varphi(u)} \mathtt{c}_{l, 1}(b)H(v_1(u)), e^{i\varphi(u')} \mathtt{c}_{l, 1}(b)H(v_1(u))\big) \\
&{}- \Theta\big(e^{i\varphi(u')} \mathtt{c}_{l, 1}(b)H(v_1(u)), e^{i\varphi(u')} \mathtt{c}_{l, 1}(b)H(v_1(u'))\big) \\
&{}- \Theta\big(\mathtt{c}_{l, 2}(b)H(v_2(u)), \mathtt{c}_{l, 2}(b)H(v_2(u'))\big) \\
&{}- \Theta\big(e^{i\varphi(u')} \mathtt{c}_{l, 1}(b)H(v_1(u')), \mathtt{c}_{l, 2}(b)H(v_2(u'))\big) \\
={}&|\varphi(u) - \varphi(u')| - \Theta\big(H(v_1(u)), H(v_1(u'))\big) - \Theta\big(H(v_2(u)), H(v_2(u'))\big) \\
{}&- \Theta\big(e^{i\varphi(u')} \mathtt{c}_{l, 1}(b)H(v_1(u')), \mathtt{c}_{l, 2}(b)H(v_2(u'))\big) \\
\geq{}&\frac{\delta_0 \rho \epsilon_1}{16} - \frac{\delta_0 \rho \epsilon_1}{64} - \frac{\delta_0 \rho \epsilon_1}{64} - \Theta(V_1(u'), V_2(u')) = \frac{\delta_0 \rho \epsilon_1}{32} - \Theta(V_1(u'), V_2(u')).
\end{align*}
Hence $\Theta(V_1(u), V_2(u)) + \Theta(V_1(u'), V_2(u')) \geq \frac{\delta_0 \rho \epsilon_1}{32}$ for all $u \in \mathtt{D}_k(b)$ and $u' \in \mathtt{D}_{k'}(b)$. Thus, without loss of generality, we can assume that $\Theta(V_1(u), V_2(u)) \geq \frac{\delta_0 \rho \epsilon_1}{64}$ for all $u \in \mathtt{D}_k(b)$, which establishes the required bound on relative angle. For the bound on relative size, let $(j, j') \in \{(1, 2), (2, 1)\}$ such that $h(v_j(u_0)) \leq h(v_{j'}(u_0))$ for some $u_0 \in \mathtt{D}_k(b)$. Then by \cref{lem:HTrappedByh}, we have
\begin{align*}
\frac{\|V_j(u)\|_2}{\|V_{j'}(u)\|_2} &= \frac{e^{f_m^{(a)}(v_j(u))}\big\|H(v_j(u))\big\|_2}{e^{f_m^{(a)}(v_{j'}(u))}\big\|H(v_{j'}(u))\big\|_2} \leq \frac{4e^{f_m^{(a)}(v_j(u)) - f_m^{(a)}(v_{j'}(u))}h(v_j(u))}{h(v_{j'}(u))} \\
&\leq \frac{16e^{2mT_0}h(v_j(u_0))}{h(v_{j'}(u_0))} \leq 16e^{2mT_0}
\end{align*}
for all $u \in \mathtt{D}_k(b)$, which establishes the required bound on relative size. Now applying \cref{lem:StrongTriangleInequality,eqn:Constantmu} and $\|H\| \leq h$ on $\|V_j(u) + V_{j'}(u)\|_2$ gives $\chi_j^{[\xi, H, h]}(u) \leq 1$ for all $u \in \mathtt{D}_k(b)$.
\end{proof}

\begin{lemma}
\label{lem:DominatedByDolgopyat}
There exists $a_0 > 0$ such that for all $\xi \in \mathbb C$ with $|a| < a_0$ and $|b| > b_0$, and nonzero $q \in \mathcal{O}$, if $H \in C(\Lambda, L^2(\tilde{\mathbf{G}}_q))$ and $h \in \mathcal{C}_{E|b|}(\Lambda)$ satisfy \cref{itm:DominatedByh,itm:LogLipschitzh} in \cref{thm:Dolgopyat}, then there exists $J \in \mathcal{J}(b)$ such that
\begin{align*}
\left\|M_{\xi, q}^m(H)(u)\right\|_2 \leq \mathcal{N}_{a, J}(h)(u) \qquad \text{for all $u \in \Lambda$}.
\end{align*}
\end{lemma}

\begin{proof}
Fix $a_0 = a_0'$ and recall that we already fixed $b_0 = 1$. Let $\xi \in \mathbb C$ with $|a| < a_0$ and $|b| > b_0$, and $q \in \mathcal{O}$ be nonzero. Suppose $H \in C(\Lambda, L^2(\tilde{\mathbf{G}}_q))$ and $h \in \mathcal{C}_{E|b|}(\Lambda)$ satisfy \cref{itm:DominatedByh,itm:LogLipschitzh} in \cref{thm:Dolgopyat}. For all integers $1 \leq l \leq c_b$, we can choose a $(j_l, k_l) \in \Xi(b)$ as guaranteed by \cref{lem:chiLessThan1}. Let $J = \{(j_l, k_l) \in \Xi(b): 1 \leq l \leq c_b\} \subset \Xi(b)$ which is then dense by construction and so $J \in \mathcal{J}(b)$. Now we show that $\big\|M_{\xi, q}^m(H)\big\| \leq \mathcal{N}_{a, J}(h)$ for this choice of $J \in \mathcal{J}(b)$. Let $u \in \Lambda$. If $u \notin \mathtt{D}_k(b)$ for all $(j, k) \in J$, then $\beta_J(v) = 1$ for all $v \in \sigma^{-m}(u)$ and hence the bound follows trivially by definitions. Otherwise, there is an integer $1 \leq l \leq c_b$ such that $u \in \mathtt{D}_{k_l}(b)$ corresponding to $(j_l, k_l) \in J$. Note that by construction, $(j, k_l) \notin J$ if $j \neq j_l$. Let $(j_l, j_l') \in \{(1, 2), (2, 1)\}$. Then, we have $\chi_{j_l}^{[\xi, H, h]}(u) \leq 1$, $\beta_J(v_{j_l}(u)) = 1 - \mu$, and $\beta_J(v_{j_l'}(u)) = 1$. Hence, we compute that
\begin{align*}
&\left\|M_{\xi, q}^m(H)(u)\right\|_2 \\
\leq{}&\sum_{v \in \sigma^{-m}(u), v \neq v_1(u), v \neq v_2(u)} \left\|e^{(f_m^{(a)} + ib\tau_m)(v)} \mathtt{c}^m(v) H(v)\right\|_2 \\
&{}+ \Big\|e^{(f_m^{(a)} + ib\tau_m)(v_{j_l}(u))} \mathtt{c}_{l, j_l}(b) H(v_{j_l}(u)) + e^{(f_m^{(a)} + ib\tau_m)(v_{j_l'}(u))} \mathtt{c}_{l, j_l'}(b) H(v_{j_l'}(u))\Big\|_2 \\
\leq{}&\sum_{\substack{v \in \sigma^{-m}(u)\\ v \neq v_1(u)\\ v \neq v_2(u)}} e^{f_m^{(a)}(v)} h(v) + \Bigl((1 - \mu)e^{f_m^{(a)}(v_{j_l}(u))}h(v_{j_l}(u)) + e^{f_m^{(a)}(v_{j_l'}(u))}h(v_{j_l'}(u))\Bigr) \\
\leq{}&\mathcal{N}_{a, J}(h)(u).
\end{align*}
\end{proof}

\section{Converting uniform spectral bounds to a uniform count}
\label{sec:ConvertingUniformSpectralBoundsToUniformCounting}
Here we follow \cite[Section 3]{MOW19} and outline the derivation of \cref{thm:MainTheorem} from the uniform spectral bounds in \cref{thm:CongruenceOperatorSpectralBounds} using the congruence renewal equation. We omit all the proofs as they can be found in \cite[Section 3]{MOW19}.

We define $L^\star(\mathbb H^n \cup \mathbb R^{n - 1}, \mathbb R) \subset L(\mathbb H^n \cup \mathbb R^{n - 1}, \mathbb R)$ to be the subspace of bounded functions in $L(\mathbb H^n \cup \mathbb R^{n - 1}, \mathbb R)$ which are locally constant on some neighborhood of $\Lambda \subset \mathbb H^n \cup \mathbb R^{n - 1}$. We fix $F \in L^\star(\mathbb H^n \cup \mathbb R^{n - 1}, \mathbb R_{\geq 0})$ and $f = F|_{\Lambda}$ in this section.

Let $q \in \mathcal{O}$ be nonzero. Define the function $N_q^*:  \mathbb R \times \Gamma \times L^2(\tilde{\mathbf{G}}_q) \to L^2(\tilde{\mathbf{G}}_q)$ by
\begin{align*}
N_q^*(r, \gamma_0, \phi) = \sum_{\substack{\gamma \in \Gamma \text{ such that}\\ d(o, \gamma \gamma_0 o) - d(o, \gamma_0 o) \leq r}} F(\gamma \gamma_0 o) \cdot \pi_q(\tilde{\gamma}) \phi
\end{align*}
for all $(r, \gamma_0, \phi) \in  \mathbb R \times \Gamma \times L^2(\tilde{\mathbf{G}}_q)$.

\begin{remark}
We need to compare the hyperbolic distance and the Frobenius norm. Recalling the choice $o = e_n$, we have the formula $\|\gamma\|^2 = 2\cosh(d(o, \gamma o))$ for all $\gamma \in G$ where $\|\cdot\|$ is the Frobenius norm on the space of $(n + 1) \times (n + 1)$ matrices. Writing $R = e^{\frac{r}{2}}$, since $\sqrt{2\cosh(r)} = e^{\frac{r}{2}} + O\bigl(e^{-\frac{3}{2}r}\bigr)$ as $r \to +\infty$, we can use the above formula to deduce that if $\frac{\|\gamma \gamma_0\|}{\|\gamma_0\|} \leq R$ for some $\gamma, \gamma_0 \in \Gamma$, then we have
\begin{align}
\begin{aligned}
\label{eqn:FrobeniusNormToHyperbolicDistance}
d(o, \gamma \gamma_0 o) - d(o, \gamma_0 o) \leq{}& 2\log(R) + \log\bigl(1 + e^{-2d(o, \gamma_0 o)}\bigr) \\
&{}+ \log\left(1 - \frac{e^{-d(o, \gamma \gamma_0 o)}}{2R^2\cosh(d(o, \gamma_0 o))}\right) \\
\leq{}& r + C +  O(e^{-r})
\end{aligned}
\end{align}
as $r \to +\infty$ where $C$ depends on $\gamma_0$ and the implied constant depends on $\Gamma$ and $\gamma_0$. Consequently, this estimate allows us to replace the condition $\frac{\|\gamma \gamma_0\|}{\|\gamma_0\|} \leq R$ with the condition $d(o, \gamma \gamma_0 o) - d(o, \gamma_0 o) \leq r$ only at the expense of having to modify the constant factor in the main term in \cref{thm:MainTheorem}.
\end{remark}

The above function is relevant to the uniform counting result because if $F = \chi_{\mathbb H^n \cup \mathbb R^{n - 1}} \in L^\star(\mathbb H^n \cup \mathbb R^{n - 1},  \mathbb R)$, then taking $\gamma_0 = e$ and $\phi = \delta_e$, we have
\begin{align*}
N_q^*(r', e, \delta_e)(e) = \sum_{\substack{\gamma \in \Gamma \text{ such that}\\ d(o, \gamma o) \leq r'}} (\pi_q(\tilde{\gamma}) \delta_e)(e) \geq \#(\Gamma_q \cap B_R(e))
\end{align*}
for all $R > 0$, where $r'$ is the right hand side from \cref{eqn:FrobeniusNormToHyperbolicDistance} and $B_R(e) \subset G$ denotes the ball of radius $R > 0$ centered at $e \in G$ with respect to the Frobenius norm.

\subsection{Pushing the uniform count to the boundary}
Let $q \in \mathcal{O}$ be nonzero. Define the function $N_q:  \mathbb R \times \Lambda \times L^2(\tilde{\mathbf{G}}_q) \to L^2(\tilde{\mathbf{G}}_q)$ by
\begin{align*}
N_q(r, u, \phi) = \sum_{j = 0}^\infty \sum_{u' \in T^{-j}(u)} f(u') \cdot \mathtt{c}_q^j(u') \phi \cdot \chi_{\{r' \in \mathbb R: r' \leq r\}}(\tau_j(u'))
\end{align*}
which satisfies the recursive formula
\begin{align*}
N_q(r, u, \phi) = \sum_{j = 0}^\infty \sum_{u' \in T^{-j}(u)} \mathtt{c}_q(u') N_q(r - \tau(u'), u', \phi) + f(u) \cdot \phi \cdot \chi_{\{r' \in \mathbb R: r' \leq r\}}(0)
\end{align*}
called the \emph{congruence renewal equation}, for all $(r, \gamma_0, \phi) \in  \mathbb R \times \Lambda \times L^2(\tilde{\mathbf{G}}_q)$.

Let $q \in \mathcal{O}$ be nonzero. We now relate $N_q^*$ to its boundary version $N_q$. Define $\tau^*(\gamma) = d(o, \gamma o) - d(o, T(\gamma o))$ for all $\gamma \in \Gamma$. Taking sums as in \cref{eqn:BirkhoffSums}, we also use the notation $\tau_k^*(\gamma) = d(o, \gamma o) - d(o, T^k(\gamma o))$ for all $k \in \mathbb N$ and $\tau_0^*(\gamma) = 0$, for all $\gamma \in \Gamma$. For any $\gamma_0 \in \Gamma$, let $\lambda_0 \in \Lambda$ denote a corresponding admissible point henceforth. Observe that we can rewrite the formula for $N_q^*$ as
\begin{align*}
N_q^*(r, \gamma_0, \phi) = \sum_{j = 0}^\infty \sum_{\substack{\gamma \in \Gamma \text{ such that}\\ \gamma \lambda_0 \in T^{-j}(\gamma_0 \lambda_0)}} F(\gamma o) \cdot \pi_q(\tilde{\gamma} \tilde{\gamma}_0^{-1}) \phi \cdot \chi_{\{r' \in \mathbb R: r' \leq r\}}(\tau_j^*(\gamma))
\end{align*}
which satisfies the recursive formula
\begin{align*}
N_q^*(r, \gamma_0, \phi) ={}&\sum_{\substack{\gamma \in \Gamma \text{ such that}\\ \gamma \lambda_0 \in T^{-1}(\gamma_0 \lambda_0)}} N_q^*(r - \tau^*(\gamma), \gamma, \pi_q(\tilde{\gamma} \tilde{\gamma}_0^{-1}) \phi) \\
&{}+F(\gamma_0 o) \cdot \phi \cdot \chi_{\{r' \in \mathbb R: r' \leq r\}}(0)
\end{align*}
also called the \emph{congruence renewal equation}, for all $(r, \gamma_0, \phi) \in  \mathbb R \times \Gamma \times L^2(\tilde{\mathbf{G}}_q)$.

\begin{lemma}
\label{lem:EstimateFor_tau}
There exists $\kappa \in (0, 1)$ such that for all $\gamma, \gamma_0 \in \Gamma$, if $\len(\gamma) = k + K + 1$ for some $k, K \in \mathbb N$ and $\gamma \cdot \gamma_0$ is admissible, then we have
\begin{align*}
\tau_k^*(\gamma \gamma_0) = \tau_k(\gamma \gamma_0 \lambda_0) + O\bigl(\kappa^K\bigr).
\end{align*}
\end{lemma}

The following lemma is proved using \cref{lem:EstimateFor_tau} and the hypothesis on $F$.

\begin{lemma}
\label{lem:BoundN_q^*UsingN_q}
There exist $K_0 \in \mathbb N$, $\kappa \in (0, 1)$, and $C > 0$ depending on $F$ such that for all $\gamma_0, \gamma_1 \in \Gamma$, if $K > K_0$ and the reduced word $\gamma_1$ ends in the reduced word $\gamma_0$ with $\len(\gamma_1) = K + \len(\gamma_0)$, then we have
\begin{align*}
N_q(r - C\kappa^K, \gamma_1 \lambda_0, \phi) \leq N_q^*(r, \gamma_1, \phi) \leq N_q(r + C\kappa^K, \gamma_1 \lambda_0, \phi)
\end{align*}
for all nonzero $q \in \mathcal{O}$, $r \in \mathbb R$, and $\phi \in L^2(\tilde{\mathbf{G}}_q, \mathbb R_{\geq 0})$.
\end{lemma}

We can derive the following lemma from \cref{lem:BoundN_q^*UsingN_q} which allows us to convert expressions with $N_q^*$ to expressions with $N_q$. Recall that $N = \#\mathcal{A}$.

\begin{lemma}
\label{lem:ConvertN_qToN_q^*}
There exists $C > 0$ such that for all nonzero $q \in \mathcal{O}$, $c > 0$, and $(r, \gamma_0, \phi) \in  \mathbb R \times \Gamma \times L^2(\tilde{\mathbf{G}}_q, \mathbb R_{\geq 0})$, we have
\begin{align*}
X_{-} - C\|F\|_\infty \cdot \|\phi\|e^{\log(N)cr} \leq N_q^*(r, \gamma_0, \phi) \leq X_{+} + C\|F\|_\infty \cdot \|\phi\|e^{\log(N)cr}
\end{align*}
where $X_{\pm} = \sum_{\substack{\gamma \in \Gamma \text{ }\mathrm{such}\text{ }\mathrm{that}\\ \gamma \lambda_0 \in T^{-k}(\gamma_0 \lambda_0)}} N_q(r - \tau_k^*(\gamma) \pm C\kappa^k, \gamma \lambda_0, \pi_q(\tilde{\gamma} \tilde{\gamma}_0^{-1}) \phi)$ and $k = \lfloor cr\rfloor$.
\end{lemma}

\subsection{Proof that \texorpdfstring{\cref{thm:CongruenceOperatorSpectralBounds}}{\autoref{thm:CongruenceOperatorSpectralBounds}} implies \texorpdfstring{\cref{thm:MainTheorem}}{\autoref{thm:MainTheorem}}}
Let $q \in \mathcal{O}$ be nonzero. Define the function $n_q: \mathbb C \times \Lambda \times L^2(\tilde{\mathbf{G}}_q) \to L^2(\tilde{\mathbf{G}}_q)$
\begin{align*}
n_q(\xi, u, \phi) = \int_{-\infty}^\infty e^{-\xi r} N_q(r, u, \phi) \, dr \qquad \text{for all $(\xi, u, \phi) \in \mathbb C \times \Lambda \times L^2(\tilde{\mathbf{G}}_q)$}.
\end{align*}
Then using the congruence renewal equation, we obtain the formula
\begin{align*}
(\delta_\Gamma + \xi) \cdot n_q(\delta_\Gamma + \xi, u, \phi) = \bigl(\Id_{L(\Lambda, L^2(\tilde{\mathbf{G}}_q))} - \mathcal{M}_{\xi, q}\bigr)^{-1}(f \otimes \phi)(u)
\end{align*}
for all $(\xi, u, \phi) \in \mathbb C \times \Lambda \times L^2(\tilde{\mathbf{G}}_q)$, where we define $f \otimes \phi \in L(\Lambda, L^2(\tilde{\mathbf{G}}_q))$ by $(f \otimes \phi)(u) = f(u) \cdot \phi$ for all $u \in \Lambda$. We now focus on $\phi$ orthogonal to constants. For any $\sigma > 0$, using \cref{thm:CongruenceOperatorSpectralBounds} with the above formula gives the bounds
\begin{align*}
|\delta_\Gamma + \xi| \cdot \|n_q(\delta_\Gamma + \xi, \cdot, \phi)\|_{\Lip} &\leq
\begin{cases}
CN(q)^C(1 - e^{-\eta})^{-1} \|f \otimes \phi\|_{\Lip}, & |b| \leq b_0 \\
C_\sigma |b|^{1 + \sigma} (1 - e^{-\eta_\sigma})^{-1} \|f \otimes \phi\|_{\Lip}, &  |b| > b_0
\end{cases}
\\
&\leq C'\max\bigl(N(q)^C, |b|^{1 + \sigma}\bigr) \|f \otimes \phi\|_{\Lip}
\end{align*}
for all $\xi \in \mathbb C$ with $|a| < a_0$ and square-free $q \in \mathcal{O}$ coprime to $q_0q_0'$, where $q_0'$, $\eta$, $C$, $a_0$, $\eta_\sigma$, and $C_\sigma$ are from the same theorem, and $C' > 0$ is a constant only depending on the choice of $\sigma$. It suffices to choose $\sigma = 1$. The bounds imply that $\xi \mapsto \bigl(\Id_{L(\Lambda, L^2(\tilde{\mathbf{G}}_q))} - \mathcal{M}_{\xi, q}\bigr)^{-1}$ and hence $\xi \mapsto n_q(\delta_\Gamma + \xi, \cdot, \phi)$ is holomorphic on $\{\xi \in \mathbb C: |\Re(\xi)| < a_0\}$. For all $\lambda \in (0, 1)$, define $k_\lambda \in C_{\mathrm{c}}^\infty(\mathbb R, \mathbb R_{\geq 0})$ by $k_\lambda(t) = \lambda^{-1}k(\lambda^{-1}t)$ for all $t \in \mathbb R$ where $k \in C_{\mathrm{c}}^\infty(\mathbb R, \mathbb R_{\geq 0})$ is some bump function with $\int_\mathbb R k = 1$ and $\supp(k) \subset [-1, 1]$. Repeating the analysis in \cite[Subsection 3.4]{MOW19} (cf. \cite[Section 10]{BGS11}) gives the following lemma.

\begin{lemma}
There exist nonzero $q_0' \in \mathcal{O}$, $\epsilon \in (0, \delta_\Gamma)$, $C > 0$, and $\kappa > 0$ such that for all square-free $q \in \mathcal{O}$ coprime to $q_0q_0'$, $f \in L(\Lambda, \mathbb R)$, and $(r, u, \phi) \in \mathbb R \times \Lambda \times L_0^2(\tilde{\mathbf{G}}_q)$, we have
\begin{align*}
\left\|\int_{-\lambda}^\lambda k_\lambda(t)N_q(r + t, u, \phi) \, dt\right\|_2 \leq \kappa N(q)^Ce^{r(\delta_\Gamma - \epsilon)} \lambda^{-2} \|f \otimes \phi\|_{\Lip}.
\end{align*}
\end{lemma}

The following lemma is a higher dimensional effectivized version of Lalley's theorem \cite[Theorem 1]{Lal89}. For all $f \in L(\Lambda, \mathbb R)$, define $C_f = \frac{\nu_0(f)}{\delta_\Gamma \nu_\Lambda(\tau)} \cdot h_0 \in L(\Lambda, \mathbb R)$.

\begin{lemma}
\label{lem:Non-congruenceN_qEstimateWithMainTerm}
There exist $\epsilon \in (0, \delta_\Gamma)$ and $C > 0$ such that for all nonzero $q \in \mathcal{O}$, $f \in L(\Lambda, \mathbb R)$, and $(r, u, \phi) \in \mathbb R \times \Lambda \times \mathbb C \chi_{\tilde{\mathbf{G}}_q}$, we have as $r \to +\infty$
\begin{align*}
\int_{-\lambda}^\lambda k_\lambda(t)N_q(r + t, u, \phi) \, dt = C_f(u) e^{r\delta_\Gamma}\phi + O\bigl(N(q)^C e^{r(\delta_\Gamma - \epsilon)} \lambda^{-2} \|f\|_{\Lip}\|\phi\|_2\bigr).
\end{align*}
\end{lemma}

\Cref{lem:Non-congruenceN_qEstimateWithMainTerm} is derived again by the analysis as in the aforementioned papers using the non-congruence case of \cref{itm:CongruenceOperatorSpectralBoundsLarge|b|} in \cref{thm:CongruenceOperatorSpectralBounds} (i.e., for $q = 1$) together with the complex RPF theorem and perturbation theory.

\begin{remark}
The main term in \cref{lem:Non-congruenceN_qEstimateWithMainTerm} (cf. \cite{Lal89}) comes from the simple pole of $\xi \mapsto (\Id_{L(\Lambda)} - \mathcal{L}_\xi)^{-1}$ at $\xi = 0$. For this, \cite{BGS11} is more helpful but there is a slight error in carrying through $\frac{1}{z}$ in the first equation of page 282. Although the main term is not dealt with directly, see \cite{MOW19} for the corrected estimates.
\end{remark}

Combining the previous two lemmas gives the following lemma.

\begin{lemma}
\label{lem:N_qUniformCounting}
There exist nonzero $q_0' \in \mathcal{O}$, $\epsilon \in (0, \delta_\Gamma)$, and $C > 0$ such that for all square-free $q \in \mathcal{O}$ coprime to $q_0q_0'$, $f \in L(\Lambda, \mathbb R)$, and $(r, u, \phi) \in \mathbb R \times \Lambda \times L^2(\tilde{\mathbf{G}}_q, \mathbb R_{\geq 0})$, we have as $r \to +\infty$
\begin{align*}
N_q(r, u, \phi) = \frac{C_f(u)e^{r\delta_\Gamma} \bigl\langle \phi, \chi_{\tilde{\mathbf{G}}_q}\bigr\rangle \chi_{\tilde{\mathbf{G}}_q}}{\#\tilde{\mathbf{G}}_q} + O\bigl(N(q)^C e^{r(\delta_\Gamma - \epsilon)} \|f\|_{\Lip} \|\phi\|_2\bigr).
\end{align*}
\end{lemma}

\Cref{thm:MainTheorem} can now be derived using \cref{lem:N_qUniformCounting} and \cref{lem:ConvertN_qToN_q^*} exactly as in \cite[Subsection 3.4]{MOW19}. We describe it briefly. First, the sum in \cref{thm:MainTheorem} is converted to $N_q^*(r, \gamma_0, \delta_{\pi_q(x)})(e)$ using techniques from the beginning of this section. Next, $N_q^*(r, \gamma_0, \delta_{\pi_q(x)})$ is converted to a sum of terms involving $N_q$ using \cref{lem:ConvertN_qToN_q^*}. Next, \cref{lem:N_qUniformCounting} is applied to each term in the sum. This produces an expression which is nearly the main result except for two remaining sums. These sums have quick estimates as in \cite[Subsection 3.4]{MOW19} and finishes the derivation.

\nocite{*}
\bibliographystyle{alpha_name-year-title}
\bibliography{References}
\end{document}